\documentclass[11pt, reqno]{amsart} 

\usepackage[margin=2.5 cm]{geometry} 
\usepackage{setspace,graphicx} 
\usepackage{amsmath,amssymb,amsthm}   
                                                   
\usepackage{url, hyperref}     
\usepackage{floatflt}

\usepackage{xy}
\input xy 
\xyoption{all}

\usepackage{pstricks, epstopdf}

\usepackage{color,soul}

\newcommand{\bC}{{\mathbb C}}
\newcommand{\bD}{{\mathbb D}}
\newcommand{\bN}{{\mathbb N}}

\newcommand{\bQ}{{\mathbb Q}}
\newcommand{\bR}{{\mathbb R}}

\newcommand{\bZ}{{\mathbb Z}}

\newcommand{\bJ}{{\mathbb J}}
\newcommand{\bT}{{\mathbb T}}

\newcommand{\hx}{{\textbf x}}
\newcommand{\hy}{{\textbf y}}
\newcommand{\hz}{{\textbf z}}

\newcommand{\hal}{\mbox{\boldmath{$\alpha$}}}
\newcommand{\hbe}{\mbox{\boldmath{$\beta$}}}

\newcommand{\fs}{\mathfrak{s}}
\newcommand{\fj}{\mathfrak{j}}
\newcommand{\fd}{\mathfrak{d}}
\newcommand{\fS}{\mathfrak{S}}

\newcommand{\mA}{\mathcal{A}}
\newcommand{\mB}{\mathcal{B}}

\newcommand{\mD}{\mathcal{D}}
\newcommand{\mE}{\mathcal{E}}
\newcommand{\mF}{\mathcal{F}}

\newcommand{\mH}{\mathcal{H}}
\newcommand{\mJ}{\mathcal{J}}

\newcommand{\mL}{\mathcal{L}}

\newcommand{\mM}{\mathcal{M}}

\newcommand{\mP}{\mathcal{P}}

\newcommand{\wM}{\widehat {\mathcal{M}}}
\newcommand{\Muniv}{\mM^{\textrm{univ}}}

\newcommand{\De}{\Delta}
\newcommand{\Om}{\Omega}
\newcommand{\om}{\omega}
\newcommand{\Ga}{\Gamma}

\newcommand{\Si}{\Sigma}
\newcommand{\la}{\lambda}
\newcommand{\be}{\beta}
\newcommand{\al}{\alpha}
\newcommand{\de}{\delta}
\newcommand{\si}{\sigma}
\newcommand{\ep}{\epsilon}
\newcommand{\varep}{\varepsilon}
\newcommand{\ga}{\gamma}

\newcommand{\hfi}{HF^\infty}
\newcommand{\hfhat}{\widehat{HF}}
\newcommand{\hfkhat}{\widehat{HFK}}

\providecommand{\abs}[1]{\lvert#1\rvert} 
\providecommand{\norm}[1]{\lVert#1\rVert}

\newcommand{\bdd}{\partial}

\newtheorem{thm}{Theorem} [section]
\newtheorem{theorem/definition}{Theorem/Definition}[section]
\newtheorem{prop}[thm]{Proposition}
\newtheorem{claim}[thm]{Claim}
\newtheorem{cor}[thm]{Corollary}
\newtheorem{lemma}[thm]{Lemma}

\theoremstyle{definition}
\newtheorem{ex}[thm]{Example}

\newtheorem{deff}[thm]{Definition}

\newtheorem{rmk}[thm]{Remark}

\theoremstyle{remark}

\DeclareMathOperator{\rank} {rank}

\DeclareMathOperator{\ind} {ind}

\DeclareMathOperator{\Ker} {Ker}
\DeclareMathOperator{\Coker} {Coker}

\DeclareMathOperator{\proj} {proj} 

\DeclareMathOperator{\End} {End}

\DeclareMathOperator{\Int} {Int}

\DeclareMathOperator{\Spin} {Spin}
\DeclareMathOperator{\Sym} {Sym}

\DeclareMathOperator{\vol} {vol} 
\DeclareMathOperator{\gr} {gr}

\DeclareMathOperator{\Imagine} {Im} 
\DeclareMathOperator{\Lag} {Lag} 
 
\DeclareMathOperator{\dete} {det} 
\DeclareMathOperator{\Sp} {Sp} 

\begin{document} 

\title{Introduction to sutured Floer homology}         
\date{\today}       
\author{Irida Altman}
\maketitle

\begin{abstract} 

The following article constitutes the first chapter of my thesis.  It has been written as a standalone introduction to sutured Floer homology for graduate students in geometry and topology.  It contains most of the things the author wishes she had known when she started her journey into the world of sutured Floer homology.  

The article is divided into three parts.  The first part is an introductory level exposition of Lagrangian Floer homology.  The second part is a construction of Heegaard Floer homology as a special, and slightly modified, case of Lagrangian Floer homology.  The third part covers the background on sutured manifolds, the definition of sutured Floer homology, as well as a discussion of its most basic properties and implications (it detects the product, behaves nicely under surface decompositions, defines an asymmetric polytope, its Euler characteristic is computable using Fox calculus). Any errors in the article are the author's.

\end{abstract}

\tableofcontents

\section{Introduction}

\indent \indent A part of today's geometry and topology uses tools built on large analytic and algebraic foundations.   The goal of this article is to introduce the ideas behind sutured Floer homology in a way that exposes the depth and beauty of the theory, without getting lost in analytical detail.  The story of Floer theory begins about 35 years ago.

In the late 1980's, Andreas Floer developed an infinite-dimensional analogue of Morse theory in order to solve a special case of the Arnold conjecture \cite{F0,F1,F2,F3}.  In subsequence years, his work  proved to be the starting point of many theories in diverse areas of mathematics such as geometry, topology and dynamics. 

In the early 2000's, Ozsv\'ath and Szab\'o developed a set of invariants of closed 3-manifolds, using a slightly modified construction of Floer's theory \cite{OSmain}.   They associate to a given closed 3-manifold $Y$ a variety of  finitely generated Abelian groups that are referred to as  `flavours' of {\it Heegaard Floer homology}.  The simplest and most relevant for us is the `hat flavour' denoted by $\hfhat(Y)$.  
 
 In 2006, Juh\'asz developed a Floer theory  for a wide class of 3-manifolds with boundary, called {\it (balanced) sutured manifolds} \cite{Ju}. A sutured manifold is a pair $(M,\ga)$, where $M$ is an oriented 3-manifold with boundary, and $\ga$ is a collection of oriented, simple closed curves on $\bdd M$ that satisfy certain orientation requirements with respect to the orientation of $\bdd M$.  Then the {\it sutured Floer homology} of $(M,\ga)$, denoted by $SFH(M,\ga)$, is a finitely generated, bigraded Abelian group, whose construction is very similar to the construction of $\hfhat(Y)$ for a closed manifold $Y$.
 
In the last section, among other things, we answer simple questions such as why sutured Floer homology is defined only for balanced sutured manifolds, why it is constructed like $\hfhat$ (and not some other flavour of Heegaard Floer homology), and why it does not refer to a base point like the Ozsv\'ath-Szab\'o invariants.  The answers are obvious to those familiar with sutured Floer homology, but may be less obvious to someone starting out in Floer theory, and wanting to work with sutured manifolds.  We hope that this article provides a quick way into the field.  

The remainder of this article is divided into three parts: 
\begin{itemize}
\item Section \ref{section Lag FH}: \,\,Introductory level exposition of Lagrangian Floer homology.  
\item Section \ref{section HF}: \,Construction of Heegaard Floer homology as a special, and slightly modified, case of Lagrangian Floer homology.  
\item Section \ref{section SFH}: Background on sutured manifolds, the definition of sutured Floer homology, as well as a discussion of its most basic properties and implications (it detects the product, behaves nicely under surface decompositions, defines an asymmetric polytope, its Euler characteristic is computable using Fox calculus).
\end{itemize}
As may be expected of an introduction to sutured Floer homology, Section \ref{section SFH} attempts to balance foundational details with a sufficiently large scope, so that the reader may both see some `tricks of the trade' and witness a variety of results.  In places, we try to fill in small gaps that may have been omitted in the original, `adult' publications of (mostly) Juh\'asz's work.   The majority of the exposition in that section can be understood independently of the previous sections.  The hope is that those who do not care about the underlying Floer machinery can still familiarise themselves with sutured Floer homology as a powerful tool for studying sutured manifolds.

\subsection*{Notation}    If two topological spaces $W$ and $X$ are homeomorphic, we write $W \cong X$.  Denote by $\abs X$ the number of connected components of $X$.  If $U$ is an open set in $X$, then $\overline U$ denotes the closure of $U$ in the topology of $X$.  If $V$ is another space and $V \cap X \neq \emptyset$, then $X-V$ denotes all elements of $X$ not in $V$.  All homology groups are assumed to be given with $\bZ$ coefficients unless otherwise stated.  

\subsection*{Figures}  As this is an introductory article, and not original research, the figures have deliberately been done in a style that evokes a less formal, lecture-notes format, while retaining clarity.

\subsection*{Acknowledgements}
As this article is the first part of my thesis, I would like thank all three of my Ph.D. advisers:
\begin{itemize}
\item  {\it Stefan Friedl} for taking me as his student at the University of Warwick and for introducing me to the world of 3-manifolds.  He has given me invaluable advice, both globally in the form of project ideas and math lore, as well as locally in the form of typos, errors, and improvements of my manuscripts.
\item {\it Andr\'as Juh\'asz} for being my guide through the world of sutured Floer homology and through the projects I have done in this domain. His support and supervision of my work kept me headed in the right direction. I am most grateful for his teaching, patience and enthusiasm during my stay at the University of Cambridge.
 
\item{\it Saul Schleimer} for accepting the responsibility of watching over me, guiding me through the bureaucracy of a PhD, and making sure I stayed on track.

\end{itemize}
Many thanks to Will Merry for a careful proofreading of this article, and for pointing out a couple of minor errors.

I owe much gratitude to the University of Warwick and the Warwick Mathematics Institute for generously supporting me through a Warwick Postgraduate Research Scholarship, as well as to the Department for Pure Mathematics and Mathematical Statistics in Cambridge for their hospitality.

\section{Lagrangian Floer homology }\label{section Lag FH}

Floer homology takes its name after Andraes Floer who first used it to prove a special case of the Arnold conjecture; he showed that the sum of the Betti numbers of a symplectic manifold is a lower bound for the number of 1-periodic solutions of a 1-periodic Hamiltonian system \cite{Ar65}. The full conjecture is still open -- it states that the number of 1-periodic solutions of a 1-periodic Hamiltonian system is at least the minimal number of critical points of a function on the manifold, and if all the 1-periodic solutions are non-degenerate, then the number of solutions is at least the number of critical points of a Morse function on the manifold.  The Lefschetz fixed point theorem is related to the Arnold conjecture in the same way that the Poincar\'e-Hopf theorem is related to the Morse theoretic statement that the sum of the Betti numbers of a manifold is a lower bound on the number of critical points of a Morse function on that manifold.    The message to take away is that Floer's homology is a generalisation of Morse homology, and that the latter is a finite-dimensional model of the former.   Those familiar with Morse homology should find the exposition in this section accessible.

In fact, we care  about the Lagrangian-theoretic version of the Arnold conjecture.  We start with a given symplectic manifold $(M,\om)$ and Lagrangian submanifold $L \subset M$.  In this setting, the Arnold conjecture proposes that the number of 1-chords of a Hamiltonian system both starting and ending in $L$ is bounded bellow by the sum of the Betti numbers of the Lagrangian.  Again, more generally, it proposes that the number of such 1-chords is at least the minimal number of critical points of a function on $L$, and if $L$ is transverse to its image under the time-one map of the Hamiltonian system, then the number of 1-chorus is at least the number of critical points of a Morse function. Historically, Floer first worked on this latter version of the conjecture \cite{F0}, and his constructions form the main basis of this section.

Let $(M^{2n},\om)$ denote a closed connected symplectic manifold.  That means that $\om \in \Om^2(M)$ is a closed 2-form with the property that $\om^n$ is a volume form on $M$.  Recall that a submanifold $L$ of $M$ is {\it Lagrangian} if $\dim L=n$ and $\om |_{TL}=0$.  In this article we assume that all Lagrangians are closed. 

Under favourable conditions the {\it Lagrangian Floer homology} $HF(M,L_1,L_2)$ is a well defined abelian group associated to a pair of Lagrangian submanifolds. The meaning of `favourable conditions' varies: Floer originally assumed that $L_2= \varphi (L_1)$ for a Hamiltonian diffeomorphism  $\varphi$, that $\pi_2(M,L_1)=0$, and that $L_1$ and $L_2$ are transverse \cite{F0}.  In this article, we discuss one of the first extensions of Floer's theory due to Oh \cite{Oh}. In particular, Oh defined $HF(M, L_1,L_2)$ under the following four assumptions:
\begin{description}
\item [(F1)] $L_1$ and $L_2$ are transverse;
\item [(F2)] $L_1$ and $L_2$ are {\it (positively) monotone};
\item [(F3)] at least one of the subgroups $\iota_{1*}(\pi_1(L_1))$ and $\iota_{2*}(\pi_1(L_2))$ are torsion subgroups of $\pi_1(M)$, where $\iota_k \colon L_k \hookrightarrow M$ is the inclusion for $k=1,2$;
\item [(F4)] the {\it minimal Maslov numbers} $N_{L_1}$ and ${N_{L_2}}$ of $L_1$ and $L_2$ are both at least 3.
\end{description}
We refer to these as assumptions or conditions followed by these numbers {\bf(F1--4)}.

In the remainder of Section \ref{section Lag FH}, we explain these assumption and how they enable the construction of Lagrangian Floer homology. Here is a list of topics.
\begin{itemize}

\item[\ref{subsection pseudo}\, ] Basics of pseudo-holomorphic maps. 
\item[\ref{subsection the setup}\, ] Setup of Lagrangian Floer homology.
\item[\ref{subsection maslov}\, ] Introduction to various forms of Maslov indices.
\item[\ref{subsection floer chain}\, ] Definition of  Floer chain complex.
\item[\ref{subsection grading}\, ] Grading of  Floer complex.
\item[\ref{subsection sobolev}\, ]  $W^{k,p}$ norms and Sobolev spaces.
\item[\ref{subsection J-holo}\, ]  Some properties of $J$-holomorphic curves.
\item[\ref{subsection fredholm}\, ]   Fredholm operators and infinite-dimensional Implicit Function Theorem.
\item[\ref{subsection transversality}\, ]  Transversality of vertical derivative and zero section (of a Banach bundle).
\item[\ref{subsection spectral}] A bit about spectral flow.
\item[\ref{subsection Morse theory}] Morse theory written in the language of Floer theory.
\item[\ref{subsection floer case}] Spectral flow and Maslov index in the Floer case.
\item[\ref{subsection energy}]  Energy of pseudo-holomorphic maps.
\item[\ref{subsection gromov}]  Compactness of moduli spaces: Gromov compactness theorems.

\end{itemize}

\subsection{Pseudo-holomorphic maps}\label{subsection pseudo}
  First let us define the domain of pseudo-holomorphic maps.  Denote by $S$ the closed strip $\{\bZ \in \bC \mid 0 \leq \Imagine(z) \leq 1\}$, which we identify with $\bR \times [0,1]$.  Let $S^\circ$ denote the interior of $S$ identified with $\bR \times (0,1)$.  Next, denote by $\bD$ the closed unit disk in $\bC$, and by $\bD ^\circ$ the open unit disk in $\bC$.  The open strip $S^\circ$ is biholomorphic to the open unit disk $\bD^0$, for example, via the map $f \colon S^\circ \to \bD^\circ$ given by
\begin{equation} \label{eq identification}
f(z):=\frac{\exp(\pi z)- i}{\exp(\pi z) + i}.
\end{equation}

Let $(M^{2n}, \om)$ be a closed symplectic manifold, and $J$ be an almost-complex structure of $M$.  Recall that  $\End(TM) \to M$ is the vector bundle with fibre over $p \in M$ being the vector space of linear maps $T_p(M) \to T_p(M)$.   Then $J \in \Ga \End(TM))$, that is, $J$ is a section of  $\End(TM) \to M$, with $J^2=-1$.  An almost complex structure $J$ is said to be {\it compatible} with the symplectic from $\om$ if the bilinear form $g_J(\cdot,\cdot):=\om(\cdot,J\cdot)$ is both positive definitive and symmetric, and so it is a Riemannian metric on $M$.  An easy argument shows that the set of compatible almost complex structures $\mJ(M,\om)$ is connected \cite[Prop.\,4.1]{McDS1}.  This means that if we think of $(TM,J)$ as a complex vector bundle over $M$, then the Chern numbers $c_k(TM,J)$ are independent of $J \in \mJ(M,\om)$.  Specifically, we care about $c_1(TM,J)$, which we denote by $c_1(M)$ or even just $c_1$. 

A smooth map $u \colon S \to M$ is called {\it $J$-holomorphic} or {\it pseudo-holomorphic with respect to $J$} if the differential $du$ is a complex linear map with respect to $i$ and $J$, that is: 
\begin{equation} \label{J holo 1}
du + J \cdot du \cdot i=0,
\end{equation}
where $i$ is the standard complex structure on $\bC$, restricted to $S$.  Equivalently, $u$ is pseudo-holomorphic if the complex anitilinear part of $du$ is zero, that is, if 
\[
\bar \bdd_J (u):=\frac{1}{2}(du + J \cdot du \cdot i)\in \Om^{0,1}(u^*TM).
\]
More generally, the notion of pseudo-holomorphic is defined for any pair $(S,j)$, where $S$ is any Riemann surface together with a complex structure $j$, and any almost complex manifold $(M,J)$.  Then $u \in C^\infty(S,M)$ is called {\it $(J,j)$-holomorphic} when $du + J \cdot du \cdot j=0$.  In this language the map in \eqref{eq identification} is a pseudo-holomorphic map $f$ from $(S^\circ, i)$ to $(\bD^\circ,i)$.

For us it is sufficient to consider  the strip $S$ in $\bC$ with $j$ equal to the standard complex structure on $S$.  In particular, if $(s,t)$ are coordinates on $\bR \times [0,1]$, then $i \left( \frac{\partial}{\partial s} \right)=\frac {\partial}{\partial t}$ and $i \left( \frac{\partial}{\partial t} \right)=-\frac {\partial}{\partial s}$, so in coordinates we have
\[
\bar \bdd_J(u)\left( \frac{\bdd}{\bdd s} \right)=\frac{1}{2}\left( \frac{\bdd u}{\bdd s} + J \frac{\bdd u}{\bdd t} \right),
\hspace{1cm}
\bar \bdd_J(u)\left( \frac{\bdd}{\bdd t} \right)=\frac{1}{2}\left( \frac{\bdd u}{\bdd t} - J \frac{\bdd u}{\bdd s} \right).
\]
Thus, $\bar \bdd _J (u)=0$ if and only if 
\[
\left( \frac{\bdd u}{\bdd s} + J \frac{\bdd u}{\bdd t} \right)=0, 
\hspace{1cm}
\left( \frac{\bdd u}{\bdd t} - J \frac{\bdd u}{\bdd s} \right)=0.
\]
But the second equation gives no new information, as $J\left( \frac{\bdd u}{\bdd t} - J \frac{\bdd u}{\bdd s} \right)=\left( \frac{\bdd u}{\bdd s} + J \frac{\bdd u}{\bdd t} \right)$. Therefore, Equation \eqref{J holo 1} can equivalently be written as 
\begin{equation} \label{J holo 2}
\frac{\partial u}{\partial s} + J \frac{\partial u }{\partial t}=0.
\end{equation}

\subsection{The setup} \label{subsection the setup}

As before let $(M^{2n}, \om)$ be a closed symplectic manifold, and let $L_1$ and  $L_2$ be two Lagrangian submanifolds of $M$.  Note that the assumption {\bf (F1)} on the transversality of $L_1$ and $L_2$ is not very restrictive.  Indeed, since $L_1$ and $L_2$ are closed submanifolds of $M$, a standard Sard-Smale argument shows that we can always perturb $L_2$ to get a new Lagrangian manifold $L_2':=\varphi(L_2)$, such that $L_2$ is transverse to $L_1$, and $\varphi$ is a Hamiltonian diffeomorphism; this argument is similar to, but simpler than, the transversality arguments discussed in Section \ref{subsection transversality}.  A key property of the Floer homology groups $HF(L_1,L_2)$ proved originally by Floer \cite[Sec.\,3]{F0}, is that they are invariant under perturbations of the form $L_2 \leadsto \varphi(L_2)$. Therefore, it is possible to define $HF(L_1,L_2)$ even when {\bf (F1)} fails, by setting
\[
HF(L_1,L_2):=HF(L_1,\varphi(L_2)),
\]
for any appropriate $\varphi$.  (In particular, we could define $HF(L,L)$.)

In any case, we assume {\bf (F1)}.  By a dimension counting argument  $L_1 \cap L_2$ is zero-dimensional, and since $L_1$ and $L_2$ are compact, $L_1 \cap L_2$ is necessarily a finite set of points.

Lagrangian Floer homology studies the set $\mM_J(M, L_1,L_2)$ of smooth $J$-holomorphic maps $u$ that satisfy
\[
u\left(\bR\times \{0\}\right) \subset L_1, \hspace{1cm} u\left(\bR\times \{1\}\right)\subset L_2, 
\hspace{1cm} u(\pm \infty, \cdot) \subset L_1 \cap L_2.
\]
The set $\mM_J(M, L_1,L_2)$ is referred to as the {\it moduli space of pseudo-holomorphic maps}.

Let $\hx,\hy \in L_1 \cap L_2$.  Denote by $\mM_J(\hx, \hy) \subset \mM_J(M,L_1,L_2)$ the subset of maps $u$ such that the asymptotes are fixed, that is, such that $u(-\infty, \cdot) \cong \hx$ and $u (\infty, \cdot) \cong \hy$.  When $J$ is fixed or understood, we write only $\mM(\hx,\hy)$.  Note that, alternatively, a map $u \in \mM(\hx,\hy)$ can be thought of as a map $v \colon \bD \to M$ that is $J$-holomorphic on $\bD^\circ$ and that satisfies
\begin{align*}
& v(\bdd^- \bD) \subset L_1, \hspace{1cm}  v(\bdd^+ \bD) \subset L_2, \\
& v(-i)=\hx,  \hspace{1.5cm} v(i)=\hy,
\end{align*}
where 
\begin{gather*}
\bdd^- \bD:=\bdd \bD \cap \{z \in \bC \mid \Imagine(z)\leq 0\}, \\
\bdd^+ \bD:=\bdd \bD \cap \{z \in \bC \mid \Imagine(z)\geq 0\}.
\end{gather*}
It is sometimes useful to think of elements of $\mM(\hx,\hy)$ as maps $u \colon S \to M$ and sometimes as maps $v \colon \bD \to M$ and we do so interchangeably; see Figure \ref{fig Whitney disk} for an illustration $v$.  Nevertheless, in order to help the reader, we {\it consistently} use the letter $u$ for maps defined on $S$, and $v$ for maps defined on $\bD$.
\begin{figure}[h]
\centering
\includegraphics [scale=0.50]{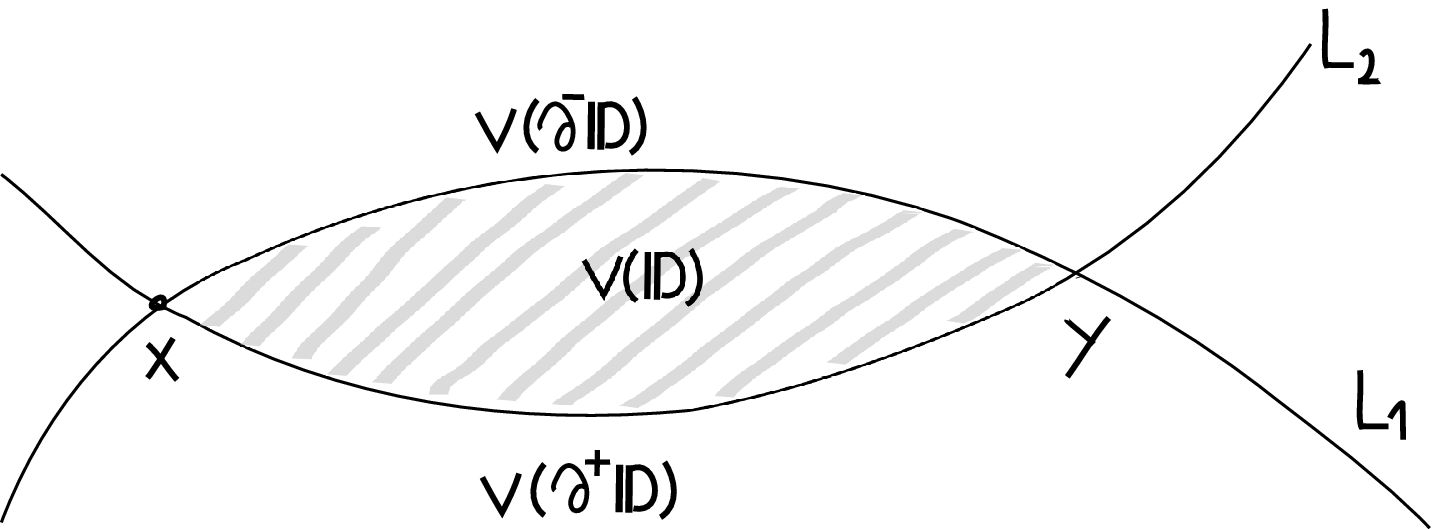} 
\caption{A $J$-holomorphic disk $v \in \mM(\hx,\hy)$.}
\label{fig Whitney disk}
\end{figure}

\begin{claim} \label{claim manifold} For a generic almost complex structure $J$ the set $\mM_J(\hx,\hy)$ admits the structure of a finite-dimensional smooth manifold (possibly with components of differing dimensions).  
\end{claim}

See Theorem \ref{thm manifold} for a more precise statement.

It is possible to view the intersection points $\hx,\hy \in L_1 \cap L_2$ as critical points of a function $\mA$ defined on the set of paths from $L_1$ to $L_2$, and to view elements of $\mM(\hx,\hy)$ as gradient flow lines of $\mA$ (see for example \cite{F0} or \cite[Sec.\,13.5]{Oh book}). Under Assumption {\bf(F1)}, the function $\mA$ is a Morse function.  One way to study the spaces $\mM(\hx,\hy)$ would thus be an attempt to do Morse theory on $\mA$.  However, this approach fails dramatically: in general the Morse index $i(\hx)$ of an intersection point $\hx$ (that is, the dimension of the negative eigenspace of the Hessian of $\mA$ at $\hx$) is infinite.  Analogously to Morse theory, we would expect the dimension $\mM(\hx,\hy)$ to be the difference of the Morse indices $i(\hx)-i(\hy)$, and this is definitely {\it not} the case here as the difference may not even be defined.

Floer's key breakthrough was to interpret elements of $\mM(\hx,\hy)$ not as a gradient flow lines of $\mA$, but as pseudo-holomorphic maps (as we do here).  Put concisely, Floer realised he could study elements of $\mM(\hx,\hy)$ as solutions to an elliptic PDE, rather than a hyperbolic ODE.  Indeed, `morally' Claim \ref{claim manifold} holds because elements of $\mM(\hx,\hy)$ are solutions of an {\it elliptic} PDE, and so $\mM(\hx,\hy)$ is the zero set of a {\it Fredholm operator} (see Section \ref{subsection fredholm}).

What Floer needed next, was a suitable index theory (to mimic the role of the Morse index in Morse theory).  It turns out that the correct object to study is the so called {\it Maslov index}.  Given $u \in \mM(\hx,\hy)$, the {\it local dimensional} $\dim_u\mM(\hx,\hy)$ is the dimension of the component of $\mM(\hx,\hy)$ containing $u$, and is given by the Maslov index $\mu([u])$ that depends only on the homotopy class of $u$ in $\mM(\hx,\hy)$.  This dependence on the homotopy class immediately highlights an important difference between Floer's theory and Morse theory --  in Morse theory the dimension of the moduli spaces depends only on the critical points themselves (and not on the gradient flow line).  In this sense, the Lagrangian Floer homology we study here is closer to {\it Novikov Morse theory}.

Actually, in the original paper \cite{F0}, Floer worked with a slightly different index, called the {\it Maslov-Viterbo index}, which was discovered by Viterbo \cite{Vi}.  However, we use the more `modern' approach based on the {\it Maslov index of a bundle pair}, which we study in the next section.

\subsection{The Maslov Index} \label{subsection maslov} In this section we define three different Maslov indices that eventually lead to the definition of $\mu([u])$ for $u \in \mM(\hx,\hy)$ as mentioned above.

Start with the standard identification of $\bR^{2^n}$ with $\bC^n$ given by $(x,y) \mapsto z=x+iy$.  Then the {\it standard symplectic form} $\om_{std} \in \Om^2(\bR^{2n})$ on $\bR^{2n}$ is defined to be 
\[
\om_{std}((u_1,v_1),(u_2,v_2)):=u_1 \cdot v_2 - u_2 \cdot v_1,
\] 
for $(u_j,v_j) \in T_{(x,y)} \bR^{2n} \cong \bR^{2n}$.  Here $\cdot$ denotes the Euclidean dot product.

We now recall some elementary facts about linear symplectic geometry; we refer the reader to \cite[Chap.\,2]{McDS1} for details and proofs of the following claims.  Denote by $\mL(n)$ the set $\Lag(\bC^n, \om_{std})$ of Lagrangian subspaces of $\bC^n$ with the standard symplectic structure.  Let $U(n)$ be the space of unitary $n\times n$ matrices, and $O(n)$ the space of orthogonal $n\times n$ matrices.  Then $\mL(n)$ can be identified with the homogenous space $U(n)/O(n)$, where $O(n)$ is considered as a subgroup of $U(n)$ under the natural embedding.  Indeed, given any $L \in \mL(n)$, we can write $L=A(\bR^n)$ for some unitary matrix $A \in U(n)$. Here $\bR^n$ is identified with $\bR^n \times \{0\}$ as a Lagrangian subset of $\bC^n \cong   \bR^{2n}$,  and $A ( \bR^n) = \bR^n$ if and only if $A \in O(n) \subset U(n)$.  It follows that 
\[
\pi_1(\mL(n))=\pi_1(U(n)/O(n))=\bZ
\]
The space $\mL(n)$ has a characteristic class $\mu_\mL \in H^1(\mL(n))$, called the {\it Maslov class}.  The Maslov class gives an explicit homomorphism $\pi_1(\mL(n)) \to \bZ$ that takes a loop of Lagrangian subspaces  $\la \colon S^1 \to \mL(n)$ and maps it into
\[
\mu_\mL(\la):= \deg (\dete^2(\la)) \in \bZ.
\]
Similarly, the symplectic linear group $\Sp(2n, \bR)$ has a characteristic class \linebreak $\mu_S \in H^1(\Sp(2n,\bR);\bZ)$.  Recall that $\Sp(2n,\bR)$ is the group of $2n \times 2n$ matrices $B$ satisfying $J_{std} = B^T J_{std} B$ for $J_{std}:= \left(\begin{smallmatrix} 0& I_n\\ -I_n&0 \end{smallmatrix}\right)$.  It can be shown that $\Sp(2n,\bR)$ deformation retracts onto $U(n)$ (regarded as a subgroup of $GL(2n,\bR)$), and so 
\[
\pi_1(\Sp(2n,\bR))=\pi_1(U(n))=\bZ.
\]
Further, this means that the determinant map $\dete \colon U(n) \to S^1$ can be extended to a map $\widetilde \dete \colon \Sp(2n,\bR) \to S^1$ that induces an isomorphism on $\pi_1$.  Thus, we can define a characteristic class $\mu_S\in H^1(\Sp(2n,\bR))$, by defining its action on a loop $\tau \colon S^1 \to \Sp(2n,\bR)$:
\[
\mu_S(\tau):= \deg(\widetilde \dete(\tau)) \in \bZ.
\]

Given two loops, $\la \colon S^1 \to \mL(n)$ and $\tau \colon S^1 \to \Sp(2n,\bR)$, we can define another loop $\tau \cdot \la \colon S^1 \to \mL(n)$ by setting $(\tau \cdot \la)(t):=\tau(t) \left( \la(t) \right)$.  Note that $\tau \cdot \la$ is well-defined, as a symplectic matrix maps a Lagrangian subspace of $\bR^{2n}$ into another Lagrangian subspace, hence there is a well-defined action of $\Sp(2n,\bR)$ on $\mL(n)$.  The two classes $\mu_\mL$ and $\mu_S$ are related in the following way
\begin{equation} \label{equation on mus}
\mu_\mL(\tau \cdot \la) = \mu_\mL(\la) + 2 \mu_S(\tau),
\end{equation}
(see \cite[Thm.\,2.35]{McDS1}).

Next, we define the {\it Maslov index of a symplectic bundle pair} $(E,F)$. Suppose $X$ is a Riemann surface with non-empty boundary.  Let $E \to X$ be a symplectic vector bundle, that is, there exists $\om \in \Ga(X, \Om^2(E))$ such that $(E_z,\om_z)$ is a symplectic vector space for each $z \in X$.  Any such vector bundle is necessarily trivial since $\bdd X \neq \emptyset$ \cite[Prop.\,2.66]{McDS1}.  Also, let $F$ be a subbundle of $E|_{\bdd X}$ with Lagrangian fibres, which means that for each $z \in \bdd X$, the subspace $F_z$ is Lagrangian in the symplectic vector space $(E_z,\om_z)$.

Let us label the connected components of $\bdd X$ by $C_1, \ldots, C_k$, and parametrize each component via a homeomorphism with the circle, $C_i \cong S^1$.  Since $E$ is a symplectic vector bundle, we can pick a symplectic trivialisation $\Phi \colon E \to X \times \bR^{2n}$.  Then $\Phi$ defines $k$ loops $\la_i \colon S^1 \to \mL(n)$ via
\[
\la_i := \Phi(F|_{C_1}).
\]
Finally, define
\[
\mu(E,F):=\sum_{i=1}^k \mu_\mL(\la_i) \in \bZ.
\]

It remains to show that the integer $\mu(E,F)$ is independent of the choice of trivialisation $\Phi$.  If $\Psi$ is another trivialisation, then the change of trivialisation function at a point $z \in X$:
\begin{equation} \label{psi and phi}
\tau_z:= \Phi \circ \Psi^{-1}|_{E_z}: \bR^{2n} \to \bR^{2n}
\end{equation}
is actually an element of $\Sp(2n,\bR)$.  Thus if $\la'_i:=\Psi(F|_{C_i})$, then $\la_i'=\tau|_C \cdot \la_i$, and so from Equation \eqref{equation on mus} it follows that
\[
\sum_{i=1}^k \mu_\mL(\la_i')=\sum_{i=1}^k \mu_\mL(\la_i) + 2 \sum_{i=1}^k \mu_S(\tau|_{C_i}).
\]
Next, by definition $\sum_{i=1}^k \mu_S(\tau|_{C_i})=\deg(\widetilde{\deg}(\tau|_{\bdd X}))$ (since degree is additive).   The key point is that since $\tau|_{\bdd X}$ is defined on all of $X$, it must have degree zero (this is proved in \cite[Lem.\,2.71]{McDS1}), so $\sum_{i=1}^k \mu_S(\tau|_{C_i})=0$.  Therefore, from \eqref{psi and phi} we see that  $\mu(E,F)$ is the same value regardless of whether we use $\Psi$ or $\Phi$.

It can also be shown that $\mu(E,F)$ is  invariant under homotopy in the following sense.  Suppose that for $i=1,2$,  $E_i \to X_i$ are symplectic vector bundles with Lagrangian subbundles $F_i \subset E_i|_{\bdd X_i}$, and that $Y$ is a cobordism from $X_1$ to $X_2$.  Next, suppose that there exists a symplectic vector bundle $E$ over $Y$ that extends each $E_i$, that is $E|_{X_i}=E_i$, and suppose that there exists a symplectic vector bundle $F$ over all of $Y$ that extends each $F_i$, that is $F|_{X_i}=F_i$.  (Although, $F$ is not required to be a Lagrangian subbundle of $E$ except over $\bdd Y$.)  Then \cite[Prop.\,13.11]{Oh book} shows that 
\[
\mu(E_1,F_1) = \mu(E_2,F_2).
\]

We proceed to use $\mu(E,F)$ to define the three {\it  Maslov indices}, two of which give rise to the {\it minimal Maslov numbers} in condition {\bf (F4)}.  We do not impose this condition now, as it is only necessary in the proof of the very last theorem of this section: Theorem \ref{theorem bdd^2=0}.
\\
\\
\noindent {\bf Definition: Maslov index $\mu_L$.}  Set $X:=\bD$.   Fix a Lagrangian submanifold $L \subset M$.  Suppose that $v \colon (\bD,\bdd \bD) \to (M,L)$ is a smooth map.  Then define
\[
\mu_L(v):= \mu(v^*TM, v^*TL).
\]
The homotopy invariance of $\mu(E,F)$ implies that $\mu_L(v)$ depends only on $v$ up to homotopy.  Thus we can define the {\it minimal Maslov number $N_L$ of $L$} to be the smallest positive generator of the group
\[
\langle \mu_L(v) \mid v \in \pi_2(M,L)\rangle.
\]

Before giving the second definition, let us set up some notation.  

\begin{deff} \label{deff of pi2} As usual let $L_1,L_2 \subset M$ be two transverse Lagrangians.  Fix $\hx,\hy \in L_1 \cap L_2$.  Denote by $\pi_2(\hx,\hy)$ the set of homotopy classes of continuous maps $f \colon [0,1] \times [0,1] \to M$ such that for all $(s,t) \in [0,1] \times [0,1]$ we have 
\begin{align*}
& f(0,t)=\hx, \hspace{1cm}  f(1,t)=\hy, \\
& f(s,0)\in L_1,  \hspace{0.8cm} f(s,1)\in L_2.
\end{align*}
Then for any three points $\hx,\hy,\hz \in L_1 \cap L_2$ there is a  natural `gluing' map 
\[
\sharp \colon \pi_2(\hx,\hy) \times \pi_2(\hy,\hz) \to \pi_2(\hx,\hz).
\]
To give an explicit definition of $\sharp$, let $f,g \colon [0,1] \times [0,1] \to M$ be continuous maps such that
\[
[f]=\phi \in \pi_2(\hx,\hy) \hspace{1cm} \text{and} \hspace{1cm} [g]=\psi \in \pi_2(\hy,\hz).
\]
\end{deff}
Then we can `glue' together $f$ and $g$ to get a map $f \# g \colon [0,1] \times [0,1] \to M$:
\[
(f \# g)(s,t):= 
\begin{cases}
f(2s,t), &(s,t) \in [0,1/2] \times [0,1], \\
g(2s-1,t), &(s,t) \in [1/2,1] \times [0,1].
\end{cases}
\]
We then define $\phi \sharp \psi:=f \# g$, which is actually independent of the choices of representatives $f$ and $g$, and so is well-defined.

Remark that $\sharp$ defines a free and transitive action of $\pi_2(\hx,\hx)$ on $\pi_2(\hx,\hy)$.
\\
\\
\noindent {\bf Definition: Maslov index $\mu$}.  We now define an index function $\mu \colon \pi_2(\hx,\hy)\to \bZ$ that we simply call the {\it Maslov index}.  One reference for this material is \cite[Sec.\,13.6]{Oh book}.

Take $X=[0,1]\times [0,1]$.  For each intersection point $\hx \in L_1 \cap L_2$, let $V_\hx \subset  \Lag(T_\hx M,\om_x)$ denote the set of all Lagrangian subspaces of $T_\hx M$ with the property that if $\la \in V_\hx$, then $\dim T_\hx L_1 \cap \la \geq 1$.  In other words,  $V_\hx$ is the set of Lagrangian subspaces of $T_\hx M$ that are {\it not} transverse to $T_\hx L_1$.  Now let $\la_\hx$ denote a path of Lagrangian subspaces such that 
\begin{gather*}
\la_\hx(0)=T_\hx L_1, \hspace{1cm} \la_\hx(1)=T_\hx L_2, \hspace{1cm}  \la_\hx(t) \pitchfork T_\hx L_1, \hspace{0.2cm} \forall t \in (0,1], \\
\frac{\bdd}{\bdd t}\la_\hx(0) \oplus T_{T_\hx L_1} V_\hx = T_{T_\hx L_1} \Lag(T_\hx M, \om_\hx).
\end{gather*}
(This last condition says that $\la$ is transverse to the {\it Maslov cycle $V_\hx$}.)
Suppose $f\colon X \to M$ represents an element $\phi \in \pi_2(\hx,\hy)$.  Set $E:=f^*TM$, and define $F$ to be given by $f^*TL_1$ and $f^*TL_2$ on the two horizontal sides ($t=0$ and $t=1$), and by $\la_\hx$ and $\la_\hy$ on the two vertical sides ($s=0$ and $s=1$).  The homotopy invariance property of the Maslov index of a bundle pair implies that $\mu(E,F)$ depends only on the homotopy class $\phi$, so $\mu(\phi):=\mu(E,F)$ gives a well-defined map
\[
\mu \colon \pi_2(\hx,\hy) \to \bZ,
\]
and since the Maslov index is additive under the concatenation of bundle pairs, we have that 

\begin{equation} \label{additivity of Maslov}
\mu(\phi \sharp \psi)=\mu(\phi) + \mu(\psi).
\end{equation}

Lastly, we give the third definition.
\\
\\
\noindent {\bf Definition: Maslov index $\mu_{L_1,L_2}$}.
Take $X=[0,1] \times S^1$.  Suppose $c \colon X \to M$ satisfies
\[
c(\{0\} \times S^1) \subseteq L_1 \hspace{0.5cm} \text{and} \hspace{0.5cm} c(\{1\} \times S^1) \subseteq L_2.  
\]
Then we can define
\[
\mu_{L_1,L_2}(c):=\mu(c^*TM,F),
\]
where $F_{(0,\cdot)}=c^*TL_1$ and $F_{(1,\cdot)}=c^*TL_2$.  We define $N_{L_1,L_2}$ to be the smallest positive generator of the group $\langle \mu_{L_1,L_2}(c) \rangle$ as $c$ ranges over the homotopy classes of such maps.

\begin{lemma} \label{two homotopy classes modulo}
For any $\phi,\psi \in \pi_2(\hx,\hy)$ the following holds
\[
\mu(\phi) \equiv \mu(\psi) \mod N_{L_1,L_2}.
\]
\end{lemma}

\begin{proof}
We can write, $\phi=\theta \sharp \psi$, for some $\theta \in \pi_2(\hx,\hx)$.  A representative $f \colon [0,1] \times [0,1] \to M$ for $\theta$ can also be regarded as a map $f \colon [0,1] \times S^1\to M$ satisfying
\[
f(\{0\} \times S^1) \subseteq L_1 \hspace{0.5cm} \text{and} \hspace{0.5cm} f(\{1\} \times S^1) \subset L_2.
\]
As $\theta \in \pi_2(\hx,\hx)$, the contribution to $\mu(\theta)$ arising from the paths $\la_\hx$ cancel out; in particular, It follows from the definitions that 
\[
\mu(\theta)=\mu_{L_1,L_2}(f).
\]
So $\mu(\theta) \equiv 0 \mod N_{L_1,L_2}$, and hence by the additivity of the Maslov index (Equation \eqref{additivity of Maslov}) if follows that $\mu(\phi) \equiv \mu(\psi) \mod N_{L_1,L_2}$.
\end{proof}

\begin{rmk} \label{rmk minimal Chern}
Fix $\al \in \pi_2(M)$.  We can always choose a representative $f \colon S^2 \to M$ of $\al$ with the additional property that 
\[
f(\text{north pole}) \in L_1 \hspace{0.5cm} \text{and}  \hspace{0.5cm} f(\text{south pole}) \in L_2.
\]
Then $\mu_{L_1,L_2}(f)$ is well-defined (just consider it as a map from $[0,1] \times S^1$ taking $\{0\} \times S^1$ to the north pole, and similarly $\{1\} \times S^1$ to the south pole). It can be shown (see \cite{Vi}) that in this case we have
\begin{equation} \label{eq division by two}
\mu_{L_1,L_2}(f)=2 c_1(f_*([S^2]),
\end{equation}
where $c_1 \in H^2(M)$ is the first Chern class of $(TM,J)$, the class $[S^2] \in H_2(S^2)$ is the fundamental class, and  $f_* \colon H_2(S^2) \to H_2(M)$ is the induced map on homology. 
\end{rmk}

Let us now give a precise definition of {\bf (F2)}.

\begin{deff} \label{deff monotone}
A Lagrangian manifold $L \subset M$ is called {\it (positively) monotone} if there exists a positive constant $c>0$ such that for any smooth map $u \colon (D^2, \bdd D^2) \to (M,L)$ we have
\begin{equation}
\mu_L(u)=c \int_{D^2} u^* \om.
\end{equation}
\end{deff}

There is an analogous notion of {\it monotonicity} for a symplectic manifold $(M,\om)$. Namely, $(M,\om)$ is a {\it (positively) monotone} symplectic manifold if there exists a constant $C>0$ such that for any smooth map $u \colon S^2 \to M$ one has
$$C\int_{S^2}u^* \omega = 2c_1(u_* ([S^2])).$$

In fact, monotone Lagrangians can only exist in monotone symplectic manifolds. Indeed, as in Remark \ref{rmk minimal Chern}, the point is that if $u \colon S^2 \to M$ satisfies $u(\text{north pole}) \in L$ then $\mu_L (u) = 2c_1 (u_* ([S^2]))$, and conversely (as $M$ is always assumed connected), any map $u\colon S^2 \to M$ may be assumed to satisfy $u(\text{north pole}) \in L$. Thus if $L$ is monotone with constant $c>0$, then $M$ itself is monotone with constant $C = 2c$. This also shows that if $L_1$ and $L_2$ are both monotone, then they are monotone with the {\it same} monotonicity constant.

Moreover, if we define $N_M$, the {\it minimal Chern number } of $(M,\omega)$ to be the smallest positive generator of the group
\[
\{ \langle c_1 , h \rangle \mid h \in H_2^{S}(M,\bZ)\},
\]
where $H_2^{S}(M,\bZ)$ denotes the subgroup of homology classes that can be represented by maps from $S^2$ into $M$, then necessarily
\[
N_L \mid 2N_M,
\]
and similarly 
\[ 
N_{L_1, L_2} \mid 2 N_M.
\]

As we see below, it is a general fact that (when defined) we can always grade the Floer homology groups $HF(L_1 , L_2)$ modulo $N_{L_1 , L_2}$. This is sometimes not very satisfactory, as it can be hard to compute $N_{L_1 , L_2}$. In the next remark we discuss how, under additional assumptions on the Lagrangians, this can be relaxed.

\begin{rmk}
We construct the Floer homology groups assuming four conditions {\bf(F1--4)} hold. As discussed above, condition {\bf(F2)} implies that $(M,\omega)$ is itself monotone. Suppose we replace condition {\bf(F2)} with the (generally weaker) statement that $(M,\omega)$ is monotone, but at the same time we strengthen condition  {\bf(F3)} to require that {\it both} $L_1$ and $L_2$ are simply connected.
In this case both $L_1$ and $L_2$ are monotone, and the following holds
\begin{equation} \label{eq equals}
N_{L_1} = N_{L_2} = N_{L_1 ,L_2} = 2 N_M.
\end{equation}
This implies that we could also replace condition {\bf(F4)} with the requirement that $N_M \ge 2$. Asking that both $L_1$ and $L_2$ are simply connected is a significantly stronger assumption, so in general we prefer to work with conditions {\bf(F1--4)} as originally stated.  However, this stronger assumption does allow for a more satisfactory solution to the grading problem. Indeed, in this case the Floer homology can be graded modulo $2N_M$. Not only is this independent of the choice of Lagrangians, but it is also often much easier to compute $N_M$ than $N_{L_1 , L_2}$.
Let us now prove \eqref{eq equals}.

\begin{proof}[Proof of \eqref{eq equals}]
We prove that $N_L = 2N_M$ if $(M,\omega)$ is monotone and $L$ is a simply connected Lagrangian. The other statements are similar. 

Suppose $u \colon (D^2,\bdd D^2) \to (M,L)$ satisfies $\mu_L (u) = N_L$. Since $L$ is simply connected, $u|_{\bdd D^2}$ bounds a disc $v \colon D^2 \to L$. Let $w \colon S^2 \to M$ denote the map obtained by gluing $v$ onto  $u$. Since $v$ is contained in $L$, we have
\[
N_L = \mu_L(u) = \mu_L (w) = 2 c_1 (w_*[(S^2)]) = 2j N_M
\]
for some integer $j \in \bN$. Since we already know that $N_L \mid 2 N_M$, we must have $N_L = 2N_M$. This completes the proof. 
\end{proof} 
\end{rmk}

Let us return to the moduli space $\mM(\hx,\hy)$.  Having now defined the Maslov index, we can give a slightly more precise statement of Claim \ref{claim manifold}.  Since $S$ is homotopy equivalent to $[0,1] \times [0,1]$, every element $u \in \mM(\hx,\hy)$ belongs to a unique class $\phi \in \pi_2(\hx,\hy)$; by a slight abuse of notation, we write $u \in \phi$. 
This induces a decomposition 
\[
\mM(\hx,\hy) = \bigsqcup_{\phi \in \pi_2(\hx,\hy)} \mM(\hx,\hy; \phi),
\]
where 
\[
\mM(\hx,\hy; \phi):=\{ u \in \mM(\hx,\hy) \mid u \in \phi\}.
\]
Here is the more precise statement of Claim \ref{claim manifold}; note that for it to be true we only require condition {\bf (F1)}.
\begin{thm} \label{thm manifold}
Assume that $L_1$ and $L_2$ satisfy {\bf (F1)}.  Then $\mM(\hx,\hy; \phi)$ has the structure of a finite-dimensional smooth manifold with 
\[
\dim \mM(\hx,\hy; \phi)=\mu(\phi).
\]
\end{thm}

Theorem \ref{thm manifold} motivates the definition of the following grading on the intersection points of the two Lagrangians.
 
\begin{deff} \label{def grading function}
Define a {\it grading function}
\[
\gr \colon L_1 \cap L_2 \times L_1 \cap L_2 \to \bZ/ N_{L_1,L_2} \bZ
\]
by
\[
\gr(\hx,\hy):=\mu(\phi) \mod N_{L_1,L_2}, \hspace{0.5cm} \text{for any } \phi \in \pi_2(\hx,\hy).
\]
\end{deff}

\begin{rmk}
Firstly, note that the additivity of the Maslov index from Equation \eqref{additivity of Maslov} implies that for any $\hx,\hy,\hz \in L_1 \cap L_2$ we have
\[
\gr(\hx,\hz)=\gr(\hx,\hy) + \gr(\hy,\hz).
\]
Secondly, note that Theorem \ref{thm manifold} says that there is a well-defined dimension of $\mM(\hx,\hy)$ modulo a Maslov number:
\[
\dim \mM(\hx,\hy) = \gr(\hx,\hy) \mod N_{L_1,L_2}.
\]
\end{rmk}

Let $\hx,\hy$ be two points in the intersection $L_1 \cap L_2$.   If $u \in \mM(\hx,\hy)$, then also $u' \in \mM(\hx,\hy)$ for every map $u'(s,t):=u(s+l,t)$ where $l \in \bR$.  Thus, if $\hx \neq \hy$, then $\mM(\hx,\hy)$ admits a free $\bR$ action.  Denote by $\widehat \mM (\hx, \hy)$ the quotient manifold $\mM(\hx,\hy)/\bR$.  Now 
\[
\dim \widehat \mM(\hx,\hy) = \gr(\hx,\hy)-1 \mod N_{L_1,L_2}.
\]

\subsection{The Floer chain complex} \label{subsection floer chain} Let us postpone discussing the proof of Theorem \ref{thm manifold} for a while and go on to explain how to define the Floer complex.

\begin{deff}
The Floer chain complex $CF(L_1,L_2)$ is defined to be 
\[
CF(L_1,L_2) := \bigoplus_{\hx \in L_1 \cap L_2} \bZ_2 \langle \hx \rangle.
\]
\end{deff}
So $CF(L_1,L_2)$ is the finite-dimensional $\bZ_2$-vector space whose basis are the intersection points $L_1 \cap L_2$.

\begin{rmk} \label{rmk coherent orientation}
Instead of working with $\bZ_2$ coefficients, we could also work with $\bZ$ coefficients.  However, that would require taking the {\it algebraic count} of $\wM(\hx,\hy;\phi)$ instead of simply the parity, which is more complicated and we do not discuss it here.  For such an algebraic count to be well-defined we need to worry about {\it coherent orientation} of the moduli spaces.  See \cite{FH} for the definition of coherent orientation and further explanations.  
\end{rmk}

\begin{claim}\label{claim compactness} If  for $\hx,\hy \in L_1 \cap L_2$ and for $\phi \in \pi_2(\hx,\hy)$ we have that $\mu(\phi)=1$, then $\wM(\hx,\hy; \phi)$ consists of finitely many points.
\end{claim}

Denote by $\#_2 \wM (\hx,\hy; \phi)$ the parity of $\wM(\hx,\hy; \phi)$.
 
 The  differential $\partial \colon CF(\hx,\hy) \to CF(\hx,\hy)$ is given by 
\begin{equation} \label{eq deff of differential}
\partial \hx := \sum_{\hy \in L_1 \cap L_2} \left( \sum_{\{\phi \in \pi_2(\hx,\hy) \mid \mu(\phi)=1\}} \#_2 \wM (\hx,\hy;\phi) \right)\cdot \hy.
\end{equation}
It is easy to see that $\bdd^2=0$ if we believe the following claim.

\begin{claim} \label{claim boundary} If $\wM(\hx,\hy;\phi)$ is 1-dimensional, then it can be compactified into a manifold $\overline{\wM(\hx,\hy;\phi)}$ in such a way that 
\begin{equation} \label {amazing fact 3}
\bdd \overline{\wM(\hx,\hy; \phi)}= \bigcup_{\hz \in L_1 \cap L_2 } \bigcup_{(\phi^-, \phi^+)} \wM(\hx,\hz;\phi^-) \times \wM(\hz,\hy; \phi^+),
\end{equation}
where the second union is over all pairs $(\phi^-,\phi^+) \in \pi_2(\hx,\hz) \times \pi_2(\hz,\hy)$ such that $\mu(\phi^\pm)=1$.
\end{claim}

For the proof of Claim \ref{claim boundary} see Theorem \ref{theorem bdd^2=0}.  
It can even be shown that any pair $(\phi^-,\phi^+)$ occurring in the right-hand side of Equation \eqref{amazing fact 3} also has the property that $\phi=\phi^- \sharp \phi^+$.

Now it follows from the definition that 
\[
\bdd^2 \hx= \sum_{\hy,\hz \in L_1 \cap L_2} \left( \sum_{(\phi^-,\phi^+)} \#_2\wM(\hx,\hz; \phi^-) \cdot \#_2 \wM(\hz,\hy ;\phi^+) \right) \cdot \hz,
\]
and thus by Equation \eqref{amazing fact 3}, it follows that 
\[
\bdd^2 \hx= \sum_{y \in L_1 \cap L_2} \left( \sum_{\{\phi \in \pi_2(\hx,\hy) \mid \mu(\phi)=2\}} \#_2 \bdd \overline{ \wM(\hx,\hy;\phi)} \right)=0,
\]
since $ \overline{\wM(\hx,\hy;\phi)}$ is a compact 1-manifold and so its boundary consists of an even number of points.

\subsection{Grading} \label{subsection grading}  It is possible to grade the Floer homology $HF(L_1,L_2)$ modulo $\bZ/N_{L_1,L_2} \bZ$; in particular, there is an absolute grading function
\[
\abs \cdot \colon L_1 \cap L_2 \to \bZ/N_{L_1,L_2} \bZ.
\]
 First, fix an arbitrary $\hx_0 \in L_1 \cap L_2$ and define $\abs {\hx_0}:=0$.  Then for any other $\hy \in L_1 \cap L_2$ define $\abs{\hy}:=\gr(\hx_0,\hy)$, where $\gr$ was given in Definition \ref{def grading function}.  Thus the grading depends on our initial choice of $\hx_0 \in L_1 \cap L_2$.  A different choice of $\hx_0$ would simply shift the grading. Denote by 
\[
CF_k(L_1,L_2) := \bigoplus_{\abs \hx = k} \bZ_2 \langle \hx \rangle.
\]

The key point is that the differential lowers the grading by one 
\[
\bdd \colon CF_* \to CF_{*-1},
\]
 and thus $HF(L_1,L_2)$ can also be graded using $\abs \cdot$.  We emphasise that the grading $\abs \cdot$ depends on the arbitrarily chosen $\hx_0$ (but only up to shift) and that this is only graded modulo $N_{L_1,L_2}$.

\subsection{The $W^{k,p}$ norms and Sobolev spaces} \label{subsection sobolev} We  move on to introduce some of the analysis behind the Floer machinery. We recommend Appendix B of \cite{McDS2} for a detailed exposition of the analysis below, written especially for symplectic geometers wanting to apply it to pseudo-holomorphic curves. 

As before, let $\Om$ denote a measurable bounded domain in some Euclidean space $\bR^l$.  Let $u \colon \Om \to \bR^N$ be a measurable map.  Then recall that the $L_p$ norm of $u$ for $p \in \bN$ is defined to be
\[
\norm u _p:=\left( \int_\Om \abs u ^p dx_1\ldots dx_l \right) ^{1/p},
\]
where $(x_1, \ldots, x_l)$ are taken to be coordinates of $\bR^l$. We use the following notation 
\[
L_p(\Om,\bR^N):=\{u \colon \Om \to \bR^N \mid u \text{ measurable}, \norm u _p <\infty\}.
\]
If the domain space $\Om$ is replaced with a compact manifold $\Si$, and the target space $\bR^N$ is replaced with a vector bundle $E$, then the definition works in a similar way by taking a finite atlas of charts on $\Si$, and making use of the linear structure on $E$.  Thus, we can speak of the space $\Ga^p(\Si,E)$ of measurable sections $u \colon \Si \to E$ with finite $L_p$-norm.  It is more complicated to define the $L_p$ spaces when the target space does not carry a linear structure; if $\Si$ and $M$ are both compact manifolds, then one can also obtain a well-defined space $L_p(\Si,M)$ by choosing a Riemannian metric on $M$ that gives rise to a Borel measure on $M$.  Compactness of $M$ implies that the obtained space $L_p(\Si,M)$ is independent of the choice of Riemannian metric on $M$.  Alternatively, we could define $L_p(\Si,M)$ by first choosing an embedding of $M$ into $\bR^N$, for some $N$, and then defining $L_p(\Si,M)$ as a subspace of $L_p(\Si,\bR^N)$.  Again, compactness of $M$ implies that the definition is independent of the choice of embedding.

The {\it Sobolev norm} is a generalisation of the $L_p$ norm.  In particular, the $(k,p)$ Sobolev norm of smooth map $u\colon \Om \to \bR^N$ is defined to be
 \[
 \norm u _{k,p} =\norm u _{W^{k,p}}:= \sum_{\al \leq k}\norm{ \bdd^\al u}_p.
 \]
 Here, $\al=\al_1 + \cdots + \al_j$ and $ \bdd^\al u:=\bdd^{\al_1}_{x_1} \cdots \bdd^{\al_j}_{x_j}$.  Now, the {\it Sobolev space} $W^{k,p}(\Om,\bR^N)$ is defined as the completion of the space 
 \[
 \{ u \in C^\infty(\Om,\bR^N) \mid \norm{u}_{k,p} <\infty \}
 \]
 with respect to the $(k,p)$ Sobolev norm.  As before, the definition extends to Sobolev spaces $W^{k,p}(\Si,E)$, where $E$ is a vector bundle over a compact manifold $\Si$.  Indeed, this is made possible because of the general fact that if $A \colon \bR^l \to GL(N,\bR)$ is a smooth function, $h \in C^\infty(\bR^l,\bR^l)$ and $u \in W^{k,p}(\Om,\bR^N)$, then $A \cdot (u \circ h)$ is also in $W^{k,p}(\Om,\bR^N)$ and there is an upper bound on its norm given by
 \[
 \norm{A \cdot (u \circ h)}_{k,p} \leq c \norm{u}_{k,p},
 \]
 where $c$ is a constant that does not depend on $u$ (see Remark B.1.23 \cite{McDS2}).  As a result, the Sobolev space $W^{k,p}(\Si,E)$ can be defined by taking a finite atlas of charts on $\Si$ (since a map with finite $W^{k,p}$ norm in one chart has finite norm in any other chart).  This leads to definition of the space $\Ga^{k,p}(\Si,E)$ as completion of the space $\Ga^\infty (\Si,E)$ of the smooth sections $u \colon \Si \to E$ in the $\norm \cdot _{k,p}$ norm.
 
 As before, defining the Sobolev spaces $W^{k,p}$ is more difficult if the target space does not carry a linear structure.  In particular, if $M$ and $\Si$ are compact manifolds, then the space $W^{k,p}(\Si,M)$ is well-defined when
 \begin{equation} \label{eq bound on dim}
 kp > \dim \Si.
 \end{equation}
Indeed, when $\Om$ is a bounded measurable domain in $\bR^l$ and $h \in C^k(\bR,\bR)$, then the map $C^\infty(\Om,\bR) \to C^\infty (\Om,\bR)$ given by $u \mapsto h \circ u$ extends to a map $W^{k,p}(\Om,\bR) \to W^{k,p}(\Om,\bR)$ only when $kp >l$; see \cite[Prop.\,B.1.19]{McDS2} and \cite[Rmk.\,B.1.24]{McDS2}.  Thus, if Equation \eqref{eq bound on dim} holds, then we can define the Sobolev space $W^{k,p}(\Si,M)$ by taking charts on both $\Si$ and $M$ (which are not defined if $kp<l$).  Similarly, we could define $W^{k,p}(\Si,M)$ by first embedding $M$ into $\bR^N$ for some $N$ and then defining $W^{k,p}(\Si,M)$ as a subspace of $W^{k,p}(\Si,\bR^N)$, but this is only independent of the choice of embedding when $kp>\dim \Si$.  

The Sobolev embedding theorem \cite[Sec.\,B.1]{McDS2} implies that for $\Om \subseteq \bR^l$, there is a compact embedding $W^{k,p}(\Om,\bR^N) \hookrightarrow C^0(\Om,\bR^N)$, when $kp>l$.  In particular, for $kp>\dim \Si$, the well-defined space $W^{k,p}(\Si,M)$ consists of continuous maps.  

In fact, if Equation \eqref{eq bound on dim} holds, then the space $W^{k,p}(\Si,M)$ carries the structure of a {\it Banach manifold}, that is, a manifold locally modelled on a Banach space.  It can be shown (see \cite{Eli}), that 
\[
T_uW^{k,p}(\Si,M)=\Ga^{k,p}(\Si,u^*TM).
\]
 
In our applications, we generally work with the space $W^{k,p}(\Si,M)$, where $\Si$ is a {\it non-compact} Riemann surface.  In this case, we need to impose additional conditions `at infinity', in order to get a well-defined Banach manifold.  One way to do this in the case when $\Si=S=\bR\times[0,1]$ is the following \cite[Sec.\,3]{F2}: fix a smooth cut-off function $\beta \colon \bR \to \bR$, such that $\be(s)=\abs s$ for $\abs s \geq 2$ and $\be(s)=0$ for $\abs s \leq 1$.  Fix $\de>0$.  Then define the norm $\norm \cdot _{k,p;\de}$ for $u \in C^\infty(S,M)$ by setting 
\[
\norm u _{k,p;\de}:=\norm{e^{\de \be (s)} u }_{k,p}.
\]
Then the {\it weighted Sobolev space} $W^{k,p}_\de(S,M)$ is defined to be the completion of $C^\infty(S,M)$ under the norm $\norm \cdot _{k,p;\de}$.  The point of this construction is that, provided $kp>2=\dim S$, the space $W^{k,p}_\de(S,M)$ is again a Banach manifold with tangent space 
\[
T_u W^{k,p}_\de(S,M)=\Ga^{k,p}_\de(S,u^*TM),
\]
where the space $\Ga^{k,p}_\de(S,u^*TM)$ is defined in a similar fashion \cite[Thm.\,3]{F2}.

\subsection{Properties of $J$-holomorphic curves}\label{subsection J-holo}
We assume for the remainder of this article that {\bf(F1)} holds.  Suppose also that $p>2$.  Define the space 
\[
W_\de^{1,p}(S,M;L_1,L_2):=\{u \in W^{1,p}_\de(S,M) \mid u(\bR\times \{0\}) \in L_1, u(\bR \times \{1\}) \in L_2\}.
\]
In words: we require that $u$ maps the `top' and `bottom' boundary components of $S$ into $L_1$ and $L_2$, respectively.  Since  $W_\de^{1,p}(S,M;L_1,L_2)$ is a submanifold of $W_\de^{1,p}(S,M)$  it is a Banach manifold, and its tangent space $T_uW_\de^{1,p}(S,M;L_1,L_2)$ at $u$ is the space of sections $\Ga^{1,p}_\de(S,u^*TM;u^*TL_1,u^*TL_2)$ defined to be
\[
\left\{ \xi \in \Ga^{1,p}_\de(S,u^*TM) \mid \xi(s,0)\in T_{u(s,0)}L_2, \xi(s,1)\in T_{u(s,1)}L_2, \forall s \in \bR \right\}.
\]

Now, define the Banach bundle $\mE_\de^p(S,M) \to W_\de^{1,p}(S,M;L_1,L_2)$ with fibre
\[
\mE_\de^p(S,M)_u:=\Ga_\de^{0,p}(S,u^*TM),
\]
and denote by $\ep_0$ its zero section.
The operator $\bar \bdd_J$ can be regarded as section
\[
\bar \bdd^p_J \colon W_\de^{1,p}(S,M;L_1,L_2) \to \mE_\de^p(S,M),
\]
with `$p$' on the $\bar \bdd _J^p$  just temporary notation to indicate the domain of $\bar \bdd _J$.

A key property of the operator $\bar \bdd^p_J $ is that it is {\it elliptic}, and thus enjoys the following {\it elliptic regularity property} given in Theorem B.4.1 \cite{McDS2}
\[
\text{for } u \in W_\de^{1,p}(S,M;L_1,L_2) \text{ and } \bar \bdd^p_J u=0 \implies u\in C^\infty (S,M).
\]
Therefore, the kernel of $\bar \bdd^p_J $ does not depend on $p$.  Since we are only interested in $J$-holomorphic maps, this shows that it does not matter what value of $p$ we take so long as $p>2$.  So from now on, we fix a particular $p>2$ (say, $p=3$) and always consider $\bar \bdd_J $ as an operator on this particular space $W_\de^{1,3}(S,M;L_1,L_2)$ thus dropping the superscript `$p$'.

Next, we want to show that $\mM_J(M,L_1,L_2)$ is a subset of of $\bar \bdd_J ^{-1}(\ep_0)$.  Since we know that the elements of $\bar \bdd_J ^{-1}(\ep_0)$ are smooth, we certainly have,
\[
\bar \bdd_J ^{-1}(\ep_0) \subseteq \mM_J(M,L_1,L_2).
\]
In order for the reverse inclusion to hold, every element $u \in  \mM_J(M,L_1,L_2)$ must satisfy the exponential decay condition required so that $\norm u _{1,p;\de} < \infty$.  In fact, with some work, it can be shown that every $J$-holomorphic curve $u$ in $ \mM_J(M,L_1,L_2)$ decays exponentially for some pair of intersection points $(\hx,\hy)$, which are also called the {\it asymptotes of $u$} \cite[Sec.\,14.1]{Oh book}.    As a consequence we have that 
\[
\bar \bdd_J^{-1}(\ep_0)= \mM_J(M,L_1,L_2),
\]
and 
\[
\mM_J(M,L_1,L_2):=\bigcup_{\hx,\hy \in L_1 \cap L_2} \mM_J(\hx,\hy).
\]

\subsection{The Fredholm property} \label{subsection fredholm}

Recall that a linear operator $F \colon V \to W$ between two Banach space $V,W$  is called {\it Fredholm} if the $ \Ker F$ and $\Coker F$ are finite-dimensional.  The {\it index of $F$} is 
\[
\ind F = \dim \Ker F - \dim \Coker F.
\]
  Clearly, if $V$ and $W$ are finite-dimensional, then $F$ is trivially Fredholm.  In particular, when $f\colon M \to N$ is a smooth map between finite-dimensional manifolds $M$ and $N$, and $q \in N$ is any regular value of $f$, then the implicit function theorem tells us that $f^{-1}(q)$ is a submanifold of $M$ of dimension 
\[
\dim \Ker df(p)=\ind df(p),
\]
 where $p$ is any point in $f^{-1}(q)$.
 
 In general, if $f \colon M \to N$ is a smooth Fredholm map between connected (possibly infinite-dimensional) manifolds, then the Fredholm index of $df(x)$ is independent of the choice of $x \in M$, hence it makes sense to define the {\it Fredholm index of $f$} to be the Fredholm index of $df(x)$ for some (and so any) $x\in M$.  However, the same is not true of $\dim \Ker df(x)$ and $\dim \Coker df(x)$ -- only their difference is independent of the choice of $x$.

Now, let $E \to M$ be a vector bundle over $M$, and $s \colon M \to E$ a section. Suppose that $s$ is transverse to the zero section $e_0$.  Then $Z:=s^{-1}(e_0)$ is a submanifold of $M$.  Saying that $s$ is transverse to $e_0$ is equivalent to making a statement about the vertical derivative.  Indeed, for every point $x\in Z$, the {\it vertical derivative} 
\[
Ds(x) \colon T_xM \to E_x
\]
is defined to be the linear map $\proj_{E_x} \circ ds(x)$, which is the composition of the differential $ds(x) \colon T_xM \to T_{s(x)}E$ and the fibre projection $\proj_{E_x} \colon T_{s(x)}E\to E_x$. (Note that in the statement of $\proj_{E_x}$ the splitting $T_{s(x)}=E_x \oplus T_xM$ is guaranteed only because $x \in Z$, that is $s(x) \in e_0$.)  So $s$ is transverse to $e_0$ if and only if $Ds(x)$ is surjective for all $x \in Z$.  Moreover, when this is the case, 
\[
\dim Z = \ind Ds(X), \text{ for any } x\in Z.
\]

The key point of working  with Banach bundles over Banach manifolds is that the same result also holds in this setting.

\begin{thm}  [Implicit Function Theorem for smooth sections of Banach bundles]  \label{thm implicit} Let $\mE \to \mM$ be a Banach bundle over a Banach manifold, and $s \in \Ga^\infty (\mM,\mE)$. Denote by \linebreak$Z:=s^{-1}(\text{zero section})$.  Suppose that for all $x \in Z$, the vertical derivative $Ds(x)$ is a surjective Fredholm operator.  Then $Z$ is a smooth submanifold of $\mM$ of finite dimension $\dim Z=\ind Ds(x)$ for some (and hence any) point $x \in Z$. 
\end{thm}
For a proof see Theorem A.3.3. \cite{McDS2}; technically they give a proof only for Banach spaces, but the more general result for Banach bundles follows by taking charts.

Going back to our setting, define for fixed $\hx,\hy \in L_1 \cap L_2$ and $\phi \in \pi_2(\hx,\hy)$ the space 
\[
\mB:=\mB(\hx,\hy;\phi):= \{ u \in W^{1,p}_\de(S,M;L_1,L_2) \mid u(-\infty, \cdot) \cong \hx, u(+\infty,\cdot) \cong \hy, u \in \phi \}.
\]
Let $\mE$ denote the Banach bundle over $\mB$ with fibre $\mE_u :=\Ga_\de^p(S,u^*TM)$.  As before, we can consider $\bar \bdd_J$ as a section of $\mE \to \mB$ and the moduli space $\mM_J(\hx,\hy;\phi)$ is its zero section.  Since $\bar \bdd _J$ is elliptic, the vertical derivative 
\[
D \bar \bdd_J(u) \colon T_u\mM_J(\hx,\hy;\phi) \to \mE_u
\]
is a Fredholm operator.  This is essentially a special case of the Riemann-Roch theorem; see Theorem C.1.10 \cite{McDS2} for a more general result.

\subsection{Transversality} \label{subsection transversality}
In general there is no reason why $D\bar \bdd_J(u)$ should be surjective for all $u\in\mM_J(\hx,\hy;\phi)$, and so allow us to conclude that $\mM_J(\hx,\hy;\phi)$ is a smooth manifold. It remains to show that we can find an almost complex structure $J$ such that $D\bar \bdd_J(u)$ is surjective. We do so by applying the Sard-Smale theorem to an appropriate Fredholm map between Banach manifold. The details are as follows.  Define
\[
\mJ:=\left\{ J \in C^\infty\left([0,1], \Ga^{1,p}(M,\End(TM)\right) \mid J(t) \text{ is $\om$-compatible al. cx. st. } \forall t\in[0,1]\right\}.
\]
Now, consider the extended operator $\mF \colon \mJ \times \mB(\hx,\hy;\phi) \to \mE$ given by 
\[
\mF(J,u)(s,t):=\left(\bar \bdd _{J(t)}u\right)(s,t).
\]
Since $D\bar \bdd_{J(t)}(u)$ is always Fredholm, it follows that the same is true for $D\mF(J,u)$.  The advantage of working with the extended operator is that the extra freedom coming from $\mJ$ is enough to show that $D\mF(J,u)$ {\it is} transverse to the zero section of $\mE$.  The proof is not easy; see \cite[Thm.\,3.2]{McDS2} for a similar, although slight simpler result, or \cite{FHS} for detailed proofs in the related setting of Hamiltonian Floer homology.   In any case, it now follows from Theorem \ref{thm implicit} that the so called {\it universal moduli space} $\Muniv(\hx,\hy;\phi):=\mF^{-1}(\text{zero section})$ is a Banach manifold.  

Consider the projection 
\[
\pi \colon \Muniv(\hx,\hy; \phi) \to \mJ.
\]
Then $\pi$ is a map between Banach manifolds.  Since $\pi$ is a projection, we see that $\pi$ is also Fredholm with 
\[
\ind \pi(J,u) = \ind D\bar \bdd_J(u).
\]
By the infinite-dimensional version of the Sard-Smale theorem we have that a generic $J \in \mJ$  is a regular value for $\pi$.  But since $\pi$ is a projection, it is easy to see that $J$ is a regular value of $\pi$ if and only if $D \bar \bdd_J(u)$ is surjective for all $u \in \mM_J(\hx,\hy;\phi)$ (see \cite[Lem.\,3.4]{We}).  

We have now essentially completed the proof of half of Theorem \ref{thm manifold}.  There is one caveat.  We have only discussed the result for $J$ of class $W^{1,p}$, whereas in Theorem \ref{thm manifold}, the result was stated for smooth $J$.  The result is still true in the smooth setting, but requires an additional argument that we omit for simplicity; we refer the reader to \cite[pp.\,54--55]{McDS2} for details.  (The extra difficulty comes from the fact that the space $C^\infty([0,1],\mJ(M,\om))$ is not a Banach manifold).  

\begin{rmk}
Since $\pi_2(\hx,\hy)$ is countable, and the intersection of countably many generic sets is also generic, it follows that generically the full moduli space 
\[
\mM_J(\hx,\hy):= \bigcup_{\phi \in \pi_1(\hx,\hy)} \mM_J(\hx,\hy;\phi)
\]
also admits the structure of a smooth manifold, as asserted in Claim \ref{claim manifold}.
\end{rmk}

\subsection{The spectral flow} \label{subsection spectral}

In order to complete the proof of Theorem \ref{thm manifold}, it remains to show that for any $\hx,\hy \in L_1 \cap L_2$, any $J \in \mJ(M,\om)$ and any $\phi \in \pi_2(\hx,\hy)$, we have 
\begin{equation} \label{eq fredholm same as maslov}
 \ind D\bar \bdd_J(u)=\mu(\phi) \hspace{1cm} \text{for any } u \in \mM_J(\hx,\hy;\phi).
\end{equation}
Actually, under Assumption {\bf(F1)}, Equation \eqref{eq fredholm same as maslov} is true regardless of whether $D \bar \bdd_J(u)$ is surjective or not.

For the remainder of this subsection, fix $u\in \mM_J(\hx,\hy;\phi)$.  In order to compute the index of $D \bar \bdd_J(u)$ we introduce the {\it spectral flow}, following Robbin and Salamon \cite{RS}.

Let $W \subset H$ be two Hilbert spaces with $W$ dense in $H$.  Set $\mA(W,H)$ to be the set of all maps $A(s)_{s \in \bR}$ such that for each $s \in \bR$: a) the map $A(s)$ is an unbounded self-adjoint linear operator on $H$ with domain $W$, b) $s \mapsto A(s)$ is continuously differentiable, and c) there exist invertible operators $A^{\pm} \in \mL(W,H)$ such that
\[
\lim_{s \to \pm \infty} \norm{A(s) - A^\pm}_{\mL(W,H)}=0.
\] 
Given such a family $A(s)_{s\in\bR}$, define a map
\[
D_A \colon W^{1,2}(\bR,W) \cap L^2(\bR,H) \to L^2(\bR,H),
\]
by setting 
\begin{equation} \label{def of DA}
(D_A w)(s):=\frac{\bdd w}{\bdd s} (s) + A(s) \cdot w(s).
\end{equation}
The operator $D_A$ is actually a Fredholm operator; although, strictly speaking, slightly more stringent conditions are needed on $A(s)_{s\in \bR}$ for this to be true, specifically see Conditions (A-1), (A-2) and (A-3) \cite[p.\,7]{RS}.

Given Hilbert spaces $W$ and $H$ as above, there is a unique way of defining the a map $\mu_s \colon \mA(W,H) \to \bZ$  called the {\it spectral flow} $\mu_s$  so that certain axioms are satisfied; for the precise statement see \cite[Thm.\,4.23]{RS}.  Morally speaking, the spectral flow of a family of maps $A(s)_{s \in \bR}$ counts the change in the number of negative eigenvalues of $A(s)$ as $s$ ranges from $-\infty$ to $+\infty$. 

 A famous theorem by Atiyah, Patodi and Singer \cite{APS} (see also \cite[Thm.\,A]{RS}), says that the spectral flow of a given $A(s)_{s \in \bR} \in \mA(W,H)$ is precisely equal to the index of a Fredholm operator $D_A$ of the form  Equation \eqref{def of DA}, that is, 
 \begin{equation}
 \ind D_A = \mu_s\left(A(s\right)_{s\in \bR}).
 \end{equation}

 \subsection{An interlude on Morse theory} \label{subsection Morse theory} Consider briefly the finite-dimensional Morse theoretic case: let $M$ be a closed $n$-manifold and $f \colon M\to \bR$ a Morse function.  Denote by $\nabla f$ the gradient of $f$ with respect to a given metric on $M$, and denote by $\varphi_t$ the flow of $-\nabla f$.  Then to equip $M$ with a metric that is {\it Morse-Smale} for $f$ means that the stable and unstable manifolds of $\varphi_t$ are always transverse:
 \begin{equation} \label{eq transversality}
  W^u(\hx) \pitchfork W^s(\hy), \hspace{1cm} \text{ for all } \hx,\hy \in \text{\{critical points of } f\}.
 \end{equation}

In this setup we can do Morse homology with $f$; in particular, define a $\bZ$ graded $\bZ_2$-vector space with generators the critical points of $f$, and graded by the Morse index $\nu(\hx):=\dim W^u(\hx)$.  The boundary operator $\bdd$ counts the parity of the set $(W^u(\hx) \cap W^s(\hy))/\bR$ whenever $\nu(\hx)-\nu(\hy)=1$.  Note that here we are dividing out by the $\bR$-action as in the Floer case, and that the transversality assumption from Equation \eqref{eq transversality} guarantees that $(W^u(\hx) \cap W^s(\hy))/\bR$ is a finite set.

Further, in the spirit of Floer homology, we can instead define $\mM(\hx,\hy)$ to be the zero set of a suitable section $\si$ of a Banach bundle that is now defined for $h\colon \bR \to M$ by setting
\[
\si(h):=\bdd_s h + \nabla f(h).
\]
There is an identification of $\mM(\hx,\hy)$ with $(W^u(\hx) \cap W^s(\hy))$ via $h \mapsto h(0)$.

We can then consider the vertical derivative $D \si(h)$ at zero of $h$.  Let $\nabla$ be the Levi-Civita connection with respect to some (background) Riemannian metric on $M$.   Then $D \si(h)$ is the operator
\begin{gather*}
D \si(h) \colon \Ga^{1,2}(\bR,h^*TM) \to \Ga^{0,2} (\bR,h^*TM), \\
D \si (h)(\xi)=\nabla_{\frac{\bdd h} {\bdd s}} \xi + \nabla _\xi(\nabla f(h)),
\end{gather*}
see \cite[Chap.\,2]{Sch}.  As in the Floer case, the operator $D \si(h)$ is Fredholm, and the dimension of $\mM(\hx,\hy)$ is equal to the Fredholm index of $D\si(h)$.  We already know that the dimension of $\mM(\hx,\hy)$ is the dimension of $(W^u(\hx) \cap W^s(\hy))$, namely $\nu(\hx)-\nu(\hy)$, but let us now explain how this can also be computed using the spectral flow.  Indeed, trivialise the bundle $h^*TM \to \bR$ using parallel translation with respect to $\nabla$.  This gives an operator of the form $D_A$ considered above, where we take $W=H=\bR^n$.  In particular,
\[
D_A \colon W^{1,2}(\bR,\bR^n) \to L^2(\bR,\bR^n),
\]
where under this trivialisation,
\[
A(s)(w)=\nabla_w \nabla f(h(s)).
\]
Here the limit operators $A^\pm$ are precisely the Hessians of $f$ at the critical points $\hx,\hy$:
\[
A^-=\text{Hess}(f;\hx), \hspace{1cm} A^+=\text{Hess}(f;\hy).
\]
Therefore, the Fredholm index of $D\si(h)$ is equal to the spectral flow of $A(s)_{s \in \bR}$.  But then in this case the spectral flow is easy to compute, since it is just the change in the number of negative eigenvalues of $A(s)$ as $s$ ranges from $-\infty$ to $+\infty$, which is precisely $\nu(\hx) - \nu(\hy)$.

\subsection{The Floer case}\label{subsection floer case}

We now return to the Floer setup.  Recall that we have fixed a flow line $u$ from $\hx$ to $\hy$ and a homotopy class $\phi\in \pi_2(\hx,\hy)$.  Take a symplectic trivialisation $\Phi \colon S \times \bR^{2n} \to u^*TM$, that is, take a diffeomorphism
\[
\Phi_{s,t} \colon (\bR^{2n}),\om_{std}) \to T(_{u(s,t)}M,\om), 
\]
such that $\Phi_{s,t}^*(\om)=\om_{std}$.  In this trivialisation the vertical derivative $\bD \bar \bdd_J(u)$ becomes the map
\begin{gather*}
D_A \colon W^{1,2}(S,\bR^{2n}) \to L^2(S,\bR^{2n}),
\end{gather*}
where this time
\[
A(w)(s,t)=i \frac{\bdd w}{\bdd t}(s,t) + S(s,t) \cdot w(s,t),
\]
for $S(s,t) \in GL(2n,\bR)$ a family of matrices determined by the trivialisation whose asymptotes $S(\pm \infty, \cdot)$ are symmetric.  As before we compute the index of $D \bar \bdd _J(u)$ using the spectral flow of $A(s)_{s\in \bR}$.  Strictly speaking, there are additional complication that we are sweeping under the carpet.  For example, $A(s)$ is not self-adjoint for $s \neq \pm \infty$.  However, $A(s)$ is self-adjoint (if and only) if $S(s,\cdot)$ is symmetric for all $t$, which can be achieved by a compact perturbation of $S(s,t)$, and compact perturbations do not affect the Fredholm index.  There are other complications as well: see \cite[Sec.\,7]{RS} for the details.

It remains to compute the spectral flow $\mu_s\left(A(s)\right)_{s \in \bR}$.  This is where the Maslov index comes in: it can be shown that 
\[
\mu_s\left(A(s)\right)_{s\in\bR}=\mu(\phi)
\]
This is a non-trivial result, and we do not discuss it here; see \cite[Thm.\,1]{F1} for a proof, or \cite[Sec.\,4]{Se} for a friendly discussion.  This completes the proof of Theorem \ref{thm manifold}.

\subsection{Energy}\label{subsection energy}

Let us now briefly recall the definition of the {\it energy of a map $u$} from a Riemann surface into a Riemannian manifold. Let $(X,g)$ be a Riemannian manifold and $(\Si,j)$ a Riemannian surface equipped with a complex structure $j$.  The almost complex structure $j$ defines a {\it conformal class} of Riemannian metrics on $\Si$ -- namely, the set of metrics $h$ that satisfy $h(v,jv)=0$ and $h(jv,jv)=h(v,v)$ for all $v$.  The {\it energy with respect to $(g,j)$ of a smooth map $u \colon \Si \to X$} is defined to be
\[
E_{g,j}(u):= \frac{1}{2} \int_\Si \norm{du}^2_{g,h} \vol_h,
\]
where $h$ is any conformal metric (see for example \cite[Sect.\,8]{Jost}).  Here the {\it energy density} 
\[
 \norm{du}^2_{g,h} \colon \Si \to \bR
 \]
  is given by 
\[
 \norm{du}^2_{g,h} := \frac{g(du_z(v),du_z(v)) + g(du_z(jv),du_z(jv))}{h(v,v)}, \hspace{1cm} \text{for any } v\neq 0 \in T_z\Si.
\]
It is easy to check that the value of the right-hand side does not depend on the choice of non-zero $v \in T_z\Si$, hence $ \norm{du}^2_{g,h} $ is well-defined.  The 2-form $ \norm{du}^2_{g,h}  \vol_h$ only depends on the conformal class of $h$.  Indeed, if $h'$ is another metric conformably equivalent to $h$, then we can write $h'=\rho h$ for some positive function $\rho$.  Then 
\[
 \norm{du}^2_{g,h'} =\frac{1}{\rho}  \norm{du}^2_{g,h}  \hspace{1cm} \text{and} \hspace{1cm} \vol_{h'} = \rho \vol_h,
\]
hence 
\[
 \norm{du}^2_{g,h} \vol_h= \norm{du}^2_{g,h'} \vol_{h'}.
\]

We are interested in the energy in the special case where $X$ is our symplectic manifold $M$, and $g$ is given by $\om(\cdot, J \cdot)$, where $J$ is a compatible almost complex structure on $M$.  We write $E_{J,j}(\cdot)$ instead of $E_{\om(\cdot, J \cdot), j}(\cdot)$, or even just $E_J(\cdot)$ when $j$ is clear.  In the case when $u \colon \Si\to M$ is $(J,j)$-holomorphic, then  
\begin{equation} \label{eq energy}
E_{J,j}(u)=\int_\Si u^*\om.
\end{equation}
Indeed, if we choose local complex coordinates $s+it$ and a metric $h$ at a point $z \in \Si$ so that 
\[
j\left(\frac{\bdd}{\bdd s}\right)=\frac{\bdd}{\bdd t} \hspace{0.5cm} \text{and} \hspace{0.5cm}j\left(\frac{\bdd}{\bdd t}\right)=-\frac{\bdd}{\bdd s},  \hspace{0.5cm} \text{and} \hspace{0.5cm}
h_z\left(\frac{\bdd}{\bdd s}, \frac{\bdd} {\bdd s}\right)=1,
\]
then 
\[
\vol_h(z)=ds \wedge dt.
\]
Further, if we compute $ \norm{du}^2_{J,j}(z)$ using the non-zero $v:=\frac{\bdd}{\bdd s}$ we get
\begin{equation*}
\begin{split}
 \norm{du}^2_{J,j}(z) \vol_h(z) &= \left(g \left(du_z \left(\frac{\bdd}{\bdd s}\right), du_z \left(\frac{\bdd}{\bdd s}\right) \right)+g\left(du_z \left(\frac{\bdd}{\bdd t}\right), du_z \left(\frac{\bdd}{\bdd t}\right) \right)\right) ds \wedge dt \\
 & = \left( \om \left( \frac{\bdd u}{\bdd s}, J \frac{\bdd u}{\bdd s} \right)+ \om  \left( \frac{\bdd u}{\bdd t}, J \frac{\bdd u}{\bdd t} \right) \right) ds \wedge dt\\
 & = \left( \om \left( \frac{\bdd u}{\bdd s}, \frac{\bdd u}{\bdd t} \right)+ \om  \left( \frac{\bdd u}{\bdd t}, - \frac{\bdd u}{\bdd s} \right) \right) ds \wedge dt\\
& = 2\om \left( \frac{\bdd u}{\bdd s}, \frac{\bdd u}{\bdd t}\right) ds \wedge dt \\
&= 2u^*\om(z),
\end{split}
\end{equation*}
where the second and third equality use the fact that $u$ is $(J,j)$-holomorphic.

This prompts the following definition.
\begin{deff}
Given $\hx,\hy \in L_1 \cap L_2$ and $\phi \in \pi_2(\hx,\hy)$, define the {\it $\om$-area} of $\phi$, written $\mA_\om(\phi)$, to be 
\[
\mA_\om(\phi):=\int_\bD v^* \om,
\]
where $v\colon \bD \to M$ is any smooth map satisfying
\begin{align*}
& v(\bdd^- \bD) \subset L_1, \hspace{1cm}  v(\bdd^+ \bD) \subset L_2, \\
& v(-i)=\hx,  \hspace{1.5cm} v(i)=\hy,
\end{align*}
\end{deff}

The following corollary is now immediate.

\begin{cor} \label{cor energy bound}
Suppose $J$ is $\om$-compatible, and $\hx,\hy \in L_1 \cap L_2$ and $\phi \in \pi_2(\hx,\hy)$.  Suppose $v \in \mM_J(\hx,\hy;\phi)$.  Then 
\[
E_J(v)=\mA_\om(\phi).
\]
In particular, if $\hx \neq \hy$, then a class $\phi \in \pi_2(\hx,\hy)$ only admits holomorphic representatives when $\mA_\om (\phi)>0$. 
\end{cor}

\subsection{Gromov compactness} \label{subsection gromov}

Let $(u_n)_{n \in \bN} \subset \mM(\hx,\hy;\phi)$ be a sequence of $J$-holomorphic curves.  Then $z \in S$ is called a {\it singular point} of the sequence $(u_n)$ if there is a sequence of points $z_n \in S$ such that $\norm{du_n(z_n)}_{L^\infty} \to \infty$.   The set $\De$ of singular points of $(u_n)$ is divided into three subsets: 
\[
\De=\De^{\Int} \cup \De^1 \cup \De^2.
\]
  Here $\De^{\Int}$ is the set of singular points contained in $\bR \times (0,1)$, $\De^1$ is the set of singular points in $\bR\times \{0\}$, and $\De^2$ is the set of singular points in $\bR \times \{1\}$. 

We now state the two key compactness theorems that we collectively call {\it Gromov compactness}, although in this context they are due to Floer \cite{F2}.  However, the versions quoted below are somewhat stronger than the original result Floer proved -- specifically, conclusion (iv) containing Equation \eqref{eq energy of bubbles} in Theorem \ref{thm gromov compactness II} which uses the so called `hard-rescaling' (proved in \cite[Chap.\,4]{McDS2}).

\begin{thm} [Gromov compactness I] \label{thm gromov compactness I}
Fix $\hx,\hy \in L_1 \cap L_2$ and fix $E \geq 0$.  Then the space $\mM^{\leq E}(\hx,\hy)$ consisting of all flow lines $u \in \mM(\hx,\hy)$ with $E_J(u) \leq E$ is pre-compact in the $C^\infty_{loc}(S,M)$ topology.  In particular, there are at most finitely many classes $\phi \in \pi_2(\hx,\hy)$ for which $\mA_\om(\phi) \leq E$, $\mM(\hx,\hy;\phi)\neq \emptyset$.
\end{thm}

The following theorem can be viewed as a refinement of the previous one, which explains more precisely the stated compactness properties.

\begin{thm}  [Gromov compactness II] \label{thm gromov compactness II}  Fix $\hx,\hy \in L_1 \cap L_2$ and $\phi \in \pi_2(\hx,\hy)$.
Let $(u_n)_{n \in \bN} \subset \mM(\hx,\hy;\phi)$ and $(s_n)_{n \in \bN} \subseteq \bR$ be a sequence of real numbers.  Set $w_n(s,t):=u_n(s+s_n,t)$.  Then (after possibly passing to a subsequence) there exist:
\begin{enumerate}
\item points $\hx',\hy' \in L_1 \cap L_2$ and a class $\psi \in \pi_2(\hx',\hy')$,
\item a finite set $\De \subseteq S$,
\item (another) subsequence $( {n(i)})_{i \in \bN}$, 
\item a curve $w \in \mM_J(\hx',\hy'; \psi)$, 
\end{enumerate}
such that the following statements hold.
\begin{enumerate}
\item On every compact set $K \subset S - \De$, the sequence $(w_{n(i)})_{i \in \bN}$ converges in the $C^\infty$ topology on $K$ to $w$.
\item  The set $\De$ is the set of singular points of $(w_{n(i)})_{i \in \bN}$.  If we write $\De = \De^{\Int} \cup \De^
1 \cup \De^2$, then for each $z \in \De^{Int}$ there exists a non-constant $J$-holomorphic map $v_z \colon S^2 \to M$, and for each $z \in \De^k$ ($k=1,2$), there exists a non-constant $J$-holomorphic map $v_z \colon (\bD, \bdd \bD)\to (M,L_k)$.
\item If $\hx' \neq \hy'$, then $\mu (\psi) \geq 1$. If either $\hx' \neq \hx$ or $\hy' \neq \hy$, then $\mu (\psi) < \mu(\phi)$.  If both $\hx' \neq \hx$ and $\hy' \neq \hx$, then $\mu(\psi) < \mu(\phi) -1$.  
\item We have that 
\begin{equation} \label{eq energy of bubbles}
E_J(w)+\sum_{z \in \De} E_J(v_z) \leq \lim \sup_i E_J(w_{n(i)}).
\end{equation}
If $\De = \emptyset$, then \eqref{eq energy of bubbles} is an equality. In particular, if $\hx\neq \hy$ but $\hx'=\hy'$ (so that the $w_n$-s are non-constant, but $w$ is constant), then $\De \neq \emptyset$. 
\item We have that 
\begin{equation} \label{equation on maslov indices}
\mu(\psi) + 2 \sum_{z \in \De^{Int}} c_1(v_{z^*}([S^2])) +\sum_{z \in \De^1} \mu_{L_1} (v_z) + \sum_{z \in \De^2} \mu_{L_2}(v_z) \leq \mu(\phi).
\end{equation}
\end{enumerate}
\end{thm}

\begin{rmk} \label{rmk two things}
This is about as far as we can get if we only assume {\bf (F1)}.  In order for the Floer homology to be well-defined, we need to know the following two things.
\begin{enumerate}
\item  For any $\hx,\hy \in L_1 \cap L_2$, there are at most finitely many classes $\phi \in \pi_2(\hx,\hy)$ such that $\mu(\phi)=1$, and such that $\mM(\hx,\hy;\phi)$ is non-empty (so that the sum in the definition of the differential given in Equation \eqref{eq deff of differential} is finite).  This requires {\bf(F1--3)} to hold, and is proved in Corollary \ref{for finitely many classes} below.  

\item If $\mu(\phi)=1$ or $\mu(\phi)=2$, and $(u_n)_{n \in \bN} \subseteq \mM(\hx,\hy;\phi)$, then there are no singular points. For $\mu(\phi)=1$, together with {\bf(F1--3)}, this implies that $\mM(\hx,\hy;\phi)$ is compact; we give the proof in Theorem \ref{thm compactness}.  For $\mu(\phi)=2$, together with {\bf(F1--4)}, this is used to prove that $\bdd^2=0$; see Theorem \ref{theorem bdd^2=0}.
\end{enumerate}
\end{rmk}

As a matter of terminology, we say that {\it bubbling} cannot happen if there are no singular points.  Indeed, Theorem \ref{thm gromov compactness II} says that at each singular point, at least one `bubble' appears: that is, a holomorphic sphere or a holomorphic (boundary) disk.

\begin{rmk} \label{rmk c1 of sphere}
Note that asking $L$ to be positively monotone implies that for any map $v \colon S^2 \to M$ satisfying $v(\text{north pole}) \in L$, (regarded as map $v \colon (\bD, \bdd \bD) \to (M,L)$ as explained in Remark \ref{rmk minimal Chern}), we have
\begin{equation} \label{eq c1 of sphere2}
2c_1(v_*([S^2]))=\mu_L(v)=c \int_\bD v^*\om.
\end{equation}
 In particular, if $v$ is a non-constant $J$-holomorphic function, then $c_1(v_*([S^2]))$ is positive since $c>0$ and $\int_\bD v^* \om=E(v)>0$.
\end{rmk}

\begin{lemma} \label{lemma monotone constant}  Suppose $(M,\om)$ admits a monotone Lagrangian $L$ with positive constant $c$ as in Definition \ref{deff monotone}, and suppose $N_L=k \in \bN$.  Then $2N_M$ is divisible by $k$.  In particular, $N_M \geq \lceil k \rceil /2$.  Moreover, if $L'$ is another monotone Lagrangian, then $L'$ is monotone with the same positive constant $c$.
\end{lemma}
\begin{proof}
Immediate from Equation \eqref{eq c1 of sphere2}.
\end{proof}

Recall that {\bf (F3)} asks that at least one of the Lagrangians $L_k$ has the property that \linebreak $\iota_{k*}(\pi_1(L_k)) \subseteq \pi_1(M)$ is torsion, where $\iota_{k*}$ is the induced map on $\pi_1$ arising from the inclusion $\iota_k \colon L_k \hookrightarrow M$.  Equivalently, if $\om \colon S^1 \to L_k$ is any loop, then there exists $m \in \bN$ such that the iterated loop $\ga^m \colon S^1 \to L_k$ bounds a disk in $M$ (not in $L_k$!).

\begin{lemma} \label{lemma energy bound}
Assume {\bf (F1--3)} hold.  Fix $k\in \bZ$, and two points $\hx,\hy \in L_1 \cap L_2$.    Then there exists a constant $C_{k}>0$ such that for any $\phi \in \pi_2(\hx,\hy)$ with $\mu(\phi)=k$ we have
\[
\mM_J(\hx,\hy; \phi) \neq \emptyset \implies \mA_\om(\phi)=C_k.
\]
\end{lemma}

\begin{proof}
Without loss of generality let us suppose that $\iota_{1*}(\pi_1(L_1))$ is torsion.  Suppose $v_1, v_2 \in \mM(\hx,\hy)$ satisfy $\mu(v_1)=\mu(v_2)$.  Consider $\ga \colon S^1 \to L_2$ defined by $\ga:=v_1|_{\bdd^- \bD} * v_2^{-1}|_{\bdd^- \bD}$.  There exists $m \in \bN$ such that $\ga^m$ bounds a disk $w \colon (\bD,\bdd \bD) \to (M,L_1)$.  Now we stitch the (trivial) $m$-fold covers of $v_1$ and $v_2$ together with $w$ to get another disk $f \colon (\bD, \bdd \bD) \to (M,L_2)$ with boundary in $L_2$.  Note that 
\begin{equation*}
\begin{split}
\int_\bD f^* \om &= m \int_\bD v_1^* \om - m \int_\bD v_2^*\om+\int_\bD w^*\om \\
&= mE(v_1) - mE(v_2) + \int_\bD w^*\om,
\end{split}
\end{equation*}
where the second equality comes from Equation \eqref{eq energy}.  

By monotonicity of $L_2$ we have that 
\[
\mu_{L_2}(f)=c \int_\bD f^* \om=c\cdot m \left( E(v_1)-E(v_2) \right) + c \int_\bD w^* \om.
\]
On the other hand, by additivity of the Maslov index from Equation \eqref{additivity of Maslov} we can also write
\begin{equation*}
\begin{split}
\mu_{L_2}(f) &= m \mu _{L_1,L_2}(v_1) - m \mu_{L_1,L_2}(v_2) + \mu_{L_1}(w)\\
&= m \mu(\phi) - m \mu(\phi) + \mu_{L_1}(w) \\
&=c \int_\bD w^* \om,
\end{split}
\end{equation*}
where in the last equation we use Lemma \ref{lemma monotone constant} to say that the monotonicity constant $c$ is the same for both $L_1$ and $L_2$.  
Finally, combining the two expressions for $\mu_{L_1}(f)$ we find that 
\[
c\cdot m\left( E(v_1) - E(v_2)\right)=0.
\]
Since $m \neq0$ and $c \neq 0$, it follows that $E(v_1)=E(v_2)$. Hence if $v_1 \in \phi_1$ and $v_2 \in \phi_2$, then $\mA_\om(\phi_1)=\mA_\om(\phi_2)$.  

\end{proof}

\begin{cor} \label{for finitely many classes}
Fix $\hx,\hy \in L_1 \cap L_2$ and $k \in \bZ$.  There are at most finitely many classes $\phi \in \pi_2(\hx,\hy)$ such that $\mu(\phi)=k$ and such that $\mM(\hx,\hy;\phi)$ is non-empty. 
\end{cor}

\begin{proof}
Theorem \ref{thm gromov compactness I} and Lemma \ref{lemma energy bound}.
\end{proof}

\begin{rmk}
Note that the proof of Lemma \ref{lemma energy bound} only needs that the monotonicity constant was non-zero.  Thus the proof would still go through if $c$ was negative.  Nevertheless. we need to assume that $c>0$ in order for Remark \ref{rmk c1 of sphere} to be valid. 
\end{rmk}

We can now finally prove Claim \ref{claim compactness}.

\begin{thm} \label{thm compactness}
Assume that {\bf (F1--3)} hold.  Then the manifold
\[
\bigcup_{\{\phi \in \pi_2(\hx,\hy) | \mu(\phi)=1\}} \wM(\hx,\hy;\phi)
\]
 is compact.
\end{thm}
\begin{proof}
Let $(u_n)_{n \in \bN}$ be any sequence in $\mM(\hx,\hy; \phi)$ and $(s_n)_{n \in \bN}$ be any reparametrisation sequence.   Then we must show that there exists a subsequence $(n(i))_{i \in \bN}$ and $u \in \mM(\hx,\hy;\phi)$ such that if $w_n(s,t):=u_n(s+s_n,t)$, then $w_{n(i)} \to u$ on every compact subset of $S$.  

We apply Theorem \ref{thm gromov compactness II}.  This gives us the desired sequence $(n(i))_{i \in \bN}$.  Suppose $w_{n(i)} \to u$ on $S - \De$, with $u \in\mM(\hx',\hy')$.  We need to show that $\De =\emptyset$ and that $\hx=\hx'$, and $\hy=\hy'$. 

Firstly, note that if $u$ a holomorphic map in $\psi \in \pi_2(\hx',\hy')$, then we have $\mu(\psi)\geq 0$.

Secondly, if $\De \neq \emptyset$, then $\De^{\Int} \neq \emptyset$ or $\De^i \neq \emptyset$ for $i=1,2$.  If $z \in \De^{\Int}$, then by Remark \ref{rmk c1 of sphere}, (or if $\De^i \neq \emptyset$ by Remark \eqref{rmk c1 of sphere},) we have that the left-hand side of Equation \eqref{equation on maslov indices} is at least 2.  Therefore, $\De=\emptyset$.  

Now, since $\De =\emptyset$, by Theorem \ref{thm gromov compactness II} we have $\hx' \neq \hy'$, and hence $\mu(\psi) \leq 1$.  Thus $\mu (\psi)=\mu(\phi)$, so we have $\hx'=\hx$ and $\hy'=\hy$.   Therefore, $u \in \mM(\hx,\hy;\phi)$.  
\end{proof}

The next result is the only place where we use {\bf (F4)}, and is half of the proof of Claim \ref{claim boundary}.

\begin {thm} \label{theorem bdd^2=0}
Assume that {\bf (F1--4)} holds.  Fix $\hx,\hy \in L_1 \cap L_2$, and fix $\phi \in \pi_2(\hx,\hy)$ with $\mu(\phi)=2$.  Fix $(u_n)_{n \in \bN} \subset \mM(\hx,\hy; \phi)$, and suppose that $(u_n)$ has no subsequence converging to an element of $\mM(\hx,\hy;\phi)$.  Then there exists $\hz \in L_1 \cap L_2$, and $\phi^- \in \pi_2(\hx,\hz)$ and $\phi^+ \in \pi_2(\hz,\hy)$ such that $\mu(\phi^\pm)=1$ and such that $\phi=\phi^- \sharp \phi^+$ have the following property. For any two sequences 
\[
(s_n^\pm)_{n \in \bN} \subset \bR \hspace{0.5cm} \text{with} \hspace{0.5cm} s_n^\pm \to \pm \infty,
\]
if $w_n^\pm(s,t):=u_n(s+s_n^\pm,t)$, then 
\begin{enumerate}
\item a subsequence of the $w_n^-$ converges to some $w^- \in \mM(\hx,\hz; \phi^-)$, and
\item a subsequence of the $w_n^+$ converges to some $w^+ \in \mM(\hz,\hy;\phi^+)$.
\end{enumerate}
\end{thm}
\begin{proof}
Let $(s_n^+)$ denote a sequence converging to $+\infty$ and set $w_n^+(s,t):=u_n(s+s_n^+,t)$.  By Theorem \ref{thm gromov compactness II}, up to a subsequence, we may assume that $w_n \to w$ on $S - \De$ for some $w \in \mM(\hx,\hy')$ and some finite subset $\De$.  We show that $\hx'=\hy$ and that $\De=\emptyset$. Since $\mu(\phi)=2$, the right-hand side of Equation \ref{equation on maslov indices} is 2.  The assumption {\bf (F4)} implies that if $\De \neq \emptyset$, then the left-hand side is at least 3.  Indeed,  {\bf (F4)} implies that $\De^1=\De^2=\emptyset$, and from Lemma \ref{lemma monotone constant} we have that $N_M \geq 2$, and hence if $\De^{\text{int}} \neq \emptyset$, then the left-hand side would be at least $4$.
\end{proof}

\begin{figure}[h]
\centering
\includegraphics [scale=0.50]{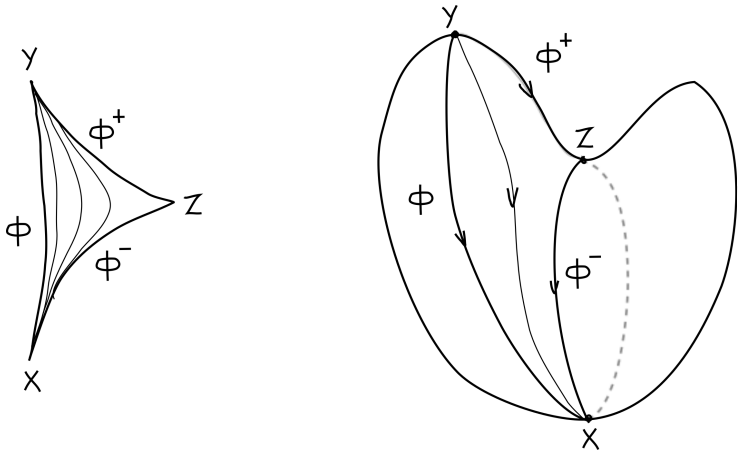}
\caption{Illustrations of broken trajectories: abstract (left), concrete (Morse homology) trajectories on a sphere (right).}
\label{fig broken trajectories}
\end{figure}

Thus we have shown that if $\wM(\hx,\hy;\phi)$ is $1$-dimensional, then it can be compactified into a manifold $\overline{\wM(\hx,\hy;\phi)}$ in such a way that 
\begin{equation}
\bdd \overline{\wM(\hx,\hy;\phi)} \subset \bigcup_{z \in L_1 \cap L_2} \bigcup_{(\phi^-,\phi^+)} \wM(\hx,\hz;\phi^-) \times \wM(\hz, \hy; \phi^+),
\end{equation}
where the second union is over all pairs $(\phi^-,\phi^+) \in \pi_2(\hx,\hz) \times \pi_2(\hz,\hy)$ such that $\mu(\phi^\pm)=1$.  In order to prove the converse, we need to know that given an element of the right-hand side we can obtain it as a limit of curves in $\wM(\hx,\hy;\phi)$.  This is the content of the {\it Floer gluing theorem}, which is in some sense a converse to Gromov compactness.  However, we do not discuss gluing, as the argument is analytically rather technical and so goes beyond the scope of this article.

\section{Oszv\'ath and Szab\'o's Heegaard Floer homology} \label{section HF}  

In this section we explain the construction behind the `hat' flavour for {\it Heegaard Floer homology} $\hfhat$  building on the theory developed in the previous section.  To begin with, fix a connected, closed, oriented 3-manifold $Y$, and let $(\Si,\hal,\hbe,z)$ be a {\it pointed} Heegaard diagram of $Y$.  We define precisely such diagrams in Section \ref{subsection heegaard diagrams}, but for now it suffices to know that $\Si$ is a genus $g$ surface, $\hal,\hbe$ are appropriate $g$-tuples of curves on $\Si$, and the points $z \in \Si$ lies in the complement of those curves.  The manifold $Y$ is easily reconstructed from the Heegaard diagram, by thickening the surface to $\Si \times I$, and then by gluing $2$-handles along the $\al$-curves on $\Si \times \{0\}$ and along the $\be$-curves on $\Si \times \{1\}$.

Our aim is to explain how $\hfhat(Y)$ can be viewed as a special case of the Lagrangian Floer homology constructed in Section \ref{section Lag FH}.  Specifically, we associate to $Y$, via the diagram $(\Si,\hal,\hbe,z)$, a symplectic manifold $(M,\om)$ and two Lagrangian submanifolds $\bT_\al,\bT_\be \subset M$.  Just briefly: $M$ is the symmetric product of $g$ copies of $\Si$, and $\bT_\al$ is the direct product of all the $\al$ curves (similarly for $\bT_\be$).  Then we define 
\[
\hfhat(Y):=HF^0(M,\bT_\al,\bT_\be).
\]
The `$0$' refers to the fact that we only count homotopy classes $\phi$ satisfying a particular intersection number requirement $n_z(\phi)=0$; we explain this shortly (or see Definition \ref{deff nz}).  

\begin{figure}[h]
\hspace{1cm}
\xymatrix{
& \text{closed, oriented 3-manifold } Y \ar[d] \\
&  \text{Heegaard diagram } (\Si,\hal,\hbe)  \ar@/_/[dl] \ar@/^/[dr]
   \ar@{.>}[dd]|-{}            \\
 M:=\Sym^g \Si \ar@/_/[dr]  &  & L_1:=\bT_\al, L_2:=\bT_\be \ar@/^/[dl]     \\
     & HF(M,\bT_\al,\bT_\be) \ar[d]^{\text{invariance of choices}}          \\
		& \hfhat(Y)
	     }
\caption{Construction of Heegaard Floer homology.}
\label{fig construction of HF}
\end{figure}
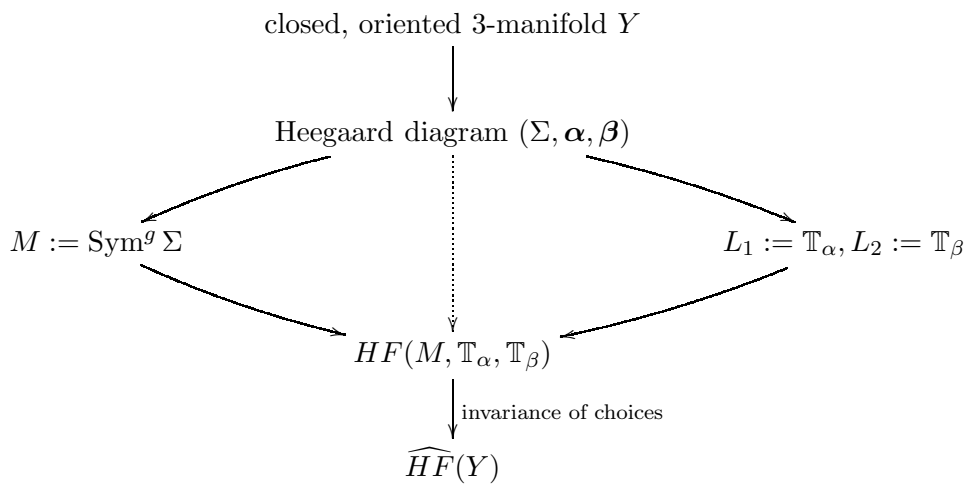

A very important aspect of Heegaard Floer homology is the decomposition of $\hfhat(Y)$ along $\Spin^c$ structures.  In this first instance, think of $\Spin^c$ structures as homotopy classes of nonsingular vector fields on $Y$ (Section \ref{subsection spinc}).  The $\Spin^c$ structure $\fs$ and the point $z \in \Si$ determine a subset
\[
Z(\fs,z) \subset \bT_\al \cap \bT_\be.
\]
This gives rise to a smaller chain complex $CF^\fs(M, \bT_\al,\bT_\be) \subset CF(M, \bT_\al,\bT_\be)$ defined to be
\[
CF^\fs(M, \bT_\al,\bT_\be):=\bigoplus_{x \in Z(\fs,z)} \bZ_2 \langle x \rangle.
\]
In fact, $CF^\fs(M, \bT_\al,\bT_\be)$ is preserved by $\bdd$, since $\pi_2(\hx,\hy) = \emptyset$ unless $\hx$ and $\hy$ both belong to same subset $Z(\fs,z)$(see Corollary \ref{cor same subset}).  Recall that $\pi_2(\hx,\hy)$ is the set of homotopy classes of disks connecting $\hx$ and $\hy$; see Definition \ref{deff of pi2}.  In the Heegaard Floer setting, the disks connecting $\hx$ and $\hy$ are also called {\it Whitney disks}. In any case, it makes sense to define $HF^{\fs,0}(M,\bT_\al,\bT_\be)$.  The complexes are independent of the choice of base point $z$, Heegaard diagram and any other variables that occur in the construction, so we are justified in writing
\[
\hfhat(Y,\fs):=HF^{\fs,0}(M,\bT_\al,\bT_\be).
\]
Then 
\[
\hfhat(Y)=\bigoplus_{\fs \in \Spin^c(Y)}\hfhat(Y,\fs).
\]
See Figure \ref{fig construction of HF} for a schematic summary of what we have said so far.

In order for $HF(M,\bT_\al,\bT_\be)$ to be well-defined, as explained in Section \ref{section Lag FH}, we need additional conditions on the Lagrangians $\bT_\al$ and $\bT_\be$, referred to as Conditions {\bf (F1--4)}.  We may assume that {\bf (F1)} holds since in our Heegaard diagram $(\Si,\hal,\hbe)$ each $\al$-curve is transverse to each $\be$-curve, which implies the transversality of $\bT_\al$ and $\bT_\be$.  From Remark \ref{rmk two things} we know that Condition {\bf (F1)} is sufficient to prove that the moduli spaces $\mM(\hx,\hy;\phi)$ are all finite-dimensional manifolds (of dimension $\mu(\phi)$), but that we also need to know two more things for the Floer homology $HF(\bT_\la,\bT_\be)$ to be well-defined.  (We drop the notation for $M$, as it is understood.)   For the convenience of the reader we repeat these two items here, slightly modified to fit the notation of this section.

\begin{enumerate}
\item There are at most finitely many classes $\phi \in \pi_2(\hx,\hy)$ such that $\mu(\phi)=1$ and such that $\mM(\hx,\hy;\phi)$ is non-empty.

\item If $\mu(\phi)=1$ or $\mu(\phi)=2$ and $(u_n)_{n \in \bN} \subseteq \mM(\hx,\hy;\phi)$, then there are no singular points (that is, no bubbling).
\end{enumerate}

Previously we used Conditions {\bf (F2--4)} to prove these  assertions.  Sadly, these conditions simply do not hold in the current setup, and as stated (i) and (ii) above are not true.  However, all is not lost.  A choice of a point $z \in \Si$ determines a certain codimension-2 submanifold $V_z \subset M$, which can be used to restrict the homotopy classes of pseudo-holomorphic curves that we count.  Specifically, given $\hx,\hy \in \bT_\al \cap \bT_\be$ and $\phi \in \pi_2(\hx,\hy)$, there is a well-defined notion of {\it algebraic intersection number} $n_z(\phi)$ of $\phi$ with $V_z$.  Define
\[
\pi_2^0(\hx,\hy):=\{\phi \in \pi_2(\hx,\hy) \mid n_z(\phi)=0\}.
\]  
We show that if we restrict to elements of $\pi_2^0(\hx,\hy)$, then both conditions (i) and (ii) stated above are true.  Most of the work goes into proving (i); the proof of (ii) is much simpler than the corresponding proof in Section \ref{section Lag FH}.  Namely, we see that there are {\it no} non-constant holomorphic spheres or disks with boundary in $\bT_\al$ (respectively, $\bT_\be$) that satisfy the condition $n_z=0$.  Thus, bubbling is automatically excluded (see Lemma \ref{lemma bubbling} and paragraph preceding it).

Before going any further we make two remarks concerning our treatment of Heegaard Floer homology.

\begin{rmk}
In this article we consider only the `hat' flavour of Heegaard Floer homology, and thus we always impose the condition $n_z=0$.  Dropping this condition leads to the more complicated construction of the `infinity' flavour $\hfi(Y)$ of Heegaard Floer homology, which does not occur in sutured Floer homology so we do not discuss it here other than in the following few sentences.  

Conditions (i) and (ii) above are not true without the assumption $n_z=0$.  To overcome the lack of finiteness, Ozsv\'ath and Szab\'o instead use a {\it Novikov ring of coefficients} to record the area $\mA_\om(\phi)$ of $\phi$.  The point is that while there may be infinitely many classes $\phi$ with $\mu(\phi)=1$ and $\mM(\hx,\hy;\phi) \neq \emptyset$, for any constant $E>0$ there are only finitely many classes with $\mA_\om(\phi)<E$.  Secondly, as far as Condition (ii) is concerned, it can be shown that whilst disk bubbles may appear when $\mu(\phi)=2$, there are always an {\it even} number of them (\cite[Thm.\,3.15]{OSmain}), which means that they do not affect the boundary operator.
\end{rmk}

\begin{rmk}
In Ozsv\'ath and Szab\'o's original construction of Heegaard Floer homology, they never equipped the manifold $M$ (defined precisely in Section \ref{subsection symmetric}) with a symplectic form.  This may seem odd to the reader, but note that the symplectic form $\om$ does not enter the $J$-holomorphic curve equation, and thus it is conceivable that it maybe done away with.  Actually, the symplectic form (and the fact that the submanifolds $\bT_\al$ and $\bT_\be$ are Lagrangians) are really only used to provide an energy bound on Floer trajectories (see Corollary \ref{cor energy bound}).

Whilst $M$ certainly does carry symplectic forms, there is in some sense no `natural' symplectic form on $M$ (see Section \ref{subsection symmetric} for more details).  Thus, instead Ozsv\'ath and Szab\'o chose to regard $M$ as an {\it orbifold} $N/G$ (that is, the quotient of a manifold $N$ by a finite group $G$), despite the fact that $M$ is actually smooth.  The point is that $N$ {\it has} a `natural' symplectic form, which allowed them to obtain energy bounds by lifting trajectories to $N$ and then applying Corollary \ref{cor energy bound} (this is the content of Lemma 3.5 \cite{OSmain}).

However, Perutz showed that it {\it is} possible to equip $M$ with a  symplectic form whose cohomology class is suitably `natural', which avoids lifting trajectories to $N$ \cite{P}.  At least as far as the construction of Heegaard Floer homology is concerned, this does not gain all that much (but see \cite{P} for nice results that do require this extra theory).  Nonetheless, it seems a better use of the setup of this article to introduce Heegaard Floer homology by making use of Perutz's result, since it allows us to view $\hfhat$ as a special case of Lagrangian Floer homology.
\end{rmk}

The remainder of this section is organised as follows.
\begin{itemize}
\item [\ref{subsection heegaard diagrams}] Definition of pointed Heegaard diagram $(\Si,\hal,\hbe,z)$.
\item [\ref{subsection symmetric}] Construction of symplectic manifold $(M,\om)$ and Lagrangians $\bT_\al,\bT_\be$; Perutz's result.  
\item [\ref{subsection spinc}] $\Spin^c$ structures and subset $Z(\fs,z)$ referred to above.
\item[\ref{subsection closed domains}] Study of domains, definition of $\pi_1^0(\hx,\hy)$, statement (without proof) of deep result that relates $\Spin^c$ structures and Maslov index $\mu(\phi)$ (Theorem \ref{thm deep}), and why  (i) holds (with the restriction to $\pi_2^0(\hx,\hy)$). 
\item [\ref{subsection bubbling}] Why (ii) holds (with restriction to $\pi_2^0(\hx,\hy)$). 
\item [\ref{subsection knot floer homology}] Brief look at knot Floer homology (a similar Floer theory, only for studying knots in 3-manifolds).
\end{itemize}

\subsection{Closed $3$-manifolds and Heegaard diagrams}   \label{subsection heegaard diagrams}

Let $Y$ be a closed oriented 3-manifold.  Then $Y$ admits a {\it Heegaard splitting}, that is, $Y=H_1\cup_{\bdd H_1 \cap H_2} H_2$ where $H_1$ and $H_2$ are two genus $g$ handlebodies, for some $g$.  One way to show this by taking a  triangulation of $Y$ and define $H_1$ to be the closure of a tubular neighbourhood of the 1-skeleton. The complement of $H_1$ in $Y$ is precisely another handlebody of the same genus.  

Another way to see that $Y$ admits a Heegaard splitting, is to consider a self-indexing Morse function $f \colon Y \to \bR$.  Recall: a smooth function $f \colon Y \to \bR$ is a {\it Morse function} if all its critical points are non-degenerate, that is, the Hessian of $f$ at a critical point $x$ is nonsingular.  The {\it index} of $x$ is the dimension of the maximal negative-definite subspace of the Hessian at $x$.  The Morse function $f$ is called {\it self-indexing} if for each critical point $x \in Y$ we have $f(x)=\ind (x)$.  See Section \ref{subsection Morse theory} for more notation and terminology of Morse theory. 

In fact, there is always a self-indexing Morse function $f$ on $Y$ with exactly one critical point $p_{min}$ of index 0, exactly one critical point $p_{max}$ of index 3, and $g$ critical points of index 1, and $g$ critical points of index 2, for some $g \in \bN$ (see \cite[Thm.\,4.8]{Milnor}).  If we take such a Morse function $f$, then we obtain a Heegaard splitting
\[
H_1:=f^{-1}([0,3/2]), \hspace{1cm} H_2:=f^{-1}([3/2,3]), \hspace{1cm} \Si:=f^{-1}(\{3/2\}).
\]

A Heegaard splitting gives rise to the concept of a {\it Heegaard diagram} $(\Si, \hal, \hbe)$ for $Y$. Here $\hal$ is a collection of pairwise disjoint, linearly independent curves  $\{\al_1, \ldots, \al_g\}$ on $\bdd H_1$ such that if we attach  3-dimensional 2-handles along those curves, we recover $H_1$.  Here $g$ is the genus of $\Si$.  Similarly, $\hbe$ is a collection of attaching circles on $\bdd H_2$ that allow us to recover $H_2$.  Now, if the Heegaard splitting was derived from a self-indexing Morse function on $Y$, then there is an easy way to describe the $\al$ and $\be$ curves.  Indeed, denote by $p_i$ and $q_i$ the critical points of index 1 and index 2, respectively.  Choose a Riemannian metric $\rho$ on $Y$ such that $(f,\rho)$ is Morse-Smale.  Then, recall that $\Si:=f^{-1}(\{3/2\})$, and set
\begin{equation} \label{eq alpha and beta curves}
\al_i:=W^s(p_i,-\nabla f) \cap \Si, \hspace{1cm} \be_i:=W^u(q_i, -\nabla f)\cap \Si.
\end{equation}
That is, $\al_i$ is the set of points in the intersection of $\Si$ and the flow lines of $-\nabla f$ flowing into $p_i$ (the stable manifold $W^s$).  Similarly, $\be_i$ is the set of points in the intersection of $\Si$ and the flow lines of $-\nabla f$ flowing out of $q_i$ (the unstable manifold $W^u$). In this case we say that $(f,\rho)$ is {\it compatible} with the Heegaard diagram $(\Si,\hal,\hbe)$.  In fact, every Heegaard decomposition can be realised in this way.

There are many Heegaard diagrams of $Y$, but it can be shown that two Heegaard diagrams define the same manifold if and only if they are related by three basic types of move called {\it isotopy, handle-slide,} and {\it (de)stabilisation}.  Let us explain the meanings of these three operations. Two sets $\hal$ and $\hal'$ of curves are {\it isotopic} if they are isotopic through a one-parameter family of curves, such that at every point the curves are pairwise disjoint.  (Similarly for $\hbe$ and $\hbe'$.)  Next, consider two disjoint curves $\al_1$ and $\al_2$ in $\Si$ that bound a pair of pants in $\Si$ together with a third curve $\ga$. Then a {\it handle-slide} of $\al_1$ over $\al_2$ results in $\ga$.  That is, a collection of curves $\hal: = (\al_1,\ldots, \al_g)$ is connected to $\hal':=(\al_1,\ldots, \al_{g-1}, \ga)$ via a handle-slide, if handle-sliding $\al_g$ over some other curve $\al_i$ results in $\ga$.  Lastly, we can {\it stabilise} a diagram $(\Si,\hal,\hbe)$ by taking the connected sum of $\Si$ and a torus with a meridian $\al$ and a longitude $\be$, thereby increasing the genus of $\Si$ and the number of $\hal$ and $\hbe$ curves by one.  

Finally, a {\it a pointed Heegaard diagram} $(\Si,\hal,\hbe,z)$ of $Y$ is a Heegaard diagram $(\Si,\hal,\hbe)$ of $Y$ together with a choice of base point $z \in \Si - \left\{ \bigcup_i \al_i \cup \bigcup_j \be_j\right\}$.

\subsection{Symmetric products}\label{subsection symmetric}

Recall that a {\it K\"ahler form} on a even-dimensional manifold $M$ is a 2-form $\om$ which is symplectic, and such that there exists an integrable almost complex structure $j$ compatible with $\om$.  In particular, if $M$ admits a K\"ahler form $\om$, then $M$ is both a symplectic and a complex manifold.  The reason that we now care about integrable almost complex structures lies in the proof of the non-negativity result of Corollary \ref{cor finite set}.

Let $\Si$ be a surface of genus $g$.  For simplicity assume that $g\geq 3$; we make this assumption because the cases $g=1,2$ are slightly different (but not harder) and would make our exposition messier (for example, see above Equation \eqref{eq c1 of sphere} for a statement that would not hold).

\begin{rmk}
{\it A posteriori}, since the Heegaard Floer homology depends only on $Y$ and not on the Heegaard decomposition $(\Si,\hal,\hbe)$, the genus of $\Si$ can be increased simply by stabilising it, as described in Section \ref{subsection heegaard diagrams}, to deal with the cases when $g=1,2$.  Of course, this can only be done {\it after} it has already been shown that Heegaard Floer homology is well-defined for all values of $g$.  However, it may reassure the reader that, at least as the final invariant is concerned, it is not unreasonable to restrict our exposition to the case when $g \geq 3$.  However, we do not discuss the independence of Heegaard Floer homology from the Heegaard decomposition, and refer the reader to Oszv\'ath and Szab\'o's paper for details on this topic \cite{OSmain}.
\end{rmk}

There is an abundance of K\"ahler structures on $\Si$.  Indeed, {\it let us now fix once and for all a complex structure $j$ on $\Si$}, such that $(\Si, j)$ is a 1-dimensional complex manifold.  Then any Riemannian metric $\langle \cdot, \cdot \rangle  $ on $\Si$ gives rise to a K\"ahler form $\eta(\cdot, \cdot):=-\langle j \cdot, \cdot \rangle$ on $\Si$.  Indeed, clearly $\eta$ is non-degenerate, and  and $d\eta=0$ since $\eta \in \om^{\mbox{top}}(\Si)$.   Thus, $\eta$ is a symplectic form.

Fix a particular such K\"ahler form $\eta$.  Let $\Si^g$ be the $2g$ dimensional manifold $\underbrace{\Si \times \cdots \times \Si}_{g \text{ copies}}$.  Then $\Si^g$ is itself K\"ahler and  $\Om := \eta \times \cdots \times \eta$ is a  K\"ahler structure on $\Si^g$, which is compatible with the integrable complex structure $\bJ:=j \times \cdots \times j$.

Let $\fS_l$ denote the symmetric group on $l$ points, for some $l \in \bN$.  There is a natural action of $\fS_g$ on $\Si^g$ that permutes the coordinates.  We denote the quotient space under this action by $M=\Sym^g\Si$, referred to as the {\it $g$th symmetric product of $\Si$}. Denote by $\pi \colon \Si^g \to M$ the quotient map.  Conveniently, $M$ is a smooth manifold of dimension $2g$.  Indeed, each point of $M$ can be identified with an unordered $g$-tuple of complex numbers, so locally $M$ can be seen as an open subset of the space of monic polynomials of degree $g$ over $\bC$ (by sending a polynomial to its zero set).  But the space of these polynomials can be identified with $\bC^g$, and hance $M$ locally looks like $\bC^g$.  In fact, the same argument shows that if $S$ is any surface and $l$ is any positive integer, the space  $\underbrace{\Si \times \cdots \Si}_{l \text{ times}}/\fS_l$ is also a manifold. It can be shown that the smooth structure on $M$ depends on our original choice of $j$, but its diffeomorphism type does not.  We refer the reader to the original paper \cite{Mac} for the proof of this result.

Let $\pi \colon \Si^g \to M$ denote the quotient map.  Let 
\[
E :=\{(z_1,\ldots,z_g) \in \Si^g\mid z_i=z_k \text{ for some } i \neq k\},
\] 
and let $D:=\pi(E)$.  Then the restriction of $\pi$ to $\Si^g-E$ is a covering space over $M-D$ of rank $g!$, and hence $\pi \colon \Si^g \to M$ is a branched cover.

The following lemma is proved in Section 2 of \cite{OSmain}; we omit the proof which is essentially a piece of algebraic topology.

\begin{lemma} \label{lemma homotopy is homology}
The group $\pi_1(M)$ is abelian, and hence $\pi_1(M)\cong H_1(M)$.  Moreover, $H_1(M) \cong H_1(\Si)$ where an isomorphism $H_1(M) \to H_1(\Si)$ is given by the map
\[
[\ga] \to \sum_{j=1}^g [\ga_j].
\]
Here $\ga \colon S^1 \to M$ is any continuous map, which after a perturbation within its homotopy class, may be assumed not to intersect the diagonal $D$.  So $\ga$ can be thought of as map into $\Si^g$, or equivalently, as a collection of maps $\ga_j\colon S^1\to \Si$ for $1 \leq j \leq g$.
\end{lemma}

There is a well-defined 2-form $\tilde \om$ on $M-D$ defined by 
\[
\tilde \om:=\pi_*(\Om)|_{M - D}.
\]
In fact, $\tilde \om$ is a K\"ahler form on $M-D$.  To see this, note that there is a well-defined integrable almost complex structure $\fj$ on $M$ that is defined by the requirement that $\pi \colon (\Si^g,\bJ) \to (M,\fj)$ is $(\bJ,\fj)$-holomorphic, and it is easy to see that $\fj|_{M-D}$ is compatible with $\tilde \om$.  We emphasise that $\fj$ is well-defined on all of $M$, whereas $\tilde \om$ is only defined on $M-D$.  However, it is easy to see that $\tilde \om$ does still determine a well-defined cohomology class $[\tilde \om] \in H^2(M;\bR)$.

Now fix any small neighbourhood $N(D)$ of $D$ in $M$.  A result due to Perutz \cite[Sec.\,7]{P} (that builds on earlier work of Varouchas) tells us that there exists a 2-form $\om \in \Om^2(M)$ with the properties that:
\begin{enumerate}
\item $\om = \tilde \om $ on $M - N(D)$,
\item $[\om]=[\tilde \om]$ in $H^2(M;\bR)$,
\item $\om$ is a K\"ahler form on all of $M$.
\end{enumerate}
In words, we can produce from $\tilde\om$ a new K\"ahler form $\om$ on $M$ that agrees with $\tilde \om$ on $M-D$ (and so is compatible with $\fj$ on $M-D$), and that determines the same cohomology class as $\tilde \om$ in $H^2(M;\bR)$.

Suppose now that $\hal:=(\al_1,\cdots, \al_g)$ and $\hbe:=(\be_1, \cdots, \be_g)$ are two $g$-tuples of embedded circles in $\Si$ such that the curves in each set are pairwise disjoint.  Then they determine two embedded tori  $\al_1 \times \ldots \times \al_g$ and $\be_1 \times \ldots \times \be_g$ in $\Si^g$ (which are Lagrangian with respect to $\Om)$, and are both contain in $\Si^g-E$, since the curves in each of the two sets $\hal$ and $\hbe$ are pairwise disjoint.  

Define $\bT_\al:=\pi(\al_1 \times \ldots \times \al_g)$ and $\bT_\be:=\pi(\be_1 \times \ldots \times \be_g)$, so that $\bT_\al$ and $\bT_\be$ are embedded tori in $M=\Si^g/\fS_g$.  We can always choose a neighbourhood $N(D)$ of $D$, so that 
\[
\bT_\al \cap N(D)=\emptyset, \hspace{1cm} \bT_\be \cap N(D)=\emptyset.
\]
Then $\bT_\al$ and $\bT_\be$ are Lagrangian submanifolds of $(M,\om)$, where $\om$ is  Perutz's K\"ahler form.

\subsection{$\Spin^c$ structures} \label{subsection spinc}

In this section we give a precise definition of the set $\Spin^c$ structures on $Y$, denoted by $\Spin^c(Y)$.  There are various ways to define them, but for us the most useful is the topological interpretation of $\Spin^c$ structures as nonsingular vector fields on $Y$ which is due to Turaev \cite{Tu90}.  

 Two nowhere vanishing vector fields $v_1$ and $v_2$ are said to be {\it homologous}, denoted by $v_1 \sim v_2$, if there is a ball $B$ in $Y$ with the property that $v_1|_{Y-B}$ is homotopic to $v_2|_{Y-B}$.    This is an equivalence relation, and $\Spin^c(Y)$ is defined to be the set of equivalence classes of nonsingular vector fields on $Y$.

There is a bijective correspondence of $\Spin^c(Y)$ and $H^2(Y) \cong H_1(Y)$.  Indeed, since $Y$ is 3-dimensional, we can take a trivialisation $\varphi \colon TY \to Y \times \bR^3$.  Then we can define a map $f_{\varphi,v} \colon Y \to S^2$, by setting $f_{\varphi, v}(p):=\frac{\varphi(v_p)}{\norm{\varphi(v_p)}}$, which is well-defined since $v$ is everywhere nonsingular.  If $v_1 \sim v_2$, then the maps $f_{\varphi,v_1}$ and $f_{\varphi,v_2}$ are homotopic on $Y-B$.  

Denote by $\Pi(Y,S^2)$ the set of equivalence classes of maps $Y \to S^2$ that are homotopic away from some ball $B$.  Then standard obstruction theory shows that $\Pi(Y,S^2)$ is classified by $H^2(Y)$ via the map
\[
f \mapsto f^*(\mu),
\]
where $f \in \Pi(Y,S^2)$, the map $f^*$ is the standard map on cohomology $H^2(S^2) \to H^2(Y)$ derived from $f$, and $\mu := PD\circ[S^2]$ (see for example \cite[Chap.\,7]{DK}).  Set $\theta_\varphi(v):=f^*(\mu)$. Note that any two trivialisations $\varphi$ and $\varphi'$ differ by a map $\psi \colon Y \to SO(3)$, which means that if $\varphi(v_p)=(p,w)$, then $\varphi'(v_p)=(p,\psi(p)\cdot w)$.  So we have
\[
f^*_{\varphi,v}(\mu) - f^*_{\varphi',v}(\mu) =\psi^*(\kappa),
\]
where $\kappa$ is the generator of $H^2(SO(3))\cong \bZ_2$.  Now we have that 
\[
\theta_\varphi(v_1) - \theta_\varphi(v_2)= \left(\theta_{\varphi'}(v_1)+\psi^*(\kappa)\right)-
\left(\theta_{\varphi'}(v_2)+\psi^*(\kappa)\right)=\theta_{\varphi'}(v_1) - \theta_{\varphi'}(v_2),
\]
and thus the difference $\theta(v_1,v_2):=\theta_\varphi(v_1) - \theta_\varphi(v_2)$ is always independent of the trivialisation.  Instead of $\theta(v_1,v_2)$ we write $v_1 - v_2$.

Hence, there is always a free, transitive action of $H^2(Y)$ on $\Spin^c(Y)$.  If $a \in H^2(Y)$, then $a+v \in \Spin^c(Y)$ is defined to be the vector field such that $\theta(a+v,v)=a$.

\begin{rmk}
If $H^2(Y)$ has no 2-torsion, then $\theta_\varphi(v)$ is {\it always} independent of the choice of $\varphi$.  Indeed,
\begin{align*}
2 (\theta_\varphi(v)-\theta_{\varphi'}(v)) &= (\theta_\varphi(v)-\theta_\varphi(-v))-(\theta_{\varphi'}(v)-\theta_{\varphi'}(-v)) \\
&=\theta(v,-v) - \theta(v,-v)\\
& = 0.
\end{align*}
\end{rmk}

Suppose $\hx,\hy \in \bT_\al \cap \bT_\be$.  Given two paths, one from $\hx$ to $\hy$ in $\bT_\al$ and the other from $\hy$ to $\hx$ in $\bT_\be$, we can concatenate them to form a loop $\ga \colon S^1 \to M$.   Clearly if $[\ga]\neq 0 \in \pi_1(M)$ for every $\ga$, then $\pi_2(\hx,\hy)=\emptyset$.   Thus we are led to consider the well-defined homology class in $H_1(M)$ obtained from $[\ga]$:
\[
\tilde \varep (\hx,\hy) \in \frac{H_1(M)} {H_1(\bT_\al) \oplus H_1(\bT_\be)}.
\]
Under the isomorphism $H_1(M)\cong H_1(\Si)$, from Lemma \ref{lemma homotopy is homology}, $\tilde \ep (\hx,\hy)$ determines the class
\[
\varep(\hx,\hy) \in \frac{H_1(\Si)}{[\al_1],\cdots, [\al_g], [\be_1],\cdots,[\be_g]} \cong H_1(Y).
\]
Now, because Lemma \ref{lemma homotopy is homology} says that $\pi_1(M)\cong H_1(M)$, we have that $\ga$ can be extend to a map $\bD \to M$ if and only if $\tilde \varep(\hx,\hy)=0$, and hence if and only if $\varep(\hx,\hy)=0$.  So we have the following corollary.

\begin{cor} \label{cor pi2 and ep}
The set $\pi_2(\hx,\hy)$ is non-empty if and only if $\varep(\hx,\hy)=0 \in H_2(Y)$.
\end{cor}

There is an alternative way to describe the class $\varep(\hx,\hy)$ that proves to be useful in what follows.  Recall from Section \ref{subsection heegaard diagrams} that we can choose a Morse function $f$ and a Riemannian metric $\rho$ on $\Si$ compatible with the Heegaard diagram $(\Si,\hal,\hbe)$.  We use the notation from Equation \eqref{eq alpha and beta curves}.  Set
\[
\ga_{jk}:=W^u(q_j,-\nabla f) \cap W^s(p_k,-\nabla f),
\]
so that $\ga_{jk}$ is a flow line in $Y$, whose closure $\bar \ga_{jk}$ may be thought of as a map $\bar \ga_{jk} \colon [0,1] \to Y$ with $\bar \ga_{jk}(0)=q_j$ and $\bar \ga_{jk}(1)=p_k$.  

Set $x_{jk}:=\ga_{jk} \cap \Si$.  Then an intersection point $\hx \in \bT_\al \cap \bT_\be$ is just a collection $\{x_{j_1k_1}, \cdots, x_{j_gk_g}\}$ of points $x_{j_i,k_i}$.  Thus, given $\hx=\{x_{j_1k_1}, \ldots, x_{j_gk_g}\}$ and $\hy=\{x_{j'_1k'_1}, \ldots, x_{j'_gk'_g}\}$, we can consider the 1-cycle in $Y$ given by
\[
\bar \ga_{j_1k_1} + \cdots +\bar \ga_{j_gk_g}-\bar \ga_{j'_1k'_1} - \cdots - \bar \ga_{j'_gk'_g}.
\]
This determine a homology class in $H_1(Y)$ that is seen to be $\varep(\hx,\hy)$.

This second way of defining $\varep(\hx,\hy)$ makes the relationship between $\varep(\hx,\hy)$ and $\Spin^c(Y)$ particularly clear.  Suppose $\hx \in \bT_\al \cap \bT_\be$ with $\ga_{j_ik_i}$ as above.  Take tubular neighbourhoods $N(\bar \ga_{j_ik_i})$ of each $\bar \ga_{j_ik_i}$.  Fix a flow line $\ga_0$ of $-\nabla f$ from $p_{min}$ to $p_{max}$ passing through the base point $z$, and let $N(\bar \ga_0)$ denote the tubular neighbourhood of $\bar \ga_0$.  Note that each of these tubular neighbourhoods is homeomorphic to a ball.  Set
\[
B:=N(\bar \ga_{j_1k_1}) \cup \cdots \cup N(\bar \ga_{j_gk_g}) \cup N(\bar \ga_0).
\]
Note that $- \nabla f$ does not vanish on $Y-B$, and in particular on $\bdd B$.  Moreover, by the Poincar\'e-Hopf theorem, the map $-\nabla f|_{\bdd B}$ extends to a non-vanishing vector field on all of $B$ since the sum of the indices of the zeros of $-\nabla f$ in $B$ is zero.  Indeed, each $N(\bar \ga_{j_ik_i})$ and $N(\bar \ga_0)$ connect critical points of $f$ that have different parity.  After performing this extension, we obtain a non-vanishing unit vector field $v$ on $Y$, which defines a $\Spin^c$ structure $\fs$ on $Y$.  Thus, there is a well-defined map 
\[
s_z \colon \bT_\al \cap \bT_\be \to \Spin^c(Y),
\]
given by $s_z(\hx):=\fs$, where $\fs=[v]$ and $v$ is the vector filed obtained as explained above.  For some $\fs \in \Spin^c(Y)$, we  denote by $Z(\fs,z)$ the set of intersection points $\hx \in \bT_\al \cap \bT_\be$ with $s_z(\hx)=\fs$.  

The following result is essentially clear from the second definition of $\varep(\hx,\hy)$ above; we refer the reader to Lemma 2.19 \cite{OSmain} for a detailed proof.

\begin{lemma} \label{lemma pi2 and ep}
Given two points $\hx,\hy \in \bT_\al \cap \bT_\be$, we have that $s_z(\hx)=s_z(\hy)$ if and only if $\varep(\hx,\hy)=0$.  Thus if $\hx \in Z(\fs,z)$, then 
\[
Z(\fs,z)=\{\hy \in \bT_\al \cap \bT_\be \mid \varep(\hx,\hy) =0\}.
\]
\end{lemma}

The following corollary is immediate from Corollary \ref{cor pi2 and ep} and Lemma \ref{lemma pi2 and ep}.

\begin{cor} \label{cor same subset}
Given two points $\hx,\hy \in \bT_\al \cap \bT_\be$, we have that $\pi_2(\hx,\hy) \neq \emptyset $ if and only if $s_z(\hx)=s_z(\hy)$.
\end{cor}

There is also a well-defined map $c_1 \colon \Spin^c(Y) \to H^2(Y)$.  Indeed, if $v$ is a nonsingular vector field, then the orthogonal complement $v^\perp$ can be regarded as a complex line bundle over $Y$, and thus the first Chern class $c_1(v^\perp)$ is a well-defined element of $H^2(Y)$.  It is easy to check that $v_1 \sim v_2$ implies that $v_1^\perp$ and $v_2^\perp$ are isomorphic as complex line bundles and hence $c_1(v_1^\perp)=c_1(v_2^\perp)$.

\subsection{Domains and $\pi_2(\hx,\hx)$} \label{subsection closed domains}

Let $(\Si, \hal,\hbe,z)$ be a pointed Heegaard diagram of a closed three-manifold $Y$, and as usual we assume that $\Si$ has genus $g \geq 3$.  Let $M=\Sym^g(\Si)$ as above.   Instead of studying the Whitney disks connecting points in $\bT_\al \cap \bT_\be$ in $M$, it can be helpful to study their ``shadows'' in $\Si_g$.  In other words, we can associate to each $\phi \in \pi_2(\hx,\hy)$ a subsurface of $\Si_g$ called a {\it domain} $\mD(\phi)$.  

Define the submanifold $V_z \subset M$ of codimension-2 to be
\[
V_z:=\pi(\{z\} \times \underbrace{\Si \times \cdots \times \Si}_{g-1 \text{ copies}}).
\]
Recall that $\pi$ is the projection map $\Si^g \to M$.

\begin{deff} \label{deff nz}
Given $\hx,\hy \in \bT_\al \cap \bT_\be$ and $\phi \in \pi_1(\hx,\hy)$, fix a map $v \colon \bD \to M$ that represents $\phi$ and is transverse to $V_z$.  Then define $n_z(\phi)$ to be the algebraic intersection number of $v(\bD) \cap V_z$.  Note that $n_z(\phi)$ is well-defined as it does not depend on the choice of $v$ representing $\phi$ as long as $v \pitchfork V_z$.
\end{deff}

Now for some more algebraic topology.  It can be shown, assuming $g \geq 3$, that $\pi_2(M)=\bZ$, and $\pi_1(M)$ acts trivially on $\pi_2(M)$ \cite[Prop.\,2.7]{OSmain}.  Moreover, there is an explicit generator $\si \colon S^2 \to M$ of $\pi_2(M)$ with the property that 
\begin{equation} \label{eq c1 of sphere}
n_z(\si)=1, \hspace{1cm} c_1(\si_*([S^2]))=1.
\end{equation}
See \cite[Lem.\,2.8]{OSmain} and \cite{Mac}.  Since $\pi_1(M)$ acts trivially on $\pi_2(M)$, by slight abuse of notation for any given $\hx \in \bT_\al \cap \bT_\be$, we can regard $\si$ as defining an element of $\pi_2(\hx,\hx)$.  Then we have the following lemma (the proof is just algebraic topology).

\begin{lemma} \cite[Prop.\,2.15]{OSmain}  \label{lemma pix} Given any $\phi \in \pi_2(\hx,\hx)$, we have 
\[
\phi = k \si \# \phi_0,
\]
for some $k \in \bZ$ and some $\phi_0 \in \pi_2(\hx,\hx)$ with $n_z(\phi_0)=0$.  Therefore, if we set
\[
\pi_2^0(\hx,\hx):=\{\phi \in \pi_2(\hx,\hx) \mid n_z(\phi)=0\}, 
\]
then 
\[
\pi_2(\hx,\hx)\cong \bZ \oplus \pi_2^0(\hx,\hx).
\]
\end{lemma}

In fact, we can identify $\pi_2^0(\hx,\hx)$ with $H_2(Y)$, and in order to explain this we introduce the notion of a {\it domain}.  For every $\hx \in \bT_\al \cap \bT_\be$, a class $\phi \in \pi_2^0(\hx,\hx)$ is called a {\it periodic class}.

\begin{deff}  \label{deff domains1}
Let $(\Si,\hal,\hbe,z)$ be a given Heegaard diagram of $Y$.
Let $D_1, \ldots, D_m$ denote the closures of the components of $\Si - \cup_{i=1}^g \al_i - \cup_{j=1}^g \be_j$, {\it indexed} so that $z \in D_1$, and denote by $D(\Si,\hal,\hbe)$ the free abelian group generated by these components and isomorphic to $\bZ^m$.  Thus if $\mD \in D(\Si,\hal,\hbe)$, then 
\[
\mD=\sum_{i=1}^m p_i D_i, \hspace{1cm} \text{for some } p_i \in \bZ.
\]
We define
\[
n_i(\mD):=p_i.
\]
We say that $\mD \geq 0$ if $n_i(\mD) \geq 0$ for each $i$, and given two domains $\mD, \mD'$ we say $\mD \geq \mD$ if $\mD-\mD' \geq 0$.
\end{deff}

Next, given a homotopy class $\phi \in \pi_2(\hx,\hy)$ we can define another notion of domain based on the intersection number $n_z$.

\begin{deff} Let $(\Si,\hal,\hbe,z)$ be a given Heegaard diagram of $Y$.  Let $D(\Si,\hal,\hbe)$ be the set of domains isomorphic to $\bZ^m$ from Definition \ref{deff domains1}. Set $z_1:=z$ and for each $k \in \{2,\ldots, m\}$ choose an arbitrary point $z_k$ in the interior of $D_k$.  Then given $\hx,\hy \in \bT_\al \cap \bT_\be$ and $\phi \in \pi_2(\hx,\hy)$, define 
\[
\mD(\phi):=\sum_{k=1}^m n_{z_k}(\phi) D_i.
\]
\end{deff}

Clearly, $\mD(\phi)$ is well-defined as it is independent of the choice of the points $z_k$.

Note that if $\phi \in \pi_2^0(\hx,\hx)$, then $\bdd \mD(\phi)$ is a union of $\al$ and $\be$ curves:
\[
\bdd \mD(\phi)=\sum a_i \al_i + \sum b_j \be_j,
\]
for some integer coefficients $a_i,b_j$.  In general, any domain $\mD=\sum_{i=1}^m p_i D_i$, such that its boundary is a union of $\al$ and $\be$ curves, is called a {\it periodic domain}. To any periodic domain $\mD(\phi)$ we can associate a homology class $\mH(\phi) \in H_2(Y)$ given by
\[
\mH(\phi):=\mD(\phi)+\sum a_i A_i + \sum b_j B_j,
\]
where $A_i$ and $B_j$ are the cores of the two handles attached to $\al_i$ and $\be_j$, respectively.  As a result we have, in addition to Lemma \ref{lemma pix}, the following statement.

\begin{lemma} \cite[Prop.\,2.15]{OSmain}  \label{lemma pix2}
The map $\pi_2^0(\hx,\hx) \to H_2(Y)$ given by $\phi \mapsto \mH(\phi)$ is an isomorphism.
\end{lemma}

We are now ready to state a deep result of Ozsv\'ath and Szab\'o about the relation of the Maslov index of $\phi \in \pi_2^0(\hx,\hx)$, the $\Spin^c$ structure $s_z(\hx)$, and $\mH(\phi)$.  The proof goes beyond the scope of this introductory article.

\begin{thm} \cite[Thm.\,4.9]{OSmain} \label{thm deep}
Let $(\Si,\hal,\hbe,z)$ be a pointed Heegaard diagram for $Y$.  Then for any $\hx \in \bT_\al \cap \bT_\be$, and any $\phi \in \pi_2^0(\hx,\hx)$, we have 
\begin{equation} \label{eq deep}
\mu(\phi)=\langle c_1(\fs_z(\phi)), \mH(\phi)\rangle.
\end{equation}
\end{thm}

Next, we introduce the definition of an {\it admissible Heegaard diagram}\footnote{The definition we give for admissible is what Ozsv\'ath and Szab\'o call {\it weakly admissible} in Definition 4.10 of \cite{OSmain}).}; in Section 5 of  \cite{OSmain} it is shown that there always exists an admissible diagram for $Y$.  

\begin{deff}
A pointed Heegaard diagram $(\Si,\hal,\hbe,z)$ is called {\it admissible} if for every point $\hx \in \bT_\al \cap \bT_\be$ and every class $\phi \in \pi_2^0(\hx,\hx)$ with $\mu(\phi)=0$, the corresponding (periodic) domain $D(\phi)$ has both positive and negative coefficients.  
\end{deff}

We can now prove the main result of this section.

\begin{thm} \label{thm finite set}
Suppose $(\Si,\hal,\hbe,z)$ is an admissible pointed Heegaard diagram.  Then for any $\fs \in \Spin^c(Y)$, any two points $\hx,\hy \in Z(\fs,z)$, and any pair of integers $r,l$, the set
\begin{equation}
K:=\{ \phi \in \pi_2(\hx,\hy) \mid n_z(\phi)=r,  \mu(\phi)=l, \mD(\phi) \geq 0\}
\end{equation}
is finite.
\end{thm}

\begin{proof}
If $K \neq \emptyset$, then fix $\psi \in K$.  Given any other $\phi \in K$, since $\phi,\psi \in \pi_2(\hx,\hy)$, we can write $\phi=\phi' \# \psi$, for some $\phi' \in \pi_2(\hx,\hx)$.  From Lemma \ref{lemma pix}, we know that 
\[
\phi'=k \si \# \phi_0,
\]
 for some $k\in\bZ$ and $\phi_1 \in \pi_2^0(\hx,\hx)$.  By additivity of the Maslov index (see Equation \eqref{additivity of Maslov}), we have $\mu(\phi')=\mu(\phi)-\mu(\psi)=0$. Next consider the computation 
 \begin{align*}
 0 &= \mu(\phi') \\
 & = \mu(\phi_0) + k \mu(\si) \\
 & = \langle c_1(\fs), \mH(\phi_0)\rangle + 2 k c_1(\si_*([S^2])),
 \end{align*}
 where the third inequality uses Equations \eqref{eq division by two} and \eqref{eq deep}\footnote{There are two different $c_1$'s in the last line of the computation!  The first one, $c_1(\fs)$, lives in $H^2(Y)$, and the second one, which is really $c_1(TM,\fj)$, lives in $H^2(M)$.}.  Plugging in $c_1(\si_*([S^2]))=1$ from Equation \eqref{eq c1 of sphere}, we have 
 \[
 k=-\frac{1}{2} \langle c_1(\fs),\mH(\phi_0)\rangle.
 \]
From Equation \eqref{eq c1 of sphere} we also know that $n_z(\si)=1$, which means that since $n_z(\phi)=n_z(\psi)$, we have $n_z(\phi')=0$ and so $k=0$.  Thus $\phi'=\phi_0$ and
\[
\phi=\phi_0 \# \psi, 
\]
where $\phi_0 \in \pi_2^0(\hx,\hx)$ satisfies
\[
\langle c_1(\fs),\mH(\phi_0)\rangle=0,
\]
or equivalently $\mu(\phi_0)=0$ by Equation \eqref{eq deep}. 

Next, since $\mD(\phi)\geq 0$ and $\mD(\psi) \geq 0$, then $\mD(\phi_0) \geq -\mD(\psi)$.  Recall that $\psi$ was an arbitrary fixed element.  So we may identify the set $K$ with the following set
\[
Q:=\{\phi_0 \in \pi_2^0(\hx,\hx) \mid \mu(\phi_0)=0, \mD(\phi_0) \geq \mD(\psi) \}.
\]
To complete the proof we show that $Q$ is finite.

Suppose that $Q$ is not finite.  Consider $Q$ as a lattice in $\bZ^m$, which in turn lives in $\bR^m$.  Denote by $\bar Q$ the infinite polytope defined by $Q$ in $\bR^m$.  Then since $Q$ is infinite, there exists a sequence of points $(q_j)$ with $\norm {q_j} \to \infty$.  Now consider the compact space given as the intersection of the unit sphere $S^{m-1}$ and  $\bar Q$; the sequence $({q_j}/{\norm {q_j}})$ is a subset of this compact space and therefore has a subsequence that converges to some point $q$.  However, as the coefficients of $q_j$ were bounded bellow, and $\norm {q_j} \to \infty$ it follows that $q$ has all non-negative coefficients (although probably non-rational).  As we were working in a convex polytope, there must exists a point corresponding to a periodic domain with rational coefficients  that are also non-negative (otherwise $q$ would be outside of $\bar Q$).  Clearing denominators we find a point with non-negative integer coordinates, corresponding to a domain $\phi \in \pi_2^0(\hx,\hx)$ such that $\mu(\phi)=0$, hence contradicting our assumption that the diagram is admissible.

\end{proof}

\begin{cor} \label{cor finite set}
For any $\hx,\hy \in \bT_\al \cap \bT_\be$, there are at most finitely many classes $\phi \in \pi_2(\hx,\hy)$, such that $\mu(\phi)=1$ and such that $\mM(\hx,\hy;\phi)$ is non-empty. 
\end{cor}

\begin{proof}
This is an immediate consequence of Theorem \ref{thm finite set}, together with the following claim: if $\mM(\hx,\hy;\phi)\neq \emptyset$, then $\mD(\phi)\geq 0$.  Indeed, note that since $\fj$ is an {\it integrable} almost complex structure, any $\fj$-holomorphic curve $v \colon \bD \to M$ has as its image a complex manifold $v(\bD)$ in $M$.  The same is true of each submanifold $V_z$, and it is a general fact that if two complex submanifolds intersect transversely, then their algebraic intersection number is always non-negative.
\end{proof}

\begin{rmk}
There is actually a complication in the proof of Corollary \ref{cor finite set} that we have sidestepped.  As the reader may recall from the Section \ref{section Lag FH}, we actually need to perturb $\fj$ to obtain transversality of moduli space $\mM(\hx,\hy;\phi)$.  In general, such a perturbation results in a non-integrable almost complex structure, rendering the argument in our proof invalid.  However, Ozsv\'ath and Szab\'o work very hard in Section 3.3 \cite{OSmain} to show that it is possible to obtain transversality for the moduli spaces $\mM(\hx,\hy;\phi)$ by only perturbing the almost complex structure away from the submanifolds $V_{z_i}$ for $1 \leq i \leq m$ (where, recall, $m$ is the number of domains forming the basis of $\mD(\Si,\hal,\hbe)$).  This means that in a neighbourhood of each $V_{z_i}$, we still work with an integrable almost complex structure, and hence the preceding proof still works. 
\end{rmk}

\begin{rmk}\label{rmk on Z2} Note that Theorem \ref{thm deep} allows us to conclude more about the grading on $\hfhat$ than the material from Section \ref{section Lag FH}. Indeed, it follows from Equation \eqref{eq deep} (see also \cite[Lemma 3.3]{OSmain}) that Heegaard tori have minimal Maslov number 2. Thus, our constructions from Section \ref{section Lag FH} would only give $\hfhat$ a $\bZ_2$ grading, that is,
\[
\gr \colon  Z(\fs,z) \times Z(\fs,z) \to \bZ_2.
\]
However, Theorem \ref{thm deep} shows that in fact $\hfhat$ can be graded modulo $\fd(\fs)$, where $\fd(\fs)$ is given by
\[
\fd(\fs):= \gcd_{\xi \in H_2(Y)} \langle c_1(\fs),\xi \rangle.
\]
In general, being able to grade modulo $\fd(\fs)$ is  a stronger statement. For instance, if $Y$ is a homology sphere, then  \ref{thm deep} shows that we get a $\bZ$ grading.

  In other words, for some $\hx,\hy \in Z(\fs,z)$, and $\phi,\psi \in \pi_2(\hx,\hy)$, we have \linebreak $\dim\mM(\hx,\hy;\phi)=\mu(\phi)$ and $\dim\mM(\hx,\hy;\psi)=\mu(\psi)$, and 
\[
\mu(\phi)=\mu(\psi) \mod \fd(\fs).
\]

\end{rmk}

 \subsection{Bubbling} \label{subsection bubbling}
 
In this short section we explain how bubbling off of holomorphic spheres and disks is precluded as we only work with classes $\phi \in \pi_2^0(\hx,\hy)$.  We use the notation from Theorem \ref{thm gromov compactness II} (Gromov compactness II).  If a sequence $(v_n)_{n \in \bN} \subset \mM(\hx,\hy ;\phi)$ has a singular set $\De \subset \bD$, so that (up to subsequence), $v_n \to w$ on $\bD - \De$, and a holomorphic sphere or disk $v_p$ appears at each point $p \in \De$, then for large $n$, we have
\[
n_z(v_n)\geq n_z(w) + \sum_{p \in \De} n_z(v_p).
\]
As in Equations \eqref{eq energy of bubbles} and \eqref{equation on maslov indices}, we have an inequality (and not equality) because it is possible that additional bubbles can appear on the bubbles (forming so called `bubble trees').   Moreover, using the same argument as in the proof of Corollary \ref{cor finite set}, we see that the algebraic intersection numbers $n_z$ are always non-negative.  Thus, if a bubble exists, it must have $n_z=0$.  We show that this is impossible.

In the case of holomorphic spheres this is immediate from the fact that $\pi_2(M)\cong \bZ$ is generated by a map $\si \colon S^2 \to M$ with $n_z(\si)=1$ (see Equation \eqref{eq c1 of sphere}). Therefore, there are no spheres, let alone holomorphic spheres, with $n_z=0$.  The argument for disk bubbles is slightly more involved, so we state it as a lemma.

\begin{lemma} \label{lemma bubbling}
Let $v \colon (\bD, \bdd \bD) \to (\Sym^g(\Si),\bT_\al)$ satisfy $n_z(v)=0$.  If $v$ is pseudo-holomorphic, then $v$ is constant.
\end{lemma}
\begin{proof}
If $v (\bD) \subset \bT_\al \cup \bT_\be$ and $v$ is pseudo-holomorphic, then $E(v)=0$, which means that $v$ is constant. So suppose there exists a point $r$ in the interior of $\bD$, such that $v(r) \not \in \bT_\al$.

Let $p_i \colon \Si \times \cdots \times \Si \to \Si$ denote the projection onto the $i$-th coordinate. Then there is an $i$ such that 
\[
p_i(v(r)) \in \Si - \bigcup_{j=1}^g \al.
\]
That is, $q:=p_i(v(r))$ is a point in $\Si$ that is not on any of the $\al$ or $\be$ curves but in some domain $\mD$.  This means that $n_q(v)\neq 0$, and since $n_q(v)=n_z(v)$ for any point $ z \in \mD$, we have that $p_i(v(\bD))$ is a {\it surface} in $\Si$ with $p_i(v(\bdd \bD)) \subset \bigcup_{j=1}^g \al$.  But the $\al$ curves are linearly independent, therefore no subset of the $\al$ curves can bound a subsurface of $\Si$; that is, $p_i(v(\bD))=\Si$.  This contradicts $n_z(v)=0$.
\end{proof}

\subsection{Knot Floer homology} \label{subsection knot floer homology}

We now briefly describe Ozsv\'ath and Szab\'o's {\it Knot Floer homology} denoted by  $\hfkhat$: an invariant for null-homologous knots in 3-manifolds defined using a slightly modified construction of $\hfhat$ \cite{OSknots}.   

Let $Y$ be a closed, oriented 3-manifold, and let $K$ be a null-homologous knot in $Y$.  Then $\hfkhat(Y,K)$ is a bigraded abelian group constructed from a {\it two-pointed Heegaard diagram $(\Si, \hal,\hbe,w,z)$} that encodes the topology of $Y$, as well as of the embedding of $K \subset Y$. It is easy to see that a given Heegaard diagram $(\Si,\hal,\hbe)$ of $Y$, together with two basepoints $w$ and $z$, defines a knot in $Y$.  Indeed, we can connect $w$ with $z$ by an arc $a$ in $\Si - \cup \al_i$, and $z$ with $w$ by an arc $b$ in $\Si -\cup \be_i$.  By pushing $a$ into the handlebody obtained by attaching 2-handles to the $\al$ curves, and $b$ into the handlebody obtained from the $\be$ curves, we create a loop $ab$ (that intersects $\Si$ in exactly two points).  The loop $ab$ is a knot in $Y$.  

On the other hand, given $K\subset Y$, it is not hard to construct a two-pointed diagram.  Consider the Morse-theoretic description of a given two-pointed Heegaard diagram.  Recall that a Heegaard diagram of $Y$ can be obtained from a self-indexing Morse function (see Section \ref{subsection heegaard diagrams}).  Take the knot $K$ in $Y$, and a height function $h$ on $K$, which has only two critical points $p,q$ with $h(p)=0$ and $h(q)=3$.  Extend $h$ to a self-indexing function with the index 1 and index 2 critical points disjoint from $K$, and construct a Heegaard diagram $(\Si,\hal,\hbe)$ for $Y$ as before.  Taking $\Si = h^{-1}(3/2)$, it is clear that $K$ intersects $\Si$ at exactly two points that we define to be the basepoints $w$ and $z$. 

The construction of the Floer complex $\widehat{CFK}$ proceeds identically as before, with the only change that in the differential we require not only $n_z(\phi)=0$, but also $n_w(\phi)=0$.  The resulting {\it Knot Floer homology} is similarly invariant under all the choices that were made in its construction. The complex admits a decomposition along {\it relative $\Spin^c$ structures} denote by $\underline{\Spin^c}(Y,K):=\Spin^c(Y_0(K))$, where $Y_0(K)$ is the result of (canonical) zero-surgery along $K$.

\section{Sutured Floer homology} \label{section SFH}

Sutured manifolds $(M,\ga)$ are 3-manifolds $M$ with boundary together with a set $\ga$ of pairwise disjoint annuli and tori on their boundary. In 2006, Juh\'asz  constructed an important algebraic invariant for sutured manifolds, called {\it sutured Floer homology}, building on Ozsv\'ath and Szab\'o's revolutionary package of Heegaard Floer homology  tools.  Sutured Floer homology combines the powerful theory of sutured manifolds developed by Gabai and the analytical foundations of Floer theory.  


Gabai defined sutured manifolds in 1983 and used them to prove a number of long-standing conjectures.  For example, he proved that zero surgery on a nontrivial knot in $S^3$ cannot result in $S^1 \times S^2$ (the {\it Property R Conjecture} \cite{Gabai III}) or in $S^1 \times S^2 \# Y$, where $Y$ is some closed 3-manifold (the {\it Poenaru Conjecture} \cite{Gabai III}).  He showed that only trivial surgery on satellite knots results in a homotopy $S^3$ ({\it Property P} for satellite knots \cite{Gabai IV}). Gabai, and independently Scharlemann, also proved the {\it superadditivity of knot genus}, which means that the sum of the genera of two knots is a lower bound for the genus of their band connect sum \cite{Gabai V, Scharlemann}.  (A {\it band connect sum} of two knots generalises the concept of a connect sum of two knots.) 


Gabai defined decomposition of sutured manifolds along surfaces that gives new `simpler' sutured manifolds.  He showed that every sutured manifold admits such decompositions that eventually, after a finite number of steps, terminate in the simplest possible manifold -- a product manifold.  These {\it sutured manifold hierarchies} gave a controlled way of breaking down complex manifolds into simpler ones, and in a historical context, extend the work of Haken and Waldhausen in the 1960's on hierarchies of {\it Haken manifolds}, that is, manifolds that are {\it $P^2$-irreducible} (irreducible and have no 2-sided projective planes) and contain a properly embedded, 2-sided incompressible surface.


Whereas Heegaard Floer theory revolutionised the way one studies 3-manifolds, sutured Floer homology did the same for sutured manifolds.    Firstly, sutured Floer homology {\it detects the product manifold}, that is, a (balanced) sutured manifold is a product manifold (homeomorphic to $\text{surface} \times I$) if and only if its sutured Floer homology is equal to $\bZ$ (Corollary \ref{cor product hard direction}).  Since Juh\'asz's theory generalises knot Floer homology, this leads to a new proof that knot Floer homology {\it detects the genus of the knot} (first proved by Ozsv\'ath and Szab\'o \cite{OS genus}), and also a (simpler) proof that knot Floer homology {\it detects the unknot} (first proved by Giggini for genus one knots using contact  topology \cite{Ghi}, and then by Ni in the general case \cite{Ni}); see Remark \ref{rmk detecting genus} .   

Juha\'sz studies the way the sutured Floer homology changes under Gabai's surface decompositions, and shows that (if the decomposition is ``nice'') the sutured Floer homology of the resulting manifold is a direct summand of the sutured Floer homology of the original manifold (Theorem \ref{thm nice intro}).  This gives a way of associating an integer invariant (the rank of SFH) to manifolds at each stage of the hierarchy.  This leads to conclusions about the {\it depth of the hierarchy}; for example, if the rank of sutured Floer homology is less than $2^{k+1}$, then there is a sequence of at most $2k$ sutured decompositions that ends in a product manifold (Proposition \ref{prop stronger}).  In a similar vein, if a genus $g$ knot $K$ has rank of knot Floer homology $\hfkhat(K,g)$ less than $2^{k+1}$, then $K$ has at most $k$ pairwise disjoint non-isotopic minimal genus Seifert surface (Theorem \ref{thm disjoint surfaces}).  

Further, the support of the sutured Floer homology, as a subset of $H^2(M,\bdd M)$, gives rise to the {\it sutured polytope}, which has many properties that can be used to analyse the topology of $M$ (Section \ref{subsection polytope intro}).  For example, all faces of the polytope correspond to surface decompositions: given a face of the polytope there is a surface decomposition such that the resulting manifold's sutured polytope is precisely that face (Corollary \ref{cor decomposition}), and vice versa, given a surface decomposition, the resulting manifold's sutured polytope is precisely a face of the polytope (Proposition \ref{prop decomposition}).    In true spirit of a generalisation, whereas the Heegaard Floer homology polytope is symmetric, the sutured Floer polytope is asymmetric (Example \ref{ex asymmetric}).  Moreover, whereas the Heegaard Floer homology polytope is dual to the Thurston norm unit ball, and its vertices supporting the $\bZ$ group correspond to fibrations of the manifold (\cite[Thm.\,1.1]{Ni09}), the sutured Floer polytope is dual to the foliation cones of Cantwell and Conlon \cite{CC99}, and its vertices supporting the $\bZ$ group correspond to taut, depth one foliations of the sutured manifold (precisely the extension of the concept of a fibration) \cite{Altman2}.  Another comparison worth making, is that the Euler characteristic of the knot Floer homology is a topological invariant of the 3-manifold, namely the Alexander polynomial \cite{OS knots, Ra03}, whereas the Euler characteristic of sutured Floer homology is a type of Turaev torsion, and \cite{FJR10} give an easy algorithm for computing it (Section \ref{subsection torsion}). 


The construction of sutured Floer homology follows closely  the construction of the `hat' flavour of Heegaard Floer homology $\hfhat$.  In the language of the previous sections, we have that every sutured manifold $(M,\ga)$ admits a type of Heegaard diagram $(\Si,\hal,\hbe)$, called a {\it sutured diagram}; here the surface $\Si$ has boundary and its boundary describes the sutures on $(M,\ga)$.  There is no basepoint involved, because any basepoint could be taken to lie in a component of $\Si - \cup_i \al_i \cup_j \be_j$ that contains a boundary component of $\Si$.  Since no holomorphic disk `projects' onto such a component, it follows that automatically $n_z=0$ for every holomorphic disk.  From there, Juh\'asz proves that sutured Floer homology is well-defined by adapting Ozsv\'ath-Szab\'o's Floer theoretic framework for closed manifolds to the setting of sutured manifolds. 
 

This article is meant only as an introduction to the topic of sutured Floer homology (not as a survey article of all work done in this area), so we have omitted Juh\'asz's work on cobordisms \cite{Jucobord}.  This article has grown from the Lagrangian setting, into the Heegaard Floer, and lastly into the sutured Floer setting, and so we decided to stay focused only on the implications that could be explained without much additional theory.  Thus, we have also completely omitted substantial literature concerning the connection between sutured manifolds and contact manifolds (for example see the papers of Honda, Kazez and Mati\'c \cite{contact1,contact2,contact3}).
 
A small notational caveat: unlike in previous sections where $M$ denotes $\Sym^g(\Si)$, in this section $M$ is exclusively used to denote a sutured manifold. 

This section is organised as follows. 
 \begin{itemize}
 \item [\ref{subsection sutured manifolds intro}\, ] An introduction to sutured manifolds, sutured manifold decomposition and hierarchies.
  \item [\ref{subsection sutured diagrams}\, ] Definition of sutured Heegaard diagrams, and why they must be balanced (so that the Floer machinery works).
 \item [\ref{subsection sutured chain}\, ]  Definition of sutured Floer homology $SFH$.
  \item [\ref{subsection relative structures}\, ] Definition of relative $\Spin^c$ structures, as well as the norm-like geometric quantity $c(S,t)$ for sutured manifolds.
 \item [\ref {subsection samples}\, ] Basic properties of $SFH$, its relation to $\hfhat$ and $\hfkhat$.
 \item [\ref{subsection well-behaved surfaces intro}\, ] Types of ``well-behaved'' surfaces and the effect of surface decompositions on SFH.
 \item [\ref{subsection various}\, ] Other Juh\'asz's results, such as how $\rank SFH$ gives a bound on the depth of manifold,  the relation of $SFH$  of Seifert surface complements and $\hfkhat$, and the $SFH$ of a solid torus with $(p,q)$ torus knots as sutures.
 \item [\ref{subsection polytope intro}\, ] Definition of the sutured Floer polytope and some of its properties.
 \item [\ref {subsection norms intro}\, ] A short survey of 3-manifold norms, including Juh\'asz's new invariant $c(S,t)$ that plays the role of a norm for $SFH$.
 \item [\ref{subsection torsion}] Algorithm of Friedl, Juh\'asz, and Rasmussen for computing the Euler characteristic of $SFH$ using Fox calculus methods.
  \end{itemize}

\subsection{Sutured manifolds} \label{subsection sutured manifolds intro}

In this section we define sutured manifolds, balanced sutured manifolds, and strongly balanced sutured manifolds listed in order of decreasing generality.  In order to use the construction of $\hfhat$ for sutured manifolds, Juh\'asz had to impose certain `symmetry' constraints on the sutured manifolds to make them `balanced'; we explain why his definition is the only natural one.  Strongly balanced sutured manifolds, are necessary for the definition of  the sutured polytope, which we give in Section \ref{subsection polytope intro}. Next, we recall the definition of tautness, and explain Gabai's operation of decomposing a sutured manifold along a surface into simpler sutured manifolds.  Lastly, we give a couple of examples and define sutured hierarchies.

\begin{deff}
A {\it sutured manifold} $(M, \gamma)$ is a compact oriented 3-manifold $M$ with boundary together with a set $\gamma \subset \bdd M$ of pairwise disjoint annuli $A(\gamma)$ and tori $T(\ga)$. Furthermore, in the interior of each component of $A(\ga)$ one fixes a suture, that is, a homologically nontrivial oriented simple closed curve. We denote the union of the sutures by $s(\ga)$.

Finally, every component of $R(\gamma): = \bdd M - \Int(\ga)$ is oriented. Define $R_+(\ga)$ (or $R_-(\ga)$) to be those components of $\bdd M - \Int(\ga)$ whose normal vectors point out of (into) $M$. The orientation on $R(\ga)$ must be coherent with respect to $s(\ga)$, that is, if $\de$ is a component of $\bdd R(\ga)$ and is given the boundary orientation, then $\de$ must represent the same homology class in $H_1(\ga)$ as some suture.
\end{deff}

A trivial example is the {\it product sutured manifold} given by $(\Si \times I, \bdd \Si \times I)$ where $\Si$ is a surface with boundary and with no closed components; see Figure \ref{fig product manifolds}. Other simple examples are obtained from any closed, connected 3-manifold by removing a finite number of 3-balls and adding one trivial suture to each spherical boundary component.

\begin{figure}[h]
\centering
\includegraphics [scale=0.40]{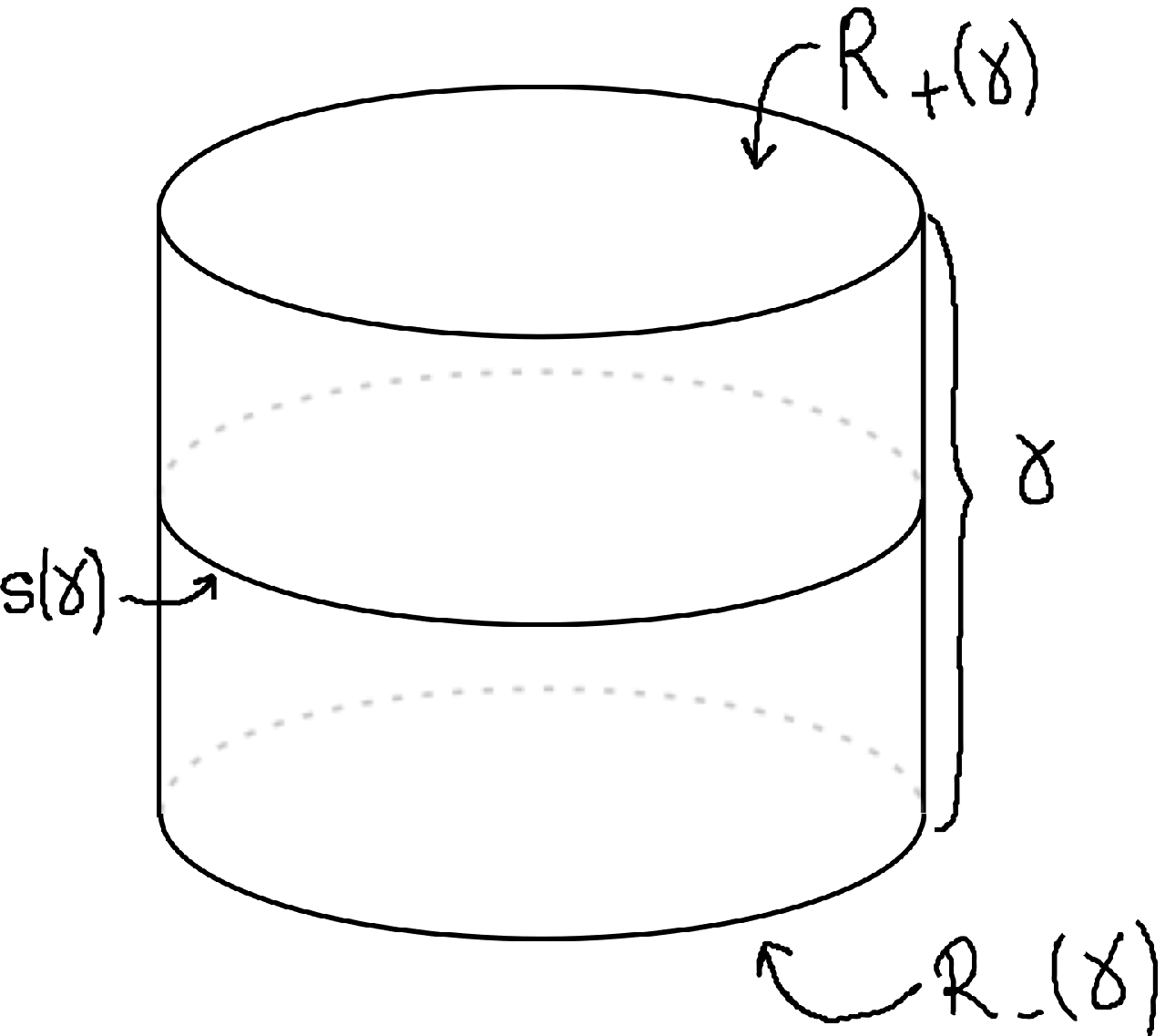}
\hspace{1cm}
\includegraphics [scale=0.40]{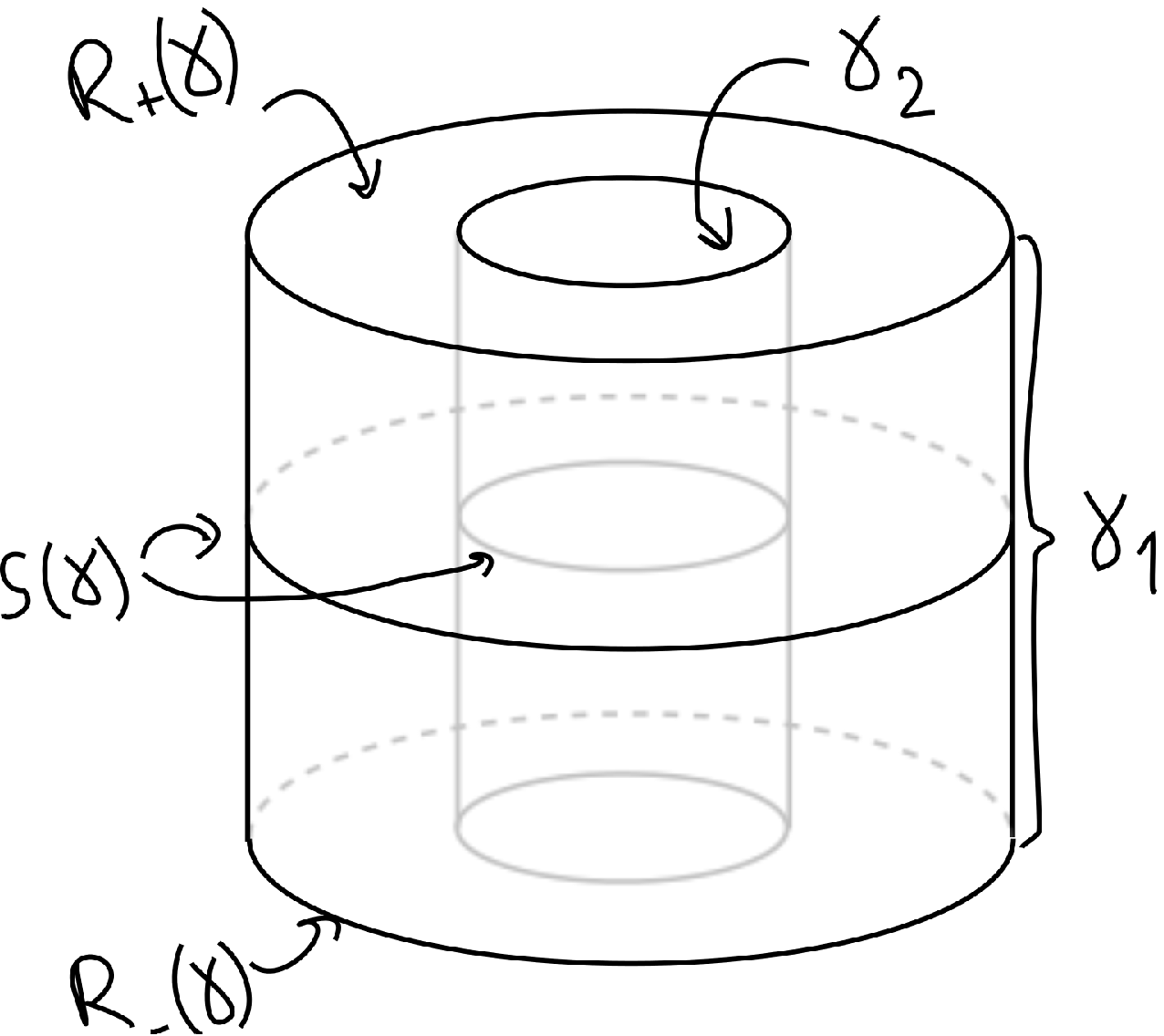}
\caption{Product sutured manifolds $\Si \times I$, with $\Si$ a disk (left) and $\Si$ an annulus (right).}
\label{fig product manifolds}
\end{figure}

Less trivial examples are those of link complements in closed 3-manifolds with sutures consisting of an even number of $(p,q)$-curves on the toroidal components. For example, see Figure \ref{fig toroidal} for an example of solid torus with $(1,1)$ and $(0,1)$ sutures, also denoted by $T(1,1;2)$ and $T(0,1;2)$, respectively.  In particular, if $K$ is a knot in a connected, oriented 3-manifold $Y$, denote by $Y(K)$ the sutured manifold homeomorphic to $Y - N(K)$ with two parallel sutures on $N(K)$ corresponding to meridians of $K$.

\begin{figure}[h]
\centering
\includegraphics [scale=0.40]{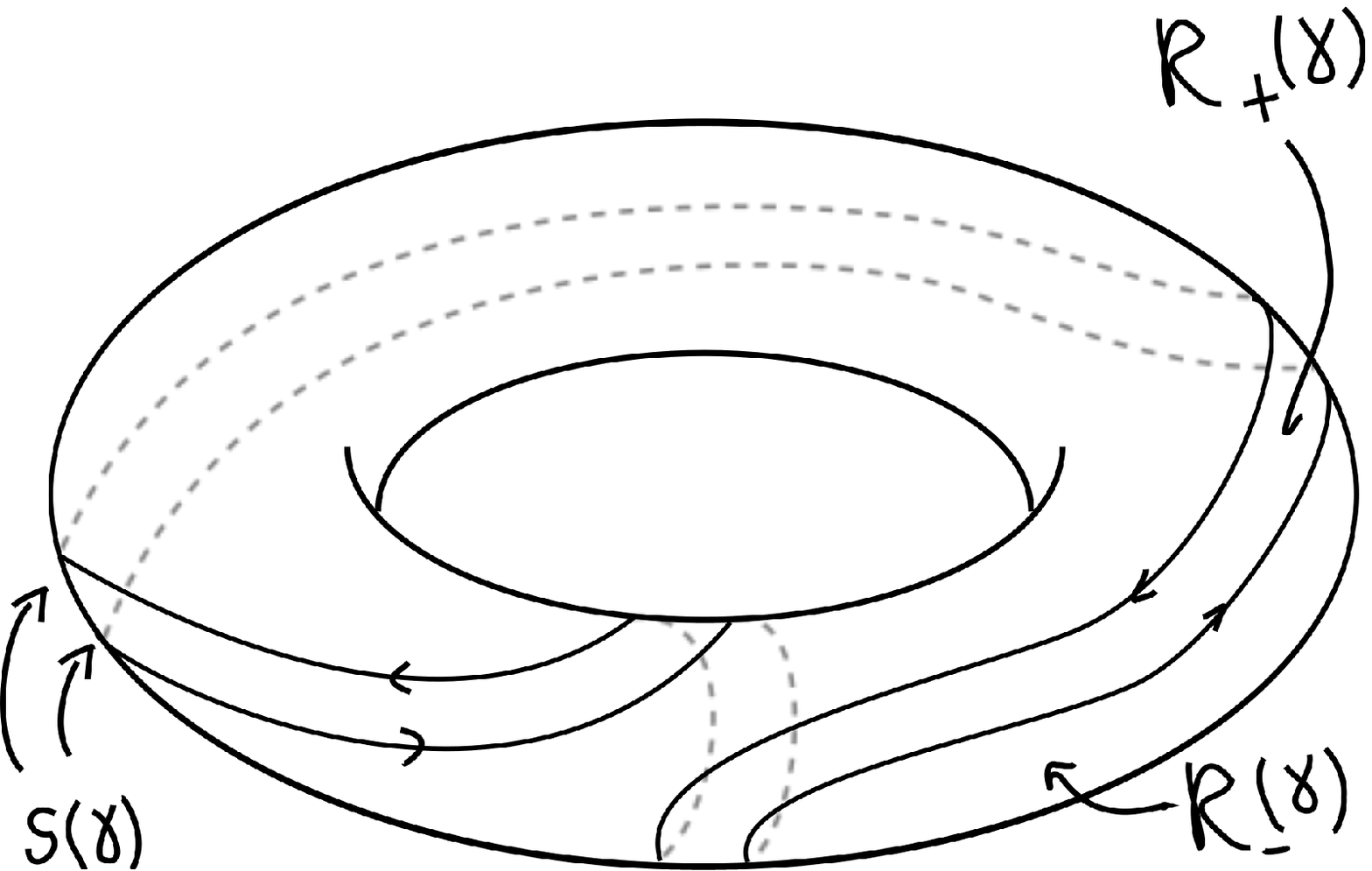}
\hspace{1cm}
\includegraphics [scale=0.40]{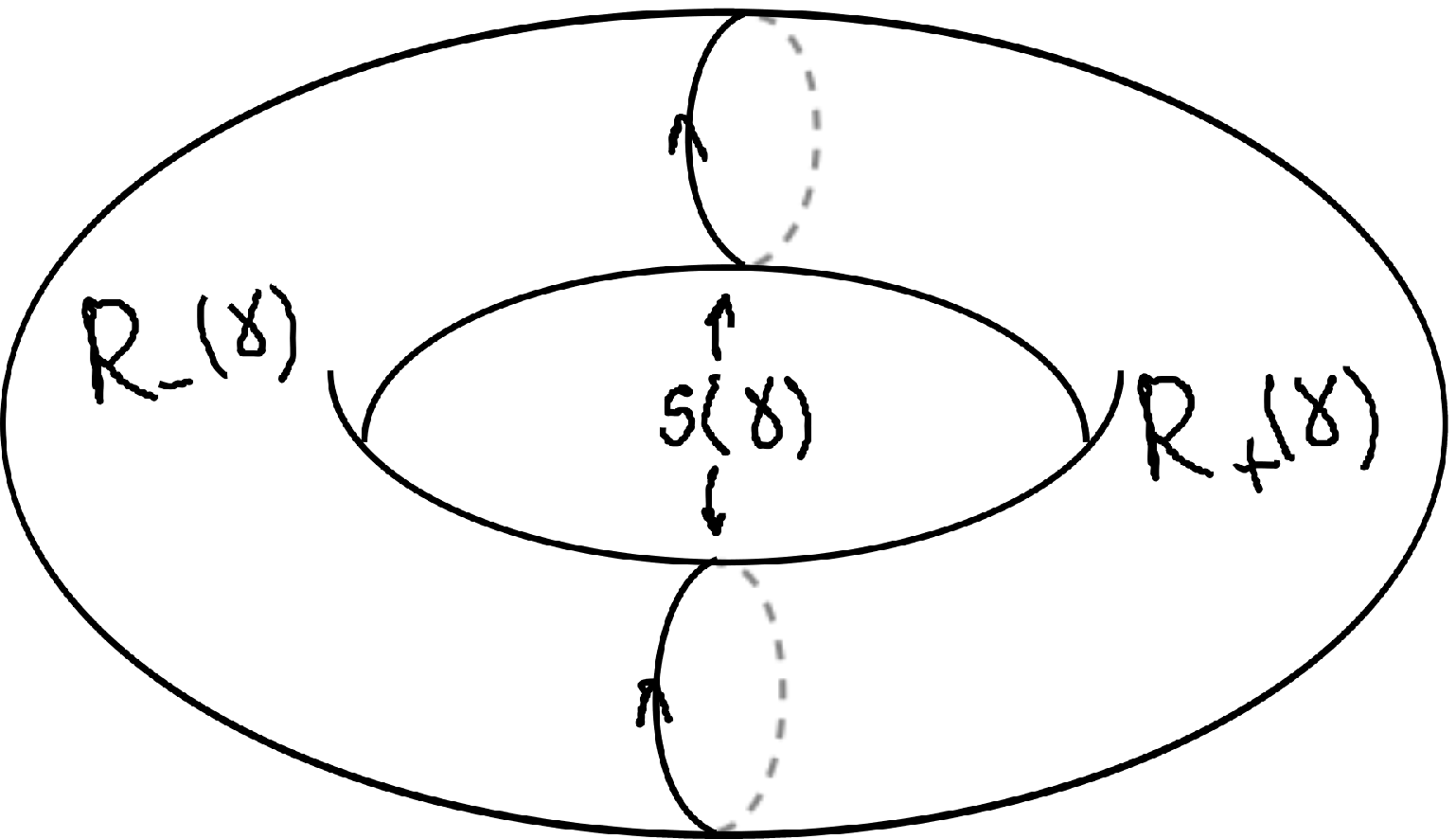}
\caption{Left: parallel $(1,1)$ torus knots give rise to $T(1,1;2)$.  Right: $T(0,1;2)$ is not taut, because $R_{\pm}(\ga)$ are not Thurston-norm minimising.}
\label{fig toroidal}
\end{figure}

\begin{ex} \label{example surface complement}
Also often studied are the complements of surfaces in closed 3-manifolds with sutures derived from the boundary of the surface (for example, the complement of a Seifert surface of a knot).   In particular, if $Y$ is a closed connected oriented 3-manifold, and $R \subset Y$ is a compact, oriented surface with no closed components, then we obtain the balanced sutured manifold $Y(R):=(M,\ga)$, where $M:=Y - N(R\times [0,1])$ and $\ga:=\bdd R \times [0,1]$.
\end{ex}

Now consider the definition of a balanced sutured manifold.

\begin{deff} [Jintro, Def.\,2.2] \label{deff balanced}
A {\it balanced} sutured manifold is a sutured manifold $(M, \ga)$ such that $M$ has no closed components, the equality $\chi(R_+(\ga)) = \chi(R_-(\ga))$ holds, and the map $\pi_0(A(\ga)) \to \pi_0(\bdd M)$ is surjective.
\end{deff}

\begin{rmk} Note that in the definition $M$ has no closed components, because those could already be studied with Heegaard Floer homology.  We only care about manifolds with boundary.  Further,  without loss of generality,  from now on we consider $M$ to be connected.
\end{rmk}

By definition, balanced sutured manifolds have $T(\ga) = \emptyset$ and each component of $\bdd M$ contains at least one element of $A(\ga)$, which can be thought of as a ``thickened'' suture.  At first glance it may not be transparent why sutured Floer homology should be defined only for this class of sutured manifolds, however it is not hard to see why this definition is only natural.  As we have already mentioned, sutured Floer homology follows the construction of $\hfhat$: start with a diagram associated to a sutured manifold and construct an associated homology group that is an invariant of the starting manifold.  The question is what properties do we need a ``sutured diagram'' $(\Si,\hal,\hbe)$ to have so that the Floer machinery can run its course?  Part of the answer is obvious: $\Si$ has to have boundary, $\hal$ and $\hbe$ have to have the same number of elements, and $\hal$, $\hbe$ each have to be a set of linearly independent cures in $H_1(\Si;\bQ)$.  Following Figure \ref{fig conditions}, we recover the definition of a balanced sutured manifold; we give a detailed explanation of the figure in Section \ref{subsection sutured diagrams}.

Next, consider the definition of a strongly balanced sutured manifold; the sutured Floer polytope is only defined for this class of manifolds.  The reason behind this becomes apparent in Section \ref{subsection polytope intro}.

\begin{deff} [Jsurfaces, Def.\,3.5]
A {\it strongly balanced} sutured manifold is a balanced sutured manifold $(M, \ga)$ such that for every component $F$ of $\bdd M$ the equality $\chi(F \cap R_+(\ga)) = \chi(F \cap R_-(\ga))$ holds.
\end{deff}

Further, a sutured manifold can be {\it taut}.

\begin{deff}
A sutured manifold $(M,\ga)$ is said to be {\it taut} if $M$ is irreducible, $R(\ga)$ is incompressible and Thurston-norm minimising in $H_2(M,\ga)$.
\end{deff}

See Figure \ref{fig toroidal} (right) for a simple example of a non-taut sutured manifold.

\begin{rmk} \label{rmk seifert complement is taut}
Note that the sutured manifold defined by a surface complement is strongly balanced.  In particular, when $K \subset S^3$ is a knot with a minimal genus Seifert surface $R$, then $S^3(R)$ is is a taut, strongly balanced sutured manifold.
\end{rmk}

Knowing whether the sutured manifold $(M,\ga)$ is taut or not, says something about its sutured Floer homology: for example, if $M$ is irreducible, but $(M,\ga)$ is not taut then $SFH(M,\ga)=0$ \cite[Prop.\,9.18]{Jusurface}; but if $(M,\ga)$ is taut, then $SFH(M,\ga) \geq 0$ (Theorem \ref{thm nonzero} below).

Now we define the operation of decomposing sutured manifolds into simpler pieces that was introduced by Gabai \cite[Def.\,3.1]{Gabai}.  
\begin{deff} \label{def decomp}
Let $(M,\ga)$ be a sutured manifold and $S$ a properly embedded surface in $M$ such that for every component $\la$ of $S \cap \ga$ one of (i)--(iii) holds:
\begin{enumerate}
\item $\la$ is a properly embedded non-separating arc in $\ga$.
\item $\la$ is  simple closed curve in an annular component $A$ of $\ga$ in the same homology class as $A \cap s(\ga)$.
\item $\la$ is a homotopically nontrivial curve in a toroidal component $T$ of $\ga$, and if $\de$ is another component of $T \cap S$, then $\la$ and $\de$ represent the same homology class in $H_1(T)$.
\end{enumerate}
Then $S$ defines a {\it sutured manifold decomposition} 
\[
(M,\ga) \leadsto^S (M',\ga'),
\]
where $M':=M - \Int(N(S))$ and 
\begin{align*}
\ga': & = (\ga \cap M') \cup N(S_+' \cap R_-(\ga)) \cap N(S'_- \cap R_+(\ga)),\\
R_+(\ga'): & = ((R_+(\ga) \cap M') \cup S_+') - \Int(\ga'), \\
R_-(\ga'): & = ((R_-(\ga) \cap M') \cup S_-') - \Int(\ga'),
\end{align*}
where $S_+'$ ($S'_-$) are the components of $\bdd N(S) \cap M'$ whose normal vector points out of (into) $M$. The manifolds $S_+$ and $S_-$ are defined in the obvious way as copies of $S$ embedded in $\bdd M'$ that are obtained by cutting $M$ along $S$.
\end{deff}

Note the following  special, simple case of surface decomposition.
\begin{deff}
A sutured manifold decomposition $(M,\ga) \leadsto^D (M',\ga')$ where $D$ is a disk properly embedded in $M$ and $\abs {D \cap s(\ga)}=2$ is called a {\it product decomposition}.
\end{deff}

\begin{rmk}
If $(M,\ga)$ is balanced and if $(M,\ga) \leadsto (M',\ga')$ is a product decomposition, then $(M',\ga')$ is also balanced.
\end{rmk}

We now work through two examples of product decomposition, which we later use to compute the sutured Floer homology of a connected sum; see Section \ref{subsection samples}.  Both of the ideas behinds these examples were described by Juh\'asz in proving Proposition 9.14 and 9.15 \cite{Ju}.

\begin{ex}  \label{ex s1s2} In this example we find a product disk $D$ in $S^1\times S^2(1)$ and show that
\[
S^1\times S^2(1) \leadsto^D S^3(2).
\] 
We can choose a ball $B_1 \subset S^1\times S^2$ such that for some point $p \subset S^1$ the intersection $B_1 \cap \{p\} \times S^2$ is a disk  $D_1$.  Remove $B_1$ and put a suture $s_1 \subset \bdd B_1$ so that $s_1 \cap \bdd D_1$ consists of two points; see Figure \ref{fig connected_sum} (left).  Clearly, there is a disk $D$ such that $D \cup D_1=\{p\}\times S^2$ and $D$ is a product disk in $S^1\times S^2 - B_1$ by construction.  Decomposing along $D$ is topologically equivalent to cutting $S^1\times S^2$ along $\{p\} \times S^2$, which leaves $[0,1] \times S^2$ with one suture on each of the boundary balls $\{0,1\}\times S^2$; see Figure \ref{fig connected_sum} (right).

\begin{figure}[h]
\centering
\includegraphics [scale=0.40]{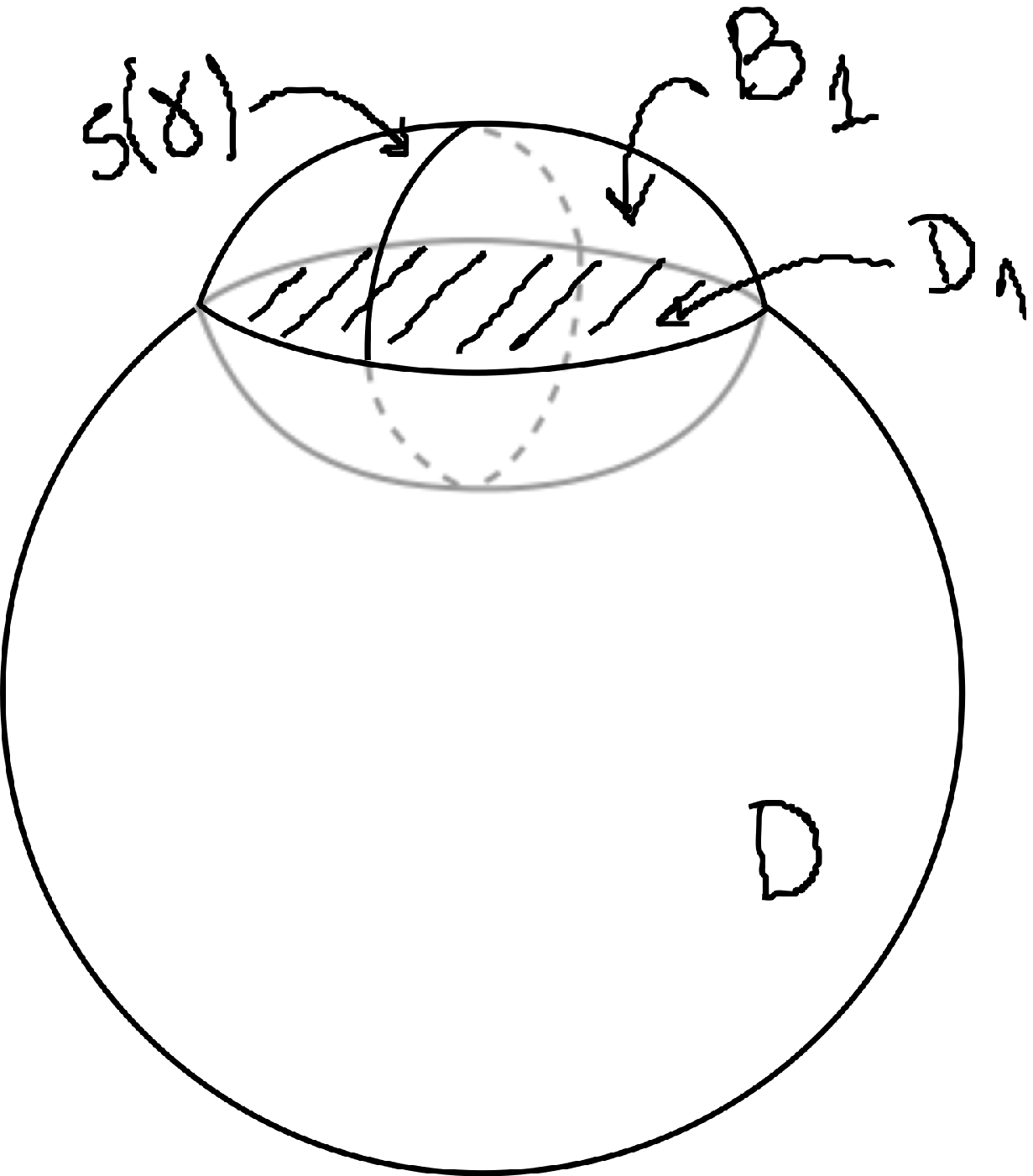} 
\hspace{1cm}
\includegraphics [scale=0.40]{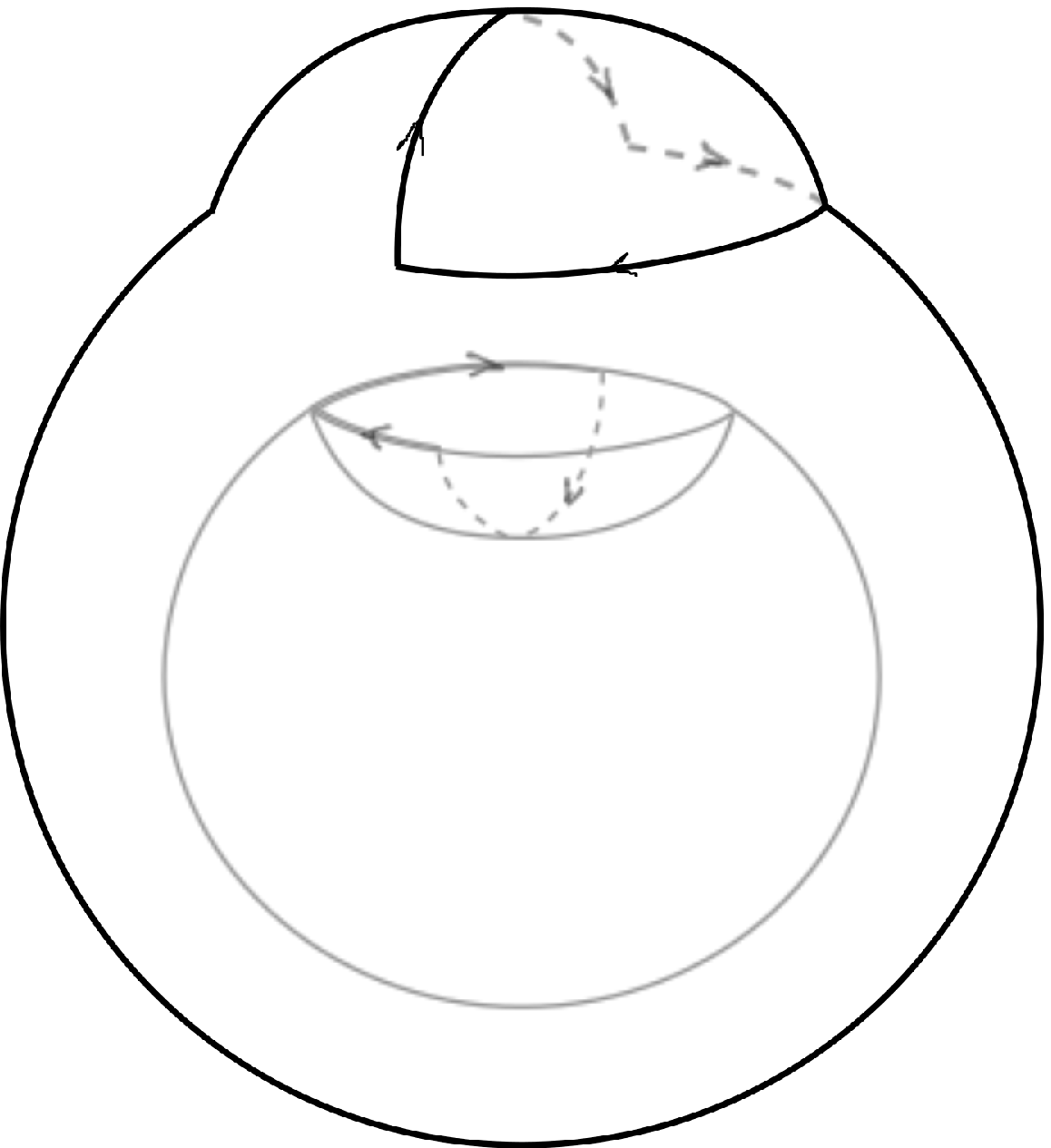}
\caption{Left: removing $B_1$ from $S^1 \times S^2$ leaves $S^1 \times S^2(1)$ and product disk $D$; $D_1 \cup D=\{p\}\times S^2$.  Right: decomposition  $S^1\times S^2 \leadsto^DS^3(2)$ with the two new sutures, one on each of the spheres.}
\label{fig connected_sum}
\end{figure}
\end{ex}

\begin{ex} \label{ex connect sum}  We use a similar idea to show that 
\[
(M,\ga) \# Y \leadsto (M,\ga) \# Y(1).
\]
Specifically, using a finger move, push a part of the boundary of $M$ containing a suture into the connect sum sphere; see Figure \ref{fig product decomp1} (left).  This leaves a product disk $D$ as depicted, such that decomposing along $D$ disconnects the manifold into two components $(M,\ga)$ and $Y(1)$ Figure \ref{fig product decomp1} (right).  Similarly, if $Y$ is replaced by a sutured manifold $(N,\nu)$ we have 
\[
(M,\ga) \# (N,\nu) \leadsto (M,\ga) \# N(1),
\]
where $N(1)=(N,\nu)\# S^3(1)$.
\begin{figure}[h]
\centering
\includegraphics [scale=0.5]{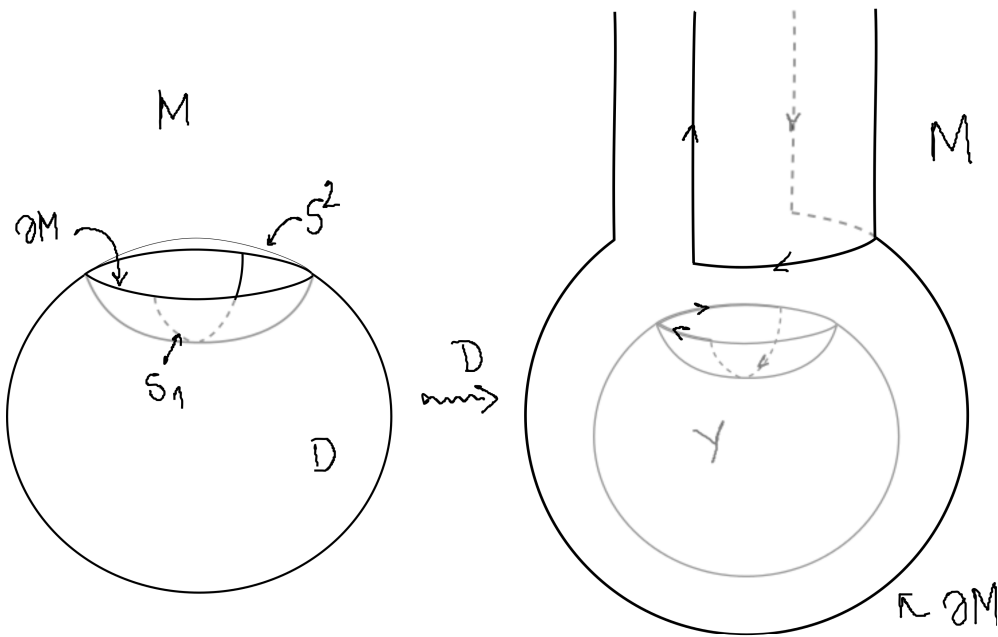}
\caption{The decomposition $(M,\ga) \# Y \leadsto^D (M,\ga) \sqcup Y(1)$.  The manifold $M$ is depicted as being on the ``outside'' (i.e. containing the point at infinity), pushed into the connect sum sphere $S$, and decomposed along $D$.  }
\label{fig product decomp1}
\end{figure}

\end{ex}

As we discuss later (Section \ref{subsection samples}) product decompositions do not change the sutured Floer homology; that is, $SFH(M,\ga)=SFH(M',\ga')$, which means that they are a very good tool for simplifying the original manifold in the hope of obtaining something for which the sutured Floer homology is more easily computed.

Lastly, each taut sutured manifold (that is not a rational homology sphere containing no essential tori) admits a series of decompositions ending in a product manifolds; this is a theorem of Gabai, see Theorem \ref{thm gabai hierarchy}.

\begin{deff} \label{deff hierarchy}
A  {\it sutured manifold hierarchy} is a sequence of decompositions along surfaces 
\[
(M_0,\ga_0) \leadsto^{S_1} (M_1,\ga_1) \leadsto^{S_2} \cdots \leadsto^{S_n} (M_n,\ga_n),
\]
where $(M_n,\ga_n)$ is a product sutured manifold.   The number $n$ is called the {\it depth} of the hierarchy.
\end{deff}

Define the {\it depth of a sutured manifold} $(M,\ga)$ to be the minimum depth over all sutured hierarchies of $(M,\ga)$.

\subsection{Sutured Heegaard diagrams} \label{subsection sutured diagrams}

In this section we define {\it sutured (Heegaard) diagrams} $(\Si, \hal,\hbe)$, and explain the necessary restrictions on the diagrams in order for the Floer machinery to work (see Figure \ref{fig conditions}).  For example, unlike the Heegaard diagrams of closed 3-manifolds, a general sutured diagram does not necessarily have the same number of $\al$ and $\be$ curves. So an obvious example of a Floer theoretic restriction is that the number of $\al$ and $\be$ curves must be the same so that the $\bT_\al$ and $\bT_\be$ are of the same dimension.  This, and other natural conditions, lead to the definition of what Juh\'asz called {\it balanced sutured manifolds} in Definition \ref{deff balanced}.

Without loss of generality we may assume that $M$ is connected.
\begin{deff} [Jintro, Def.\,2.7] A {\it sutured Heegaard diagram} is a tuple $(\Si, \hal,\hbe)$, where $\Si$ is a compact oriented surface with boundary and $\hal=\{\al_1, \ldots, \al_m\}$ and $\hbe=\{\be_1, \ldots, \be_n \}$ are two sets of pairwise disjoint simple closed curves in $\Int(\Si)$.
\end{deff}
Remark that since we are assuming that $M$ is connected, this means that $\Si$ is connected as well. Hence if $\bdd M \neq \emptyset$, then $H_2(\Si)=0$.

Any given sutured Heegaard diagram $(\Si,\hal,\hbe)$ defines a sutured manifold $(M,\ga)$ using the following construction.  Take the product manifold $\Si \times I$ and attach 3-dimensional 2-handles along the curves $\al_i \times \{0\}$ and $\be_j \times \{1\}$ for $i=1, \ldots, m$ and $j=1,\ldots, n$.  Smoothing the corners we obtain a three-manifold with boundary $M$, and with sutures $s(\ga)=\bdd \Si \times \{1/2\}$. Actually the converse is also true: given a sutured manifold (with a certain small restriction) we can find such a sutured Heegaard diagram defining it, see Lemma \ref{lemma self-indexing} for details.

However, in order for use the Floer machinery to construct a homology group from a sutured Heegaard diagram defining a sutured manifold $(M,\ga)$, we need that $\abs \hal=\abs \hbe$ and that $\hal$ and $\hbe$ are each linearly independent in $H_1(\Si; \bQ)$, so that the tori $\bT_\al$ and $\bT_\be$ are well-defined and the Lagrangian Floer construction can be applied. 

For now consider the following definition of a {\it balanced diagram}.

\begin{deff} [Jintro, Def.\,2.11]
A sutured Heegaard diagram $(\Si, \hal, \hbe)$ is called {\it balanced} if $\abs {\hal} = \abs{\hbe}$ and the maps $\pi_0(\bdd \Si) \to \pi_0(\Si - \bigcup \hal)$ and $\pi_0(\bdd \Si) \to \pi_0(\Si - \bigcup \hbe)$ are surjective.
\end{deff}
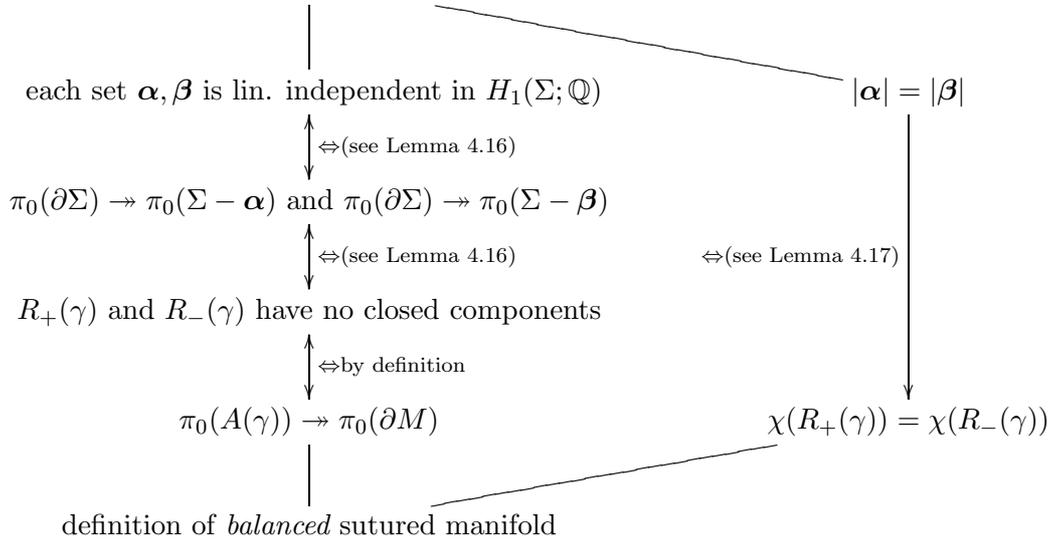
\begin{figure}[h]
\centering
\xymatrix{
 \text{Conditions on } (\Si,\hal,\hbe), \bdd \Si \neq \emptyset, \text{ for doing Floer homology}            \\
 \text{ each set $\hal,\hbe$ is lin. independent in }H_1(\Si;\bQ) \ar@{-}[u] & \abs \hal =\abs \hbe \ar[ddd]_{\Leftrightarrow \text{(see Lemma \ref{setup of definition of balanced 2}})}  \ar@{-}[ul] \\
   \pi_0(\bdd\Si) \twoheadrightarrow \pi_0(\Si - \hal) \text{ and } \pi_0(\bdd\Si) \twoheadrightarrow \pi_0(\Si - \hbe) \ar@{<->}[d]^{\Leftrightarrow \text{(see Lemma \ref{setup of definition of balanced 1}})}  \ar@{<->}[u]_{\Leftrightarrow \text{(see Lemma \ref{setup of definition of balanced 1}})} &      \\
 R_+(\ga) \text{ and } R_-(\ga) \text{ have no closed components}         \ar@{<->}[d]^{\Leftrightarrow \text{by definition}}   &       \\
  \pi_0(A(\ga)) \twoheadrightarrow \pi_0(\bdd M)  \ar@{-}[d] & \chi(R_+(\ga))=\chi(R_-(\ga)) \ar@{-}[dl] \\
 \text{definition of {\it balanced} sutured manifold}
	     }
\caption{Why a sutured manifold has to be balanced before we can use the construction of Heegaard Floer homology.}
\label{fig conditions}
\end{figure}

\begin{lemma} \label{setup of definition of balanced 1}
Let $(\Si,\hal,\hbe)$ be a sutured diagram.  Then the following statements are equivalent
\begin{enumerate} 
\item The elements of $\hal$ are linearly independent in $H_1(\Si;\bQ)$.
\item   $\pi_0(\bdd \Si) \to \pi_0(\Si - \bigcup \hal)$ is surjective.  
\item $R_+(\ga)$ has no closed components.
\end{enumerate}
An analogous list of equivalent statements can be made for $\hbe$.
\end{lemma}

\begin{proof}
$(i) \Leftrightarrow (ii)$ All homology groups in this proof are taken with rational coefficients.  Let $A$ denote the union of all the curves in $\hal$.

Consider the embedding $i \colon A \hookrightarrow \Si$; the components of $A$ are linearly independent if and only if the induced map on homology $i_* \colon H_1(A) \to H_1(\Si)$ is injective, that is, if $\Ker (i_*)=0$.  Next, the part of the long exact sequence of the pair $(\Si,A)$ given by
\[
0 \to H_2(\Si) \to H_2(\Si,A) \to H_1(A) \xrightarrow{i_*} H_1(\Si) \to \cdots
\]
implies that $H_2(\Si,A) \cong H_2(\Si) \oplus \Ker (i_*)$.  Since we assumed $\Si$ to be connected and have boundary, $H_2(\Si)=0$.  

By excision $H_2(\Si,N(A))=H_2(\Si - \Int(N(A)), \bdd N(A))$.  The left-hand side is just $H_2(\Si,A)$, and the right-hand side splits as a direct sum of homology groups of the form
\begin{equation} \label{C}
H_2(C,\bdd N(A) \cap C),
\end{equation}
where $C$ runs over the components of $\Si - \Int (N(A))$.  Thus, $H_2(\Si,A)=0$ if and only if each of the homology groups in \eqref{C} is zero, if and only if each $C$ contains a boundary component of $\Si$.  This last statement is equivalent to saying that $\pi_0(\bdd \Si) \to \pi_0(\Si - A)$ is surjective.
\\
\\
\noindent $(ii) \Leftrightarrow (iii)$ 
Since $R_+(\ga)$ is obtained from $\Si$ by performing surgery on the attaching $A$ curves, it follows that the components of $R_+(\ga)$ naturally correspond to components of $\Si -  A$. Thus, components of $R_+(\ga)$ have nonempty boundary if and only if the corresponding component of $\Si - A$ contains a component of $\bdd \Si$.  The statement follows.
\end{proof}

\begin{lemma} \label{setup of definition of balanced 2}
Let $(\Si,\hal,\hbe)$ be a sutured diagram defining a sutured manifold $(M,\ga)$.  Then $\abs \hal = \abs \hbe$ if and only if $\chi(R_+(\ga))=\chi(R_-(\ga))$.
\end{lemma}

\begin{proof}
First consider the relationship between $R_+(\ga)$ and $\Si$: a 2-handle $D^2 \times [0,1]$ is attached to each $\al$ curve on $\Si$ falong $\bdd D^2 \times \{1/2\}$, so $R_+(\ga)$ consists of $\Si$ minus a neighbourhood of each $\al$ curve union $D^2 \times \{0,1\}$ of each 2-handle.  In other words, $R_+(\ga)$ is the result of doing surgeries on $\Si$ along the $\al$ curves, so $\chi(R_+(\ga))=\chi(\Si)+2\abs \hal$ and similarly $\chi(R_-(\ga))=\chi(\Si)+2\abs \hbe$.  Thus $\chi(R_-(\ga))=\chi(R_+(\ga))$ if and only if $\abs \hal = \abs \hbe$.
\end{proof}

The two lemmas together give the following proposition.

\begin{prop}
Let $(\Si, \hal,\hbe)$ be a sutured diagram of a sutured manifold $(M,\ga)$.  Then $(\Si,\hal,\hbe)$ is balanced if and only if $(M,\ga)$ is balanced.
\end{prop}

We can also find a sutured diagram defining a sutured manifold.

\begin{lemma} \label{lemma self-indexing} \cite[Prop.\,2.13]{Ju}
Let $(M,\ga)$ be a sutured manifold for which the maps \linebreak $\pi_0 (R_+(\ga)) \to \pi_0(M)$ and $\pi_0(R_-(\ga)) \to \pi_0(M)$ are surjective.  Then there exists a sutured Heegaard diagram $(\Si, \al, \be)$.
\end{lemma}

\begin{proof}
Following the closed case (see Section \ref{subsection heegaard diagrams}), the idea is to construct a self-indexing Morse function with no index 0 and no index 3 critical points.  Then the construction proceeds as before.

The restrictions on $\pi_0$ imply that $\pi_0(A(\ga)) \to \pi_0(\bdd M)$ is surjective and $T(\ga)= \emptyset$.  As each element of $A(\ga)$ is an annulus, choose a diffeomorphism  $\phi \colon A(\ga) \to [-1,4]$ so that $s(\ga)= \phi^{-1}(3/2)$.  Define $f \colon \bdd M \to [-1,4]$ be given by $f|_{A(\ga)}:=\phi$, $f|_{R_-(\ga)}:= -1$ and $f|_{R_+(\ga)} := 4$.  Take an extension of $f$ to the interior of $M$: generically $f$ is Morse.  Also, as before, we can assume that $f$ is self-indexing on the interior of $M$ (\cite[Thm.\,4.8]{Milnor}).

It remains to show that $f$ can be perturbed in the interior of $M$, so that the new function $f' \colon M \to [-1,4]$ is has no index 0 and no index 3 critical points, and $f'|_{\bdd M}=f|_{\bdd M}$.  This is done in Theorem 8.1 \cite{Milnor} by appropriately pairing up critical points, applying Smale's cancellation lemma, and then modifying $f$ in the neighbourhood of the cancellation.  Let $x$ be an index 0 critical point.  We can consider $x$ to be a generator of $C_0(M,R_-(\ga);\bZ_2)$.  But since $H_0(R_-(\ga);\bZ_2)\to H_0(M;\bZ_2)$ is surjective, $H_0(M,R_-(\ga);\bZ_2)=0$.  Thus, thinking in terms of  cellular homology, we know that there exists a index 1 critical point $y$, such that $\bdd y=x$.  Since $y$ is of index 1, this means that are at most two flow lines flowing from $y$, and hence there must be exactly one flow line connecting $x$ and $y$.  Smale's cancellation lemma says that we can perturb $f$ in the neighbourhood of this flow line so that the new function has no critical points in that neighbourhood.  We can pair up every critical point of index 0 with a critical point of index 1, and similarly every critical point of index 3 and with a critical point of index 2.  Therefore, we obtain a Morse function $f'$ as desired.
\end{proof}

Note that conversely, given a sutured Heegaard diagram $(\Si,\hal,\hbe)$, we can always construct a self-indexing Morse function as described in Lemma \ref{lemma self-indexing}.

Lastly, it remains to determine when two diagrams define the same manifold.  As in the closed case, the equivalence classes arise from the relation on manifolds given by Heegaard moves.

\begin{prop}
Two balanced diagrams $(\Si,\hal,\hbe)$ and $(\Si',\hal',\hbe')$ define the same balanced sutured manifold $(M,\ga)$ if and only if the diagrams are connected by a sequence of Heegaard moves.
\end{prop}

It is clear that if two diagrams are related by a sequence of Heegaard moves, then they define the same manifold.  Indeed, recall that these Heegaard moves are isotopies, handle-slides and (de)stabilisations: the attaching procedure is isotopy invariant; handle-slides do not effect the homotopy type of the resulting manifold, and stabilisations amount to a connected sum with a 3-ball.

  Juh\'asz proves the other direction of the proposition by generalising the method of Ozsv\'ath-Szab\'o in the closed case  \cite[Prop.\,2.15]{Ju}.  The general idea behind the proof is as follows.  Take two diagrams $(\Si_0,\hal_0,\hbe_0)$ and $(\Si_1,\hal_1,\hbe_1)$ defining the same $(M,\ga)$ and constructed using the self-indexing Morse functions  $f_0$ and $f_1$ as in the proof of Lemma \ref{lemma self-indexing}.  Fix a Riemannian metric on $M$.  Then a generic path $f$ in the space of all smooth real-valued functions on $M$ connecting $f_0$ and $f_1$ is such that for all $t \in I - E$, where $E$ is a finite set, $f_t$ is Morse-Smale.  In other words, away from a set $E$, $f_t$ is Morse with gradient flow lines flowing only from larger to strictly smaller index critical points.  We can assume that $f_t$ remains unchanged on a neighbourhood of $\bdd M$. Thus, those $f_t$ define sutured diagrams $(\Si_t,\hal_t,\hbe_t)$ of $M$, where $\hal_t$ and $\hbe_t$ are intersections of $\Si_t$ with the ascending and descending manifolds of the index one and index two critical points of $f_t$, respectively.  Now we can study how $f_t$ changes when it passes through a point $e \in E$, and this leads to the conclusion that, for a small $\epsilon>0$, the diagrams $(\Si_{e-\ep}, \hal_{e-\ep},\hbe_{e-\ep})$ and $(\Si_{e+\ep}, \hal_{e+\ep},\hbe_{e+\ep})$ differ by a Heegaard move.

\subsection{Sutured Floer homology} \label{subsection sutured chain}

In this section we summarise the differences between the Floer setting for closed manifolds and for sutured manifolds, and then  define the sutured Floer complex. It is helpful if the reader is  familiar with Section \ref{section HF}, but we define all the terms here again.

The domains of a balanced sutured diagram $(\Si,\hal,\hbe)$ are defined similarly to those of a \linebreak Heegaard diagram: denote by $D_1, \dots, D_m$  the closures of the components of \linebreak $\overline \Si:=\Si - \cup_{i=1}^g \al_i - \cup_{j=1}^g \be_j$ {\it disjoint from $\bdd \Si$}.  Choose a point $z_k$ in the interior of each $D_k$.  Then
\[
\mD(\phi):=\sum_{i=1}^m n_{z_k}(\phi) D_i,
\]
where $n_{z_k}$ is the algebraic intersection number of a Whitney disk $u$ representing $\phi$ and the hypersurface $V_z:=\{z_k\} \times \Sym^{g-1}(\Si)$, as before.

\begin{rmk} \label{rmk who cares about boundary} In general, we could also pick a point $z$ in a component $K$ of $\overline \Si$ with $\bdd K \cap \bdd \Si \neq \emptyset$.  In that case also $n_{z}(u)$ does not depend on the choice of point $z$ in $K$, so we may choose $z$ to lie on $\bdd \Si$.  But we can always homotope $u$ away from $\bdd \Sym^d(\Si)$ so that $n_z(u)=0$.  Therefore, in the definition of domains we do not need to worry about such components $K$ of $\overline \Si$ and we can use all of the Floer machinery developed for closed Heegaard diagrams.
\end{rmk}

 A domain $\mP$ is {\it periodic} if the boundary of the 2-chain is a sum of $\al$- and $\be$-curves. A balanced diagram $(\Si,\hal,\hbe)$ is {\it admissible} if every periodic domain $\mP \neq 0$ has both positive and negative coefficients. Juh\'asz showed that every balanced diagram is isotopic to an admissible one \cite[Prop.\,3.15]{Ju}.  The idea is to isotope one set of curves, say the $\hbe$ curves, using finger moves along a set of curves generating $H_1(\Si,\bdd \Si)$.  Carefully chosen curves and finger moves result in a diagram where all periodic domains have both positive and negative coefficients.   
 
 \begin{deff}
Let $(\Si,\hal,\hbe)$ an admissible balanced diagram defining a balanced sutured manifold $(M,\ga)$.  Then define $CF(\Si,\hal,\hbe)$ to be the free abelian group generated by the points of the intersection $\bT_\al \cap \bT_\be$:
\[
CFH(\Si,\hal,\hbe) := \bigoplus_{\hx \in \bT_\al \cap \bT_\be} \bZ_2 \langle \hx \rangle.
\]
  Define an endomorphism $\bdd \colon CF(\Si,\hal,\hbe) \to CF(\Si,\hal,\hbe)$ in the following way:  for a point $\hx \in \bT_\al \cap \bT_\be$, let 
\begin{equation}
\partial x := \sum_{\hy \in L_1 \cap L_2} \left( \sum_{\{\phi \in \pi_2(\hx,\hy) \mid \mu(\phi)=1\}} \#_2 \wM (\hx,\hy) \right)\cdot \hy.
\end{equation}
\end{deff}
 
 We continue to work with $\bZ_2$ coefficients because we have not discussed the matter of being able to choose coherent system of orientations, but as with Heegaard Floer homology, this can be done in the case of sutured Floer homology as well; see Remark \ref{rmk coherent orientation}.
 
 Recall that in order to prove that $(CFH(\Si,\hal,\hbe),\bdd)$ is indeed a chain complex we need to show the following two statements hold.
 
 \begin{enumerate}
\item There are at most finitely many elements in the set $\{\phi \in \pi_2(\hx,\hy) \mid \mu(\phi)=1\}$ for every $\hx,\hy \in \bT_\al\cap\bT_\be$.  This proves that $\bdd$ is well-defined.

\item If $\mu(\phi)=1$ or $\mu(\phi)=2$ and $(u_n)_{n \in \bN} \subseteq \mM(\hx,\hy;\phi)$, then there are no singular points (that is, no bubbling). This proves that $\bdd^2=0$.
\end{enumerate}
 
Indeed, for an admissible balanced diagram $(\Si,\hal,\hbe)$ the set of domains over which we sum to obtain the differential is finite.  We can use the same proof as in Theorem \ref{thm finite set} with appropriate simplifications (such as $n_z=0$ is automatic).  So the set is finite, and $\bdd$ is well-defined.

As for bubbling, disk bubbles are excluded by a slight modification of the final argument of Lemma \ref{lemma bubbling}.  Namely, if there exists a disk bubble with boundary in $\bT_\al$ and with an interior point mapping into some domain $\mD$, then by homotopy invariance of $n_z$, the projection of the disk onto $\Si$ has to be a surface.  But the $\al$ curves are linearly independent, so no subset of the $\al$ curves can bound a subsurface of $\Si$. Contradiction. Sphere bubbles are excluded because $\Si$ has no closed components and the domain of the sphere is a 2-cycle.  Therefore, $\bdd^2$ is well-defined.

\begin{thm}
The {\it sutured Floer homology} $SFH(M,\ga)$ of a balanced sutured manifold $(M,\ga)$ is defined to be the homology of the chain complex $(CFH(\Si,\hal,\hbe),\bdd)$, where $(\Si,\hal,\hbe)$ is an admissible diagram for $(M,\ga)$.  
\end{thm}

Juh\'asz shows the invariance of various (other) choices using the same proofs as in the Heegard Floer homology setting for closed manifolds.

As in the case of Heegaard Floer homology, the chain complex decomposes along an equivalent notion of $\Spin^c$ structures, called {\it relative $\Spin^c$ structures of $(M,\ga)$}.  We discuss them in detail in the following section, but for not note that for each relative $\Spin^c$ structure $\fs \in \Spin^c(M,\ga)$ there is a well-defined abelian group $SFH(M,\ga,\fs)$, and the direct sum of these groups is the total sutured Floer homology of $(M,\ga)$:
\[
SFH(M,\ga):=\bigoplus_{\fs \in \Spin^c(M,\ga)} SFH(M,\ga,\fs).
\]

As before, the differential lowers the grading by one 
\[
\bdd \colon CFH(M,\ga,\fs)_* \to CFH(M,\ga,\fs)_{*-1}.
\]
 By Remark \ref{rmk on Z2}, $SFH(M,\ga,\fs)$ can always be graded modulo the minimal Chern number $\fd(\fs)$, where 
\[
\fd(\fs):= \gcd_{\xi \in H_2(M)} \langle c_1(\fs),\xi \rangle.
\]
So for some $\hx,\hy$ such that $\fs(\hx)=\fs(\hy)$, and $\phi,\psi \in \pi_2(\hx,\hy)$, 
\[
\mu(\phi)=\mu(\psi) \mod \fd(\fs).
\]
In particular, as before, there is always a relative $\bZ_2$ grading of $SFH(M,\ga,\fs)$.  What we have said so far explains why  sutured Floer homology may be summarised in the introductory section of a paper by saying something like:  `Sutured Floer homology associates to a given balanced sutured manifold $(M,\ga)$ a finitely-generated bigraded abelian group denoted by $SFH(M,\ga)$, where one grading is given by the decomposition along relative $\Spin^c$ structures of $(M,\ga)$, and the other is a relative $\bZ_2$ grading'.

\subsection{Relative $\Spin^c$ structures} \label{subsection relative structures}

In this section we describe relative $\Spin^c$ structures of a sutured manifold $(M,\ga)$.  Like in the closed case, our particular topological definition originates from Turaev's work \cite{Tu90}. Our current phrasing comes from \cite{Ju}.   Unlike before, we now need to describe what restrictions we make on the vector fields on the boundary of $M$.  For proofs we refer to Juh\'asz's papers \cite{Ju,Jusurface} and \cite{Jupolytope}.

Fix a Riemannian metric on $(M,\ga)$.  Let $v_0$ denote a nonsingular vector field on $\bdd M$ that points  into $M$ on $R_-(\ga)$ and out of $M$ on $R_+(\ga)$, and that is equal to the gradient of a height function $s(\ga) \times I \to I$ on $\ga$. The space of such vector fields is contractible.

A relative $\Spin^c$ structure is defined to be a {\it homology class} of vector fields $v$ on $M$ such that $v|{\bdd M}$ is equal to $v_0$.  Here two vector fields $v$ and $w$ are said to be {\it homologous} if there exists an open ball $B \subset \Int(M)$ such that $v$ and $w$ are homotopic through nonsingular vector fields on $M - B$ relative to the boundary.  There is a free and transitive action of $H_1(M)=H^2(M,\bdd M)$ on $\Spin^c(M,\ga)$ given by {\it Reeb turbularization} \cite[p.\,639]{Tu90}.  This action makes the set $\Spin^c(M,\ga)$ into an $H_1(M)$-torsor.  From now on, we refer to a map $\iota \colon \Spin^c(M,\ga) \to H_1(M)$ as an  {\it identification} of the two sets if $\iota$ is an $H_1(M)$-equivariant bijection.  Note that $\iota$ is completely defined by which element $\fs \in \Spin^c(M,\ga)$ it sends to $0 \in H_1(M)$ (or any other fixed element of $H_1(M)$).

At this point, as before, we can define $c_1(\fs) \in H^2(M)$, the {\it Chern class of $\fs$}, to be the first Chern class of the perpendicular two-plane field $v^\perp \to M$.   Equivalently, $c_1(\fs)$ is the first obstruction to $v^\perp$ being a trivial bundle over $M$. This $c_1(\fs)$ is used in the definition of the relative $\bZ/\fd(\fs)$ grading of each $SFH(M,\ga,\fs)$ in Section \ref{subsection sutured chain}.

The perpendicular two-plane field $v_0^\perp$ is trivial on $\bdd M$ if and only if $(M,\ga)$ is strongly balanced \cite[Prop.\,3.4]{Jusurface}.  Suppose that $(M,\ga)$ is strongly balanced.  Define $T(M,\ga)$ to be the set of trivialisations of $v_0^\perp$.  Let $t \in T(M,\ga)$.  Then there is a map dependent on the choice of trivialisation,
\[
c_1(\cdot, t) \colon \Spin^c(M,\ga) \to H^2(M,\bdd M),
\]
where $c_1(\fs,t)$ is defined to be the {\it relative Euler class} of the vector bundle $v^\perp \to M$ with respect to a partial section coming from a trivialisation $t$.  So $c_1(\fs,t)$ is the first obstruction to extending the trivialisation $t$ of $v_0^\perp$ to a trivialisation of $v^\perp$.  Here $v$ is a vector field on $M$ representing the homology class $\fs$.

\begin{rmk} \label{rmk gold}
Consider for a moment the following segment of the long exact sequence of the pair $(M,\bdd M)$:
\[
\cdots \to H^1(\bdd M) \xrightarrow{d} H^2(M,\bdd M) \xrightarrow{p^*} H^2(M) \xrightarrow{i^*} H^2(\bdd M) \to \cdots
\]
Fix $\fs \in \Spin^c(M,\ga)$ and let $v$ be a vector field representing $\fs$.  Then by the naturality of Chern classes we have $i^*(c_1(v^\perp))=c_1(v_0^\perp)$.  Suppose $v_0^\perp$ is trivial, then $c_1(v^\perp) \in \Ker i^*$.  Moreover, since $v_0^\perp$ is trivial we can define the relative Chern class $c_1(\fs,t) \in H^2(M,\bdd M)$ for some trivialisation $t$, which can be thought of as an element of $H^1(\bdd M)$.  Again by naturality, we have $p^*(c_1(\fs,t))=c_1(\fs)$.  

Moreover, Lemma 3.12 \cite{Jupolytope} says that for $\fs \in \Spin^c(M,\ga)$ and $t_1,t_2 \in T(M,\ga)$, we have 
\[
c_1(\fs,t_1) - c_1(\fs,t_2)=d(t_1-t_2).
\]
So indeed, $p^*(c_1(\fs,t_1) - c_1(\fs,t_2))=p^* \circ d(t_1-t_2)=0$, which confirms that $p^*(c_1(\fs,t_1))=p^*(c_1(\fs,t_2)) \in H^2(M)$.
\end{rmk}

A related concept to $c_1(\cdot,t)$ is the geometric quantity $c(S,t)$ associated to an oriented decomposing surface $S$ in $(M,\ga)$ \cite[Def.\,3.16]{Jupolytope}.  This invariant gives rise to the {\it sutured Floer norm} that is directly related to the {\it sutured Floer polytope}; we discuss this matter in detail in Section \ref{subsection norms intro}.  
The definition of $c(S,t)$ says that it is a sum of three geometric quantities:
\begin{equation} \label{equation cst}
c(S,t):= \chi(S) + I(S) - r(S,t),
\end{equation}
where $\chi(S)$ is the Euler characteristic, and we explain the remaining two quantities in the following paragraphs.

Any generic oriented decomposing surface $S$ is such that the positive unit normal field $\nu_S$ of $S$ is nowhere parallel to $v_0$ along $\bdd S$.  Denote the components of $\bdd S$ by $T_1, \ldots, T_k$; each of the components has an orientation coming from the orientation of $S$.  Let $w_0$ denote the nowhere zero vector field obtained by projecting $v_0$ into $TS$.  Further, let $e$ be the positive unit tangent vector field of $\bdd S$.  For $1 \leq i \leq k$, define the {\it index} $I(T_i)$ to be the (signed) number of times $w_0$ rotates with respect to $e$ as we go around $T_i$.  Then set
\[
I(T_i):=\sum_{i=1}^k I(T_i).
\]

Next, let $p(\nu_S)$ be the projection of $\nu_S$ into $v_0^\perp$.  Observe that $p(\nu_S)|_{\bdd S}$ is nowhere zero.  For $1 \leq i \leq k$ define $r(T_i,t)$ to be the number of times $p(\nu_S)|_{\bdd T_i}$ rotates with respect to $r$ as we go around $T_i$.  Then set
\[
r(S,t):=\sum_{i=1}^k r(T_i,t).
\]

Continuing with the same notation, we have the following very useful lemma.  We denote by $O_S$ the subset of $\Spin^c$ structures $\fs$ for which there is a unit vector field $v$ on $M$ whose homology class is $\fs$ and $v_p \neq (-\nu_S)_p$ for every $p \in S$; for a more detailed definition see Definition \ref{def outer intro}.
  
\begin{lemma} \label{lemma the two c are the same}  \cite[Lem.\,3.10]{Jusurface}
Let $(M,\ga)$ be a strongly balanced sutured manifold.  Let $t$ be a trivialisation of $v_0^\perp$, let $\fs \in \Spin^c(M,\ga)$, and let $S$ be a decomposing surface in $(M,\ga)$ as above. Then $\fs \in O_S$ if and only if 
\[
\langle c_1(\fs,t) [T_i] \rangle =c(T_i,t) \hspace{1cm} \text{ for every } 1 \leq i \leq k.
\]
In particular, summing both sides over $1 \leq i \leq k$ we have that if $\fs \in O_S$, then 
\[
\langle c_1(\fs,t)[S] \rangle = c(S,t).
\]
\end{lemma}

\subsection{Basic properties} \label{subsection samples}

In this section we exhibit some basic properties of suture Floer homology, and how it relates to $\hfhat$ and $\hfkhat$.

\subsubsection{Product manifolds} \label{subsection product manifolds}
Consider the case of the {\it product sutured manifold} $(M,\ga)$: let $R$ be a compact oriented surface with no closed components, then $(M,\ga):=(R \times I, \bdd R \times I)$ with $s(\ga):=\bdd R \times \{1/2\}$.  It is easy to show that $SFH(M,\ga)=\bZ$ \cite[Prop.\,9.4]{Ju}.  In particular, take the balanced sutured diagram $(R, \emptyset, \emptyset)$ that defines $(M,\ga)$, and is certainly admissible as $H_2(M)=0$.  Clearly, $CF(R,\emptyset, \emptyset)=\bZ$, since it has a single generator consisting of a point.  A very important result is that the converse is also true: namely, if $SFH(M,\ga)=\bZ$, then $(M,\ga)$ is a product manifold; see Corollary \ref{cor product hard direction}.

\subsubsection{Product decompositions}
Further, we consider an important computational tool: the sutured Floer homology of a manifold remains unchanged under product decompositions.  In particular, if  $(M,\ga) \leadsto^D (M',\ga')$ is  product decomposition, then $SFH(M,\ga)=SFH(M',\ga')$ \cite[Lemma\,9.13]{Ju}.  To see this consider the following sutured diagram of $(M,\ga)$ particularly suitable for decomposing along $D$.  

Take a closed neighbourhood $N(D)$ of $D$ and choose a diffeomorphism $t \colon N(D) \to [-1,4]^3$ that maps $D$ to $\{3/2\} \times [-1,4]^2$ and sends $s(\ga) \cap N(D)$ to $[-1,4] \times \bdd[-1,4] \times \{3/2\}$; see Figure \ref{fig product decomp2}.  

\begin{figure}[h]
\centering
\includegraphics [scale=0.60]{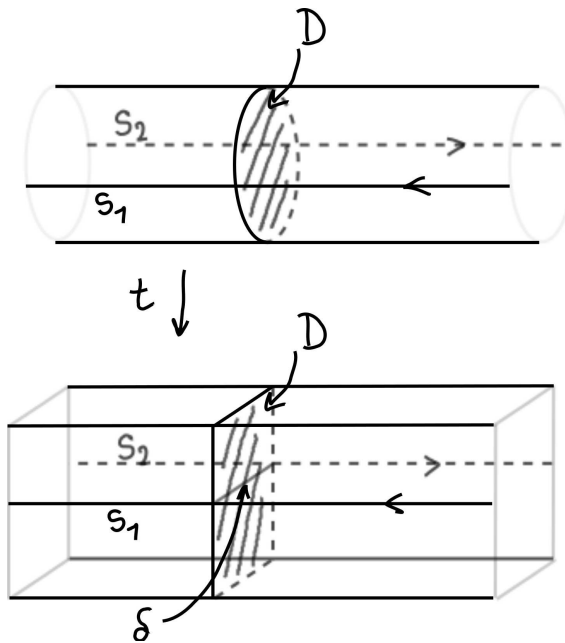}
\caption{A parametrisation $t \colon N(D) \to [-1,4]^3$ of a neighbourhood of the product disk $D$.}
\label{fig product decomp2}
\end{figure}
Denote by $p_3 \colon [-1,4]^3 \to [-1,4]$ the projection onto the third factor.  Then  we can extend $p_3 \circ t$ to a Morse function $f \colon M \to \bR$ as described in Lemma \ref{lemma self-indexing}.  Note that $f$ has no critical points in $N(D)$ and that $D$ can be seen as a union of flow lines of $-\nabla(f)$ from $R_-(\ga)$ to $R_+(\ga)$.  As before, $f$ defines a sutured diagram $(\Si,\hal,\hbe)$, where $f^{-1}(3/2)=\Si$.

Cut $\Si$ along the arc $\de=D \cap \Si$ to get a surface $\Si'$.  Since $\de$ is disjoint from $\hal \cup \hbe$, we get a new diagram $(\Si',\hal,\hbe)$ defining $(M',\ga')$.  If necessary, using finger moves (see paragraph after Remark \ref{rmk who cares about boundary}), we can isotope $\hbe$ into $\hbe'$ on $\Si'$ to make $(\Si',\hal,\hbe')$ an admissible diagram.

Note that since $\de$ is disjoint from $\hal \cup \hbe$, it is contained in a single domain that has boundary components in $\bdd \Si$.  Therefore, every domain $\mD \in D(\Si,\hal,\hbe)$, including the periodic domains, has multiplicity zero in the domain containing $\de$.  Therefore, any periodic domain $\mP \neq 0$ in $D(\Si,\hal,\hbe')$ corresponds to a periodic domain in $D(\Si',\hal,\hbe')$, so it has both positive and negative multiplicities.  Thus, both $(\Si,\hal,\hbe')$ and $(\Si',\hal,\hbe')$ are admissible, and moreover, the chain complexes $CF(\Si,\hal,\hbe')$ and $CF(\Si',\hal,\hbe')$ are isomorphic as any domain $\mD \in D(\Si,\hal,\hbe)$ has zero multiplicity in the domain containing $\de$.

\subsubsection{Closed manifolds}
Let $Y$ be a closed, connected oriented 3-manifold, and let $B_1, \ldots, B_n$ be $n$ pairwise disjoint 3-balls in the interior of $Y$.  Then we can make $M:= Y - \bigcup_{k=1}^n \Int (B_i)$ into a balanced sutured manifold by choosing the sutured to be oriented simple closed curves $s_i \subset \bdd B_i$ for all $i$.  Of course, $\ga:= \bigcup_{i=1}^n N(s_i)$.   Then $SFH(Y(1)) = \hfhat(Y)$ \cite[Prop.\,9.1]{Ju}.  Indeed, if $(\Si,\hal,\hbe, z)$ is an  admissible Heegaard diagram for $Y$, then for a small neighbourhood $U$ of $z$ in $\Si$, we let $\Si':=\Si - U$ and then $(\Si',\hal,\hbe)$ is an admissible balanced sutured diagram defining $Y(1)$.  

Moreover, in general, for any $n \in \bN$ in \cite[Prop.\,9.14]{Ju} Juh\'asz showed that
 \begin{equation}\label{eq closed mfld}
 SFH(Y(n),\ga)=\hfhat(Y) \otimes \bigotimes_{n-1} \bZ^2.
 \end{equation}

The proof is by induction on $n$.  We have just seen that Equation \eqref{eq closed mfld} holds for $n=1$; assume it is true for $n-1\geq 1$. There are three main ingredients to the proof: (1) the fact that there exists a product disk $D$ so that $Y(n-1)\#(S^1 \times S^2) \leadsto Y(n)$, (2) the fact that $\hfhat(Y\#X)=\hfhat(Y) \otimes \hfhat(X)$ for $X$ a closed 3-manifold,  and (3) the fact that $\hfhat(S^1 \times S^2) =\bZ^2$.  The latter two points are standard facts from Heegaard Floer homology (see \cite{OS properties}), and (1) we justified in Example \ref{ex s1s2}.  Putting (1)--(3) together, with $X=S^1 \times S^2$, and assuming the induction hypothesis for $n-1$, we have
 \begin{equation*}
 \begin{split}
SFH(Y(n)) &\overset{(1)}{=} SFH\left( (Y\#(S^1 \times S^2)) (n-1) \right) \\
&= \hfhat\left(Y \# (S^1\times S^2) \right) \otimes \bigotimes_{n-2}\bZ^2  \\
&\overset{(2)}{=}\hfhat(Y) \otimes \hfhat (S^1 \times S^2) \otimes \bigotimes_{n-2}\bZ^2  \\
& \overset{(3)}{=} \hfhat(Y) \otimes \bigotimes_{n-1}\bZ^2.
\end{split}
\end{equation*}
Here the second equality uses the induction hypothesis.  This completes the proof of Equation \eqref{eq closed mfld}.

Note that Equation \eqref{eq closed mfld} make sense as a generalisation of the $n=1$ case, because if we take $(M,\ga)=Y(1)$, we have 
\[
SFH(M,\ga)=SFH(Y(1))=\hfhat(Y)
\]
 as we have shown before using a sutured diagram argument.

\subsubsection{Connected sum}  Given two balanced sutured manifolds $(M,\ga)$ and $(N,\nu)$ and a closed oriented 3-manifold, the following holds \cite[Prop.\,9.15]{Ju}:
\begin{gather} \label{eq connect sum1}
SFH((M,\ga) \# (N,\nu))=SFH(M,\ga) \otimes SFH(N, \nu) \otimes \bZ^2, \\
SFH(M \# Y, \ga)=SFH(M,\ga) \otimes \hfhat(Y).\label{eq connect sum2}
\end{gather}
The key to proving Equations \eqref{eq connect sum1} and \eqref{eq connect sum2} is to note that there are product decompositions (as explained in Example \ref{ex connect sum}):
\begin{equation*}
\begin{split}
(M,\ga) \# (N,\nu) &\leadsto^D (M,\ga) \bigsqcup N(1),  \\
(M \# Y, \ga) &\leadsto^D (M,\ga) \bigsqcup Y(1).
\end{split}
\end{equation*}
Here $N(1)=(N,\nu)\#S^3(1)$.  Again, from Examples \ref{ex s1s2} and \ref{ex connect sum}  we have that
\begin{equation*}
\begin{split}
N(1) &\leadsto^D (N,\nu) \bigsqcup S^3(2), \\
S^1 \times S^2 &\leadsto^D S^3(2).
\end{split}
\end{equation*}
Therefore, taking all of these product decompositions, together with the fact that product decompositions leave the sutured Floer homology unchanged, we have
\begin{align*} 
SFH\left((M,\ga) \# (N,\nu) \right) &=SFH (M,\ga) \otimes SFH\left(N(1)\right) \\
&=SFH(M,\ga) \otimes SFH(N,\nu) \otimes SFH\left(S^3(2)\right), \\
\end{align*}
and
\begin{align*} 
SFH(M \# Y, \ga) & =SFH(M,\ga) \otimes SFH\left(Y(1)\right) \otimes \hfhat(Y) \\
& = SFH(M,\ga) \otimes \hfhat(Y).
\end{align*}
From the above it follows \cite[Cor.\,9.16]{Ju} that for $(M,\ga)$ a connected balanced sutured manifold and $n \geq 1$  we have
\[
M(n)=(M,\ga)) \#_n (S^3(1)).
\]
In particular, Equation \eqref{eq connect sum1} implies that
\begin{equation}
SFH(M(n))=SFH(M,\ga) \otimes SFH(S^3(1)) \otimes \bigotimes_n \bZ^2=SFH(M,\ga) \otimes \bZ^{2^n}.
\end{equation}
To summarise: for a closed manifold $Y$ and a sutured manifold $(M,\ga)$ we have
\begin{gather*}
SFH(Y(n))= \hfhat(Y) \otimes \bZ^{2^{n-1}}, \\
SFH(M(n))=SFH(M,\ga) \otimes \bZ^{2^n}.
\end{gather*}

\subsubsection{Knot complements} \label{subsection knot}

Lastly, we explain why sutured Floer homology can be thought of as generalising $\hfkhat$.

\begin{prop} \cite[Prop.\,9.2]{Ju}
Let $K$ be a knot in a closed connected oriented 3-manifold $Y$, then 
\[
\hfkhat (Y,K) \cong SFH(Y(K)).
\]
\end{prop}

\begin{proof}
Let $(\Si,\hal,\hbe)$ be an admissible diagram for $Y$.  Further, let $(\Si,\hal,\hbe,w,z)$ be a diagram for $K$ in $Y$.  Then remove an open neighbourhood of the points $w$ and $z$ from $\Si$ to obtain a surface $\Si'$ with boundary.  The diagram $(\Si,\hal,\hbe)$ is now a balanced sutured diagram defining $Y(K)$.  From the definitions of both homology theories it follows that the two chain complexes are isomorphic; see Section \ref{subsection knot floer homology} for the construction of $\hfkhat(Y,K)$.
\end{proof}

\subsection{Well-behaved surfaces} \label{subsection well-behaved surfaces intro}

The result of decomposition along some surfaces can be described more easily than along others.  In this  section, we summarise the different types of surfaces, groomed, well-groomed, and nice, as well as how sutured Floer homology behaves under decomposition along nice surfaces \cite{Jusurface}.  This behaviour plays a crucial role in Section \ref{subsection various} where we present more advanced results.

 Two parallel curves or arcs $\la_1$ and $\la_2$ in a surface $S$ are said to be {\it coherently oriented} if $[\la_1]=[\la_2] \in H_1(S,\bdd S)$.

\begin{deff} \cite[Def.\,0.2]{Gabai II} \label{def groomed intro}
If $(M,\ga)$ is a balanced sutured manifold, then a surface decomposition $(M,\ga) \leadsto^S (M',\ga')$ is called {\it groomed} if for each component $V$ of $R(\ga)$ one of the following is true:
\begin{enumerate}
\item $S \cap V$ is a union of parallel, coherently oriented, non-separating closed curves,
\item $S \cap V$ is a union of arcs such that for each component $\de$ of $\bdd V$ we have $\abs{\de \cap \bdd S}=\abs{\langle \de,\bdd S \rangle}$. 
\end{enumerate}
A groomed surface is called {\it well-groomed} if for each component $V$ of $R(\ga)$ it holds that $S \cap V$ is a union of parallel, coherently oriented, non-separating closed curves or arcs.
\end{deff}

In order to define a {\it nice} surface, we need the following definition. A curve $C$ is {\it boundary-coherent} in a surface $R$, if either $[C] \neq 0$ in $H_1(R;\bZ)$, or $[C]=0$ in $H_1(R;\bZ)$ and $C$ is oriented as the boundary of the component of $R-C$ that is disjoint from $\bdd R$.

\begin{deff} \cite[Def.\,3.22]{Jupolytope} \label{def nice intro}
A decomposing surface $S$ in $(M,\ga)$ is called {\it nice} if $S$ is open, $v_0$ is nowhere parallel to the normal vector field of $S$, and for each component $V$ of $R(\ga)$ the set of closed components of $S \cap V$ consists of parallel, coherently oriented, and boundary-coherent simple closed curves.
\end{deff}

\begin{rmk} \label{rmk nice}
An important observation is that any open and groomed surface can be made into a nice surface by a small perturbation which places its boundary into a generic position.
\end{rmk}

\begin{deff} \cite[Def.\,1.1]{Jusurface} \label{def outer intro}
Let $(M,\ga)$ be a balanced sutured manifold and let $(S,\bdd S) \subset (M,\bdd M)$ be a properly embedded oriented surface.  An element $\fs \in \Spin^c(M,\ga)$ is called {\it outer with respect to $S$} if there is a unit vector field $v$ on $M$ whose homology class is $\fs$ and $v_p \neq (-\nu_S)_p$ for every $p \in S$.  Here $\nu_S$ is the unit normal vector field of $S$ with respect to some Riemannian metric on $M$.  Let $O_S$ denote the set of outer $\Spin^c$ structures.
\end{deff}

Given a surface decomposition, we can understand the relationship between the sutured Floer homology groups of the starting and resulting manifold in terms of the $\Spin^c$ structures of the starting manifold.

\begin{thm} \cite[Thm.\,1.3]{Jusurface} \label{thm nice intro}
Let $(M,\ga)$ be a balanced sutured manifold and let $(M,\ga) \leadsto^S (M',\ga')$ be a sutured manifold decomposition along a nice surface $S$.  Then 
\[
SFH(M',\ga')=\bigoplus_{\fs \in O_S} SFH(M,\ga).
\]
In particular, if $O_S$ contains a single $\Spin^c$ structure $\fs$ such that $SFH(M,\ga)\neq 0$, then
\[
SFH(M',\ga')=SFH(M,\ga,\fs).
\]
\end{thm}

The proof is constructive; we give an outline, but for details see the complete proof \cite[p.\,331]{Jusurface}. A key lemma is (\cite[Lem.\,5.5]{Jusurface}).  Juh\'asz defines a balanced diagram {\it adapted} to a decomposing surface $S$ in $(M,\ga)$ as a quadruple $(\Si,\hal,\hbe,P)$, where $(\Si,\hal,\hbe)$ is a balanced diagram of $(M,\ga)$ and $P$ is a subsurface of $\Si$ satisfying a set of conditions that make it possible to easily reconstruct $S \subset M$ from $P \subset \Si$ (\cite[Def.\,4.3]{Jusurface}).  Then $O_P$ is defined to be the subset of intersection points $\hx \in \bT_\al \cap \bT_\be$ that are {\it outer} with respect to $P$, that is, $\hx \in O_P$ if and only if $\hx \cap P = \emptyset$. We are justified in called these intersection points outer, because it can be showed that $\hx \in O_P$ if and only if $\fs(\hx) \in O_S$. Cutting along $\Si$ gives a way to produce an admissible balanced diagram $(\Si',\hal',\hbe')$ for $(M',\ga')$ from the diagram $(\Si,\hal,\hbe,P)$, and in particular Juh\'asz shows that there is a bijection between $\bT_{\al'} \cap \bT_{\be'}$ and $O_P$.  This completes the proof.

We briefly discuss what happens when $[S]=0$.  The resulting manifold of such decomposition has the same sutured Floer homology as the starting manifold. 

\begin{cor} \cite[Prop.\,8.6]{Jusurface} \label{cor salami}
Let $(M,\ga)$ be a balanced sutured manifold.  Suppose that $(M,\ga) \leadsto^S (M',\ga')$ is a decomposition where $S$ is nice, $[S]=0 \in H_2(M,\bdd M)$, and $(M',\ga')$ is taut.  Then $S$ separates $(M,\ga)$ into two parts $(M_1,\ga_1)$ and $(M_2,\ga_2)$ and the following holds over any field:
\[
SFH(M,\ga) = SFH(M',\ga')=SFH(M_1,\ga_1) \otimes SFH(M_2,\ga_2).
\]
\end{cor}

For details see the proof on page 333 of the original reference \cite{Jusurface}.  Juh\'asz presents two proofs: one using the same language and techniques of adapted surface diagrams, and a shorter, quicker one using Theorem \ref{thm nice intro} and Theorem \ref{thm nonzero}.  We give the latter here, although it does use material we are yet to discuss.

As $(M',\ga')$ is taut, Theorem \ref{thm nonzero} says that $SFH(M',\ga') \neq 0$.  Then Theorem \ref{thm nice intro} says that $O_S \neq \emptyset$, so specifically, there exists $\fs_0 \in O_S$.  Fix a trivialisation $t$ of $v_0^\perp$.  By Lemma \ref{lemma the two c are the same}, $\langle c_1(\fs_0,t), [S] \rangle =c(S,t)$.  Since $[S]=0$, clearly $c(S,t)=0$ as well.  But also, for any $\fs \in \Spin^c(M,\ga)$ we have  $\langle c_1(\fs,t), [S]\rangle =0$, which implies that $\fs \in O_S$.  Thus by Theorem \ref{thm nice intro}, $SFH(M,\ga)=SFH(M',\ga')$.

\subsection{Various results} \label{subsection various}

In this section we exhibit a selection of Juh\'asz's results such as the implication of the rank of sutured Floer homology of a manifold for its topology, the relation of sutured Floer homology of Seifert surface complements and knot Floer homology, as well as the sutured Floer homology of a solid torus with toroidal knots as sutures.

\subsubsection{Rank of $SFH(M,\ga)$}  Here we show that if a balanced sutured manifold is taut, then it makes sense to talk about its sutured Floer homology, because it is non-trivial.  An important part of the proof is the following theorem of Gabai.

\begin{thm} \cite[Thm.\,4.2]{Gabai}  \label{thm gabai hierarchy}
Let $(M,\ga)$ be a connected taut sutured manifold, where $M$ contains essential tori if it is a rational homology sphere.  Then $(M,\ga)$ has a sutured manifold hierarchy
\[
(M_0,\ga_0) \leadsto^{S_1} (M_1,\ga_1) \leadsto^{S_2} \cdots \leadsto^{S_n} (M_n,\ga_n), 
\]
such that each $S_i$ is connected and well-groomed, and such that if $\bdd M_{i-1} \neq \emptyset$, then $S_i\cap \bdd M_{i-1} \neq \emptyset$.
\end{thm}

We can now prove the nontriviality result of sutured Floer homology, by applying Theorem \ref{thm nice intro} together with the existence of a sutured hierarchy along manifolds that can be perturbed to be nice. 
\begin{thm} \cite[Thm.\,1.4]{Jusurface} \label{thm nonzero}
Let $(M,\ga)$ be a balanced sutured manifold that is taut.  Then 
\[
SFH(M,\ga) \geq \bZ.
\]
\end{thm}
\begin{proof}
Every sutured manifold has a sutured hierarchy as described in Theorem \ref{thm gabai hierarchy}.  By definition of balanced sutured manifold $\bdd M \neq \emptyset$, so every well-groomed surface $S_i$ can be made into a nice surface as described in Remark \ref{rmk nice}.  Then we may apply Theorem \ref{thm nice intro}, to say that 
\[
SFH(M_i,\ga_i) \leq SFH(M_{i-1},\ga_{i-1}),
\]
where the notation `$\leq$' refers to the fact that one group is a summand of another.  Lastly, since $(M_n,\ga_n)$ is a product manifold, from Section \ref{subsection product manifolds} we know that $SFH(M_n,\ga_n)=\bZ$, so clearly $SFH(M,\ga) \geq \bZ.$
\end{proof}

\begin{rmk} \label{rmk not taut}
Note that as a contrast, if $(M,\ga)$ is an irreducible balanced sutured manifold, but is not taut, then $SFH(M,\ga)=0$ \cite[Prop.\,9.18]{Ju} (the proof is due to Yi Ni).
\end{rmk}

\begin{prop} \cite[Prop.\,7.6]{Jupolytope} \label{prop stronger}
Let $(M,\ga)$ be a taut, balanced sutured manifold such that $H_2(M)=0$ and $\rank SFH(M,\ga)<2^{k+1}$ for $k \geq 0$.  Then 
\[
d(M,\ga) \leq 2k,
\]
where $d(M,\ga)$ denotes the depth of the sutured manifold.
\end{prop}

Recall that the depth is the minimum number of surfaces that form a sutured hierarchy of $(M,\ga)$ (Definition \ref{deff hierarchy}).  

A very important property of sutured Floer homology is that it detects the product sutured manifold, and this comes as a consequence of Proposition \ref{prop stronger}.  (Initially this result was proved in  \cite[Thm.\,9.7]{Jusurface}, but due to a gap in one of the cited works of Ni, it did not quite hold.  Thereafter, Juh\'asz proved a much stronger result, Proposition \ref{prop stronger}, using completely different techniques.)

\begin{cor} \label{cor product hard direction}
An irreducible balanced sutured manifold $(M,\ga)$ is a product manifold if and only if 
\[
SFH(M,\ga)=\bZ.
\]
\end{cor}

Note that the statement would be false without the word irreducible: if  $P(1)$ is the Poincar\'e homology sphere with a 3-ball removed and a single suture along the spherical boundary, then $SFH(P(1))=\bZ$ \cite[Rmk.\,9.5]{Ju}, but $P(1)$ is not a product (and not irreducible by definition).

\subsubsection{Sutured Floer homology and Seifert surfaces} \label{subsection seifert}

If $K$ is a knot in a closed connected oriented $3$-manifold, $\al \in H_2(Y_0(K))$, and $i\in \bZ$, then define
\begin{equation} \label{eq hfkhat}
\hfkhat(Y,K,\al,i):=\bigoplus_{\{\fs \in \Spin^c(Y,K) \mid \langle c_1(\fs),\al \rangle =2i\}} \hfkhat(Y,K,\fs).
\end{equation}
Sometimes when $Y$ and $S$ are obvious, we write $\hfkhat(K,g)$.  See Section \ref{subsection knot floer homology} for details on the notation and the definition of knot Floer homology.

\begin{thm} \cite[Thm.\,1.5]{Jusurface} \label{thm surface complement}
Let $K$ be a null-homologous knot in a closed connected oriented 3-manifold $Y$ and let $S \subset Y$ be a minimal genus Seifert surface of $K$.  Then 
\[
SFH(Y(S)) \cong \hfkhat(Y,K,[S],g(S))
\]
\end{thm}

\begin{proof}
Since $\hfkhat(Y,K,\fs) = SFH(Y(K),\fs)$, when $S$ is a minimal genus Seifert surface of $K$, we may rewrite Equation \eqref{eq hfkhat} as
\[
\hfkhat(Y,K,[S],g(S)):=\bigoplus_{\{\fs \in \Spin^c(Y(K)) \mid \langle c_1(\fs,t_0),[S] \rangle =2g(S)\}} SFH(Y(K),\fs),
\]
where we take $t_0$ to be the canonical trivialisation of $v_0^\perp$ given by the meridinal vector field $\xi$.  In particular, recall that $Y(K)$ is the sutured manifold obtain from $Y$ by removing an open neighbourhood of $K$ and adding two meridinal sutures with opposite orientations.  Choose one oriented meridian and let it foliate $\bdd M$, then at every point $p \in \bdd M$, there is a unit vector field $\xi$ obtained by taking the tangent vector to the meridian at that point as an element of $T_p\bdd M$.  This vector field lies in $v_0^\perp$, and as such defines a trivialisation $t_0 \colon v_0^\perp \to \bdd M \times \bR^2$ by taking $(p, re^{i \theta}\cdot \xi) \mapsto (p, r e^{i \theta} )$ where $\bR^2$ is thought of as $\bC$. 

By Lemma \ref{lemma the two c are the same}, we have that $\fs$ is outer with respect to $S$ if and only if $\langle c_1(\fs,t), [S] \rangle =c(S,t)$, where $t$ is a fixed trivialisation of $v_0^\perp$.  Thus if we can show that $c(S,t)=2g(S)$, then 
\[
\hfkhat(Y,K,[S],g(S)):=\bigoplus_{\fs \in O_S } SFH(Y(K),\fs),
\]
where the right-hand side is precisely equal to $SFH(Y(S))$ according to Theorem \ref{thm nice intro}, and  so would complete the proof.

So we turn to proving $c(S,t_0)=2g(S)$, for the trivialisation $t_0$ as specified above.  To begin with $\chi(S)=1-2g(S)$.  Further, since $\bdd S \subset \bdd N(K)$ is a longitude of $K$, the rotation of $p(\nu_s)$ with respect to $\xi$ is zero.  It is not hard to see that $I(S)=-1$.  So $c(S,t_0)=1-2g(S)-1=-2g(S)$.  However, according to \cite{OSknots}, 
\[
\hfkhat(Y,K,[S],g(S))=\hfkhat(Y,K,-[S],-g(S)),
\]
 so decomposing along $-S$ we obtain the result.
\end{proof}

\begin{rmk} \label{rmk detecting genus} The fact that non-taut, irreducible sutured manifolds have $SFH=0$ (see Remark \ref{rmk not taut}) together with Theorem \ref{thm surface complement},  gives another proof  that knot Floer homology detects the genus of the knot. In particular, if $K$ is a genus $g$ knot in a rational homology sphere $Y$, then $\hfkhat(K,g) \neq 0$ and $\hfkhat(K,i)=0$ for all $i>g$. (The first proof was given in \cite{OS genus}.)    Further, Theorem \ref{thm surface complement}, together with the detection of product manifolds (Corollary \ref{cor product hard direction}), gives a simplified proof of the fact that knot Floer homology detects product manifolds.  This method avoids the contact topology used by Ghiggini to prove the result for genus one knots \cite{Ghi}, and simplifies some methods used by Ni to prove the general case \cite{Ni}.
\end{rmk}

\begin{thm} \cite[Thm.\,2.3]{Juseifert} \label{thm disjoint surfaces}
Let $K$ be a genus $g$ knot in $S^3$.   Let $n$ be an integer greater than zero.  If $\rank \hfkhat(K,g)<2^{n+1}$, then $K$ has at most $n$ pairwise disjoint non-isotopic genus $g$ Seifert surfaces. In particular,  if $n=1$, then $K$ has a unique Seifert surface up to equivalence.
\end{thm}

\begin{proof}
Let $R_1, \ldots, R_m$ be a $m$ pairwise disjoint non-isotopic genus $g$ Seifert surfaces.  Without loss of generality supposed that the surface $R_j$ separates $R_i$ from $R_k$ where $1\leq i<j<k \leq m$.  Recall from Remark \ref{rmk seifert complement is taut}, that the sutured manifold obtained from a minimal genus Seifert surface complement is a taut, strongly balanced sutured manifold $S^3(R_i)$ for any $i$.

Consider $S^3(R_1)$.  Note that the disjoint surface $R:=R_2\cup \ldots \cup R_m$ is nice, separates the manifold $S^3(R_1)$ into $m$ components $(M_i,\ga_i)$ labelled by the index of the Seifert surface that gives rise to $R_+(\ga_i)$.  By Corollary \ref{cor salami} it follows that 
\[
\rank SFH(S^3(R_1))=\rank SFH(M_1,\ga_1) \cdot \rank SFH(M_,\ga_2) \cdot \cdots \cdot \rank SFH(M_n,\ga_n).
\]
Then, by Corollary \ref{cor product hard direction} $(M_i,\ga_i)$ is a product if and only if $SFH(M_i,\ga_i)=\bZ$, which is true if and only if $R_i$ and $R_{i+1}$ are isotopic Seifert surfaces.  Thus $\rank SFH(M_i) \geq 2$.  So $\rank SFH(S^3(R_1)) \geq 2^{n}$.

By Theorem \ref{thm surface complement},  $SFH(S^3(R_1))=\hfkhat(K,g)$, so if $K$ has $n$ pairwise disjoint non-isotopic genus $g$ Seifert surfaces we have that  $\hfkhat(K,g)\geq 2^{m}$; the result follows by reversing this implication.

\end{proof}

\subsubsection{The SFH of a solid torus}

Juh\'asz computed the sutured Floer homology of $(M,\ga)$ when $M$ is the solid torus.  Let $T(p,q;n)$ be the balanced sutured manifold $(M,\ga)$, where $M$ is a solid torus, and the sutures are $n$ parallel $(p,q)$ torus knots.  Here $p$ denotes the number of times the curve on $\bdd M$ goes around in the longitudinal direction.  Note that $n$ has to be even.

\begin{prop} \cite[Prop.\,9.1]{Jupolytope} \label{torus prop}
Suppose that $T(p,q;n)$ is as described above, and suppose that $n=2k+2$, for some non-negative integer $k$.  Then there is an identification 
\[
\Spin^c(T(p,q;n)) \cong \bZ,
\]
such that the following holds
\[
SFH(T(p,q;n),i) \cong 
\begin{cases}
\bZ^{\binom{k}{\lfloor i/p \rfloor}}, & \textrm{if } 0 \leq i < p(k+1); \\
0 , & \textrm{otherwise.}
\end{cases}
\]
\end{prop}

Here we explain one of the  insights of the proof of Proposition \ref{torus prop}; for the complete proof see the original reference.  Let $(M_1,\ga_1):=T(1,0;n)$ and $(M_2,\ga_2):=T(p,q;m)$.  Then we can obtain the manifold $(M,\ga):=T(p,q;m+n-2)$ by simply gluing together $(M_1,\ga_1)$ and $(M_2,\ga_2)$ along two annuli $A_1 \subset \bdd M_1$ and $A_2 \subset \bdd M_2$, where $A_i$ is chosen to be a component of $R_-(\ga_i)$.   By Corollary \ref{cor salami}, we have that 
\begin{equation} \label{eq sfh of tori}
SFH\left(T(p,q;n+m-2)\right)=SFH(T(1,0;n)) \otimes SFH(T(p,q; m)).
\end{equation}
Firstly, note that $T(1,0;2)$ is just a product manifold homeomorphic to $\mbox{Annulus}\times I$, so $SFH(T(1,0;2))=\bZ$ by Corollary \ref{cor product hard direction}.  Next, notice that Equation \eqref{eq sfh of tori} implies that if we can compute $SFH(T(1,0;4))$, then we can compute $SFH(T(1,0,n))$ for any $n$ even.  Further, with this information in mind, if we can compute $SFH(T(p,q;2))$, then we also know the $SFH(T(p,q;n))$ for any $n$. Juh\'asz finds the sutured diagrams of $T(1,0;4)$ and of $T(p,q;2)$, and computes their sutured Floer homology directly; see Example 7.5 of \cite{Jupolytope} for the computation regarding $T(1,0;4)$, and the proof of Proposition \ref{torus prop} for  the computation regarding $T(p,q;2)$.  To reconstruct the $\Spin^c$ grading, apart from Equation \eqref{eq sfh of tori} we need \cite[Prop.\,3.4]{Jupolytope}, which states the existence of, and describes the properties of, an affine map $f_S \colon \Spin^c(M',\ga') \to \Spin^c(M,\ga)$ obtained from a nice surface decomposition $(M,\ga) \leadsto^S (M',\ga')$ of a strongly balanced sutured manifold $(M,\ga)$.

\subsection{The sutured polytope} \label{subsection polytope intro}

We now have all the ingredients required to define the sutured Floer polytope, which is the polytope in $H^2(M,\bdd M)$ derived from the $\Spin^c$ support of sutured Floer homology.

  Let $S(M,\ga)$ be the {\it support} of the sutured Floer homology of $(M,\ga)$:
\[
S(M,\ga):=\{ \fs \in \Spin^c(M,\ga) \mid SFH(M,\ga,\fs)\neq 0\}.
\]
Consider the map $i \colon H^2(M,\bdd M;\bZ) \to H^2(M,\bdd M;\bR)$ induced by the inclusion $\bZ \hookrightarrow \bR$.  For $t$ a trivialisation of $v_0^\perp$, define
\[
C(M,\ga,t):=\{ i \circ c_1(\fs,t) \mid \fs \in S(M,\ga)\} \subset H^2(M,\bdd M;\bR).
\]
\begin{deff}
The {\it sutured Floer polytope} $P(M,\ga,t)$ with respect to $t$ is defined to be the convex hull of $C(M,\ga,t)$.
\end{deff}

Next, we have that $c_1(\fs,t_1)-c_1(\fs,t_2)$ is an element of $H^2(M,\bdd M)$ dependent only on the trivialisations $t_1$ and $t_2$ (see Remark \ref{rmk gold}), and therefore we may write $P(M,\ga)$ to mean the polytope in $H^2(M,\bdd M;\bR)$ up to translation. 

 It is important to note that $c_1$ ``doubles the distances.''  Namely, for a fixed trivialisation $t$ and $\fs_1,\fs_2 \in \Spin^c(M,\ga)$, Lemma 3.13 of \cite{Jupolytope} says that
 \[
 c_1(\fs_1,t) - c_1(\fs_2,t)=2(\fs_1- \fs_2),
 \]  
 where $\fs_1-\fs_2$ is the unique element $h \in H^2(M,\bdd M)$ such that $\fs_1=h + \fs_2$.  Such an element exists and is unique by definition of a $H^2(M,\bdd M)$-torsor. (A {\it $G$-torsor} is a set $X$ on which $G$ acts freely and transitively.)
 
Let $t \in T(M,\ga)$.  Then an element $\al \in H_2(M,\bdd M;\bR)$ defines subsets $P_\al(M,\ga,t)$ and $C_\al(M,\ga,t)$ of $P(M,\ga,t)$ and $C(M,\ga,t)$, respectively \cite[p.17]{Jupolytope}.  Firstly, set 
\begin{equation} \label{equation calt}
c(\al,t):=\min \{ \langle c, \al \rangle \mid c \in P(M,\ga,t) \}.
\end{equation}
Then there is a subset $H_\al \subset H^2(M,\bdd M;\bR)$ given by
\[
H_\al :=\{ x \in H^2(M,\bdd M;\bR) \mid \langle x, \al \rangle=c(\al,t)\}.
\]
Lastly,
\begin{gather} 
P_\al(M,\ga,t):=H_\al \cap P(M,\ga,t), \text{ and }
C_\al(M,\ga,t):=H_\al \cap C(M,\ga,t), \\ \label{equation SFH alpha}
SFH_\al(M,\ga):=\bigoplus \{ SFH(M,\ga,\fs) \mid i(c_1(\fs,t)) \in C_\al(M,\ga,t)\}. 
\end{gather}

For an explanation of the types of well-behaved surfaces mentioned in the last part of this section see Definitions \ref{def groomed intro} and \ref{def nice intro}.

\begin{prop} \cite[Prop.\,4.12]{Jupolytope} \label{prop decomposition}
If $(M,\ga)$ is taut and strongly balanced, then $P_\al(M,\ga,t)$ is the convex hull of $C_\al(M,\ga,t)$ and it is a face of the polytope $P(M,\ga,t)$.  Furthermore, if $S$ is a nice decomposing surface that gives a taut decomposition $(M,\ga) \leadsto^S (M',\ga')$ and $[S]=\al$, then $SFH(M',\ga')=SFH_\al(M,\ga)$.  
\end{prop}

The statement of part (i) of the following Corollary is simply the application of this fact, together with Gabai's result that given any $\al \in H_2(M,\bdd M)$, there exists a groomed surface decomposition $(M,\ga) \leadsto^S (M',\ga')$ such that $(M',\ga')$ is taut and $[S]=\al$; see \cite[Lem.\,0.7]{Gabai II}.  Part (ii) is slightly more involved, and builds on the theory developed in \cite{Jupolytope}.

\begin{cor} \cite[Cor.\,4.15]{Jupolytope} \label{cor decomposition}
Let $(M,\ga)$ be a taut balanced sutured manifold, and suppose that $H_2(M)=0$.  Then the following two statements hold.
\begin{enumerate}
\item For every $\al \in H_2(M,\bdd M)$, there  exists a groomed surface decomposition $(M,\ga) \leadsto^S (M',\ga')$ such that $(M',\ga')$ is taut, $[S]=\al$, and 
\[
SFH(M',\ga') \cong SFH_\al(M,\ga).
\]
If moreover, $\al$ is well-groomed, then $S$ can be chosen to be well-groomed.
\item For every face $F$ of $P(M,\ga,t)$, there exists an $\al \in H_2(M,\bdd M)$ such that $F=P_\al(M,\ga,t)$.
\end{enumerate}
\end{cor}

\begin{ex} \label{ex asymmetric}
The sutured Floer polytope is not centrally symmetric.  For example the three-component Pretzel link $P(2,2,2)$ (see Figure \ref{fig pretzel link}) has $SFH(P(2,2,2))=\bZ^3$, and the polytope is a triangle.  In general, all three-component Pretzel links $P(2s,2t,2r)$ have asymmetric polytopes (Example 8.5 of \cite{FJR10}, and also see other examples in Section 8 of \cite{FJR10}).
\end{ex}

\begin{figure}[h]
\centering
\includegraphics [scale=0.50]{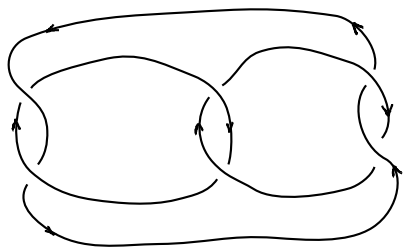} 
\caption{The sutured Floer polytope of the Pretzel link $P(2,2,2)$ is a triangle (so it is centrally asymmetric).}
\label{fig pretzel link}
\end{figure}

\begin{rmk} \label{rmk my result}
Recall that Theorem \ref{thm surface complement} said that for a null-homologous  knot $K$ in a closed oriented 3-manifold $Y$, we have  $SFH(Y(S))=\hfkhat(Y,K,[S], g(S))$, where $S$ is a minimal genus Seifert surface.  Note that if $R$ is another minimal genus Seifert surface, and if $H_2(Y)=0$ as is the case with $Y=S^3$, then $[R]=[S] \in H_2(Y(K),\bdd Y(K))$.  As a consequence $SFH(Y(S))=SFH(Y(R))$.  However, the sutured Floer polytopes $P(Y(S))$ and $P(Y(R))$ are not necessarily the same (up to translation).  See \cite{Altman1} for an infinite family of knots $K$ and pairs of Seifert surfaces $(S,R)$ for each knot that can be distinguished by their polytopes, and therefore, can also be distinguished by their sutured torsion (see Section \ref{subsection torsion}) and their $\Spin^c$ grading.
\end{rmk}

\subsection{Norms on 3-manifolds} \label{subsection norms intro} 

Here we briefly review the various (semi)norms that have been defined on the second homology group  $H_2(M,\bdd M; \bR)$ of a 3-manifold $M$ with boundary.  This survey is meant to highlight the geometric nature of the sutured Floer polytope. 

Note that we are always given a map $H_2(M,\bdd M) \to \bZ^{\geq 0}$, which is first extended to a rational-valued map on $H_2(M, \bdd M; \bQ)$ by linearity and then to a real-valued map on $H_2(M,\bdd M; \bR)$ by continuity.  Finally, in each case some work has to be done to show that the resulting map on the real-valued homology group is indeed a semi-norm.  In each case, we refer to all three maps by the same symbol, but it is obvious which one we mean. 
 
Thurston defined a semi-norm on the homology of a 3-manifold $(M,\bdd M)$ with possibly empty boundary \cite{Thurston norm}.  Given a properly embedded, oriented closed surface $S \subset M$, set 
\[
\chi_-(S):= \sum_{\textrm{components $S_i$ of } S} \max\{ 0,-\chi(S_i)\}.
\]
Then the {\it Thurston semi-norm} is given by the map
\begin{align*}
x &: H_2(M,\bdd M; \bR) \to \bZ^{\geq 0}, \\
x (\al) &:= \min \{ \chi_-(S) \mid [S]=\al \in H_2(M,\bdd M) \}.
\end{align*}
The semi-norm $x$ is a norm if there exist no subspace of $H_2(M,\bdd M)$ that is spanned by surfaces of non-negative Euler characteristic, that is, spheres, annuli and tori. The Thurston semi-norm measures the ``complexity'' of a certain homology class.  Thurston showed that some top-dimensional faces of the norm unit ball were {\it fibred}; that is, if $M$ fibres over the circle $S^1$ with fibre a closed surface $\Si \subset M$, then the ray passing through the lattice point $[\Si]$ intersects the unit ball in the interior of a top-dimensional face, and all of the rational rays through that face represent fibrations of $M$.  

Scharlemann generalised the Thurston norm \cite{Scharlemann}. 
As before let $(M,\bdd M)$ be a given 3-manifold and $S$ a properly embedded surface in $M$.  Now let $\beta$ be a properly embedded 1-complex in $M$, and define 
\[
\chi_\beta(S):=\sum_{\textrm{components $S_i$ of } S} \max \{ 0, - \chi (S_i)+ \abs{S_i \cap \beta}\}.
\]
Then the {\it generalised Thurston norm} is given by the map
\begin{align*}
x_\beta & : H_2(M,\bdd M) \to \bZ^{\geq 0}, \\
x_\beta (\al) &:= \min \{ \chi_\beta(S)\mid [S]=\al \in H_2(M,\bdd M) \}. 
\end{align*}

The generalised Thurston norm specialises to the case of sutured manifolds \cite{CC sutured Thurston norm, Scharlemann}.  In particular, suppose that $(M,\ga)$ is sutured manifold, and that $S$ is a properly embedded surface.  Then let $n(S)$ denote the absolute value of the intersection number of $\bdd S$ and $s(\ga)$ as elements of $H_1(\bdd M)$.  Define 
\[
\chi_-^s(S):=\sum_{\textrm{components $S_i$ of } S} \max \{ 0, - \chi (S_i)+\frac{1}{2}n(S_i)\}.
\]
Note that if we took $\be:=s(\ga)$, then $\chi_\be(S)=2 \chi_-^s(S) + \chi(S)$.  Similarly to before,  the {\it sutured Thurston norm}  is given by the map
\begin{align*}
x^s &: H_2(M,\bdd M) \to \bZ^{\geq 0}, \\
x^s (\al) &:= \min \{ \chi^s_-(S) \mid [S]=\al \in H_2(M,\bdd M) \}.
\end{align*}
The motivation for defining $x^s$ comes from looking at the manifold $DM$ obtained by gluing two oppositely oriented copies of $M$ along the boundary, that is,  
\[
DM:=(M,\ga) \cup (-M,-\ga) / \sim,
\] 
where the equivalence relation identifies $R_+(\ga)$ with $R_+(-\ga)$ point-wise in the obvious way.  Then $DM$ is referred to as the {\it double} of $M$.  Similarly, if $S$ is a properly embedded surface in $M$, then $DS$ is the double of $S$ in $DM$.    Now, Theorem 2.3 of \cite{CC sutured Thurston norm} says that there is a natural ``doubling map'' $D_* \colon H_2(M,\bdd M;\bR) \to H_2(DM,\bdd DM;\bR)$, so that for any $\al \in H_2(M,\bdd M;\bR)$, we have 
\[
x^s(\al)=\frac{1}{2}x(D_*(\al)).
\]

So far all of the described norms have been symmetric.  As the sutured Floer polytope is asymmetric in general, the unit balls of these norms are certainly not dual to the sutured Floer polytope. Also, in order to talk about the polytope as being dual to the unit ball of a norm, we must pick a trivialisation because otherwise the polytope is defined only up to translation in $H^2(M,\bdd M)$.  

Fix a balanced sutured manifold $(M,\ga)$ and a trivialisation $t \in T(M,\ga)$.  Using the theory developed by Juh\'asz, we define an integer-valued function on $H_2(M,\bdd M)$, dependent on $t$, that plays the role of the Thurston-type norms in the case of the sutured Floer polytope.  First of all, recall the purely geometric invariant defined in Section \ref{subsection relative structures}:
\begin{equation}
c(S,t):= \chi(S) + I(S) - r(S,t),
\end{equation}
where $\chi(S)$ is the Euler characteristic, $I(S)$ generalises the term $-\frac{1}{2}n(S_i)$ in the definition of the sutured norm, and $r(S,t)$ is an additional component, which accounts for the dependance of the polytope on the trivialisation $t$.  
  
Now, assuming that $H_2(M)=0$, for a fixed $t \in T(M,\ga)$ we define the function
\begin{align*}
y_t & : H_2(M,\bdd M) \to \bZ, \\
y_t(\al)& := \min \{ -c(S,t))\mid S \textrm{ nice decomposing surface}, [S]=\al \}.
\end{align*}
\begin{rmk} As we noted before, any open groomed surface can be slightly perturbed into a nice surface.  Any homology class $\al \neq 0$ has a groomed surface representative \cite[Lem.\,0.7]{Gabai II}, however it is not clear that a it necessarily has an open groomed representative.  Thus the condition $H_2(M)=0$.   We could have relaxed the definition and required each $S$ to satisfy all the conditions of being nice except openness, but it is not clear that this would have been helpful.
\end{rmk}

It can be shown that if $T$ is a component of $\bdd S$ such that $T \not \subset \ga$, then $I(T)=-\frac{\abs{T \cap s(\ga) }}{2}$ \cite[Lem.\,3.17]{Jupolytope}.  In other words, in this case $-I(T)=\frac{1}{2} n(T)$ which is the second term in the definition of $x^s$.

We would like to say that the function $y_t$ has some useful properties, such as that it satisfies the triangle inequality and positive homogeneity with respect to the integers. Indeed, as we see in Proposition \ref{prop properties intro}, these properties follow from the definitions and from Lemma \ref{lem calt intro}.  

\begin{lemma} \cite[Cor.\,4.11]{Jupolytope} \label{lem calt intro}
Let $(M,\ga)$ be a taut, strongly balanced sutured manifold such that $H_2(M)=0$.  Then 
\begin{equation} \label{equation calt2}
c(\al,t)=\max \{ c(S,t) \mid S \textrm{ a nice decomposing surface}, [S]=\al\}.
\end{equation} 
\end{lemma}

\begin{prop} \label{prop properties intro}
Let $(M,\ga)$ be a taut, strongly balanced sutured manifold such that $H_2(M)=0$.  Fix $t \in  T(M,\ga)$. Then for any $\al, \be \in H_2(M,\bdd M)$ and any $m \in \bZ$, the following hold
\begin{align*}
y_t(\abs{m} \cdot \al) & = \abs{m} \cdot y_t(\al ), \\
y_t(\al + \be)  & \leq y_t(\al) + y_t(\be).
\end{align*}
\end{prop}
\begin{proof}
From Lemma \ref{lem calt intro} it follows that $y_t(\al)=-c(\al,t)$ for any  $\al \in H_2(M,\bdd M)$.  By definition $c(\al,t)=\min \{ \langle c, \al \rangle \mid c \in P(M,\ga,t) \}$, thus 
\[
y_t(\al)=\max \{\langle -c, \al \rangle \mid c \in P_t\},
\]
where we have denoted $P(M,\ga,t)$ by $P_t$.

The first statement of the proposition is obvious. Proving the triangle inequality is also easy, and is identical to the proof given in  \cite[Prop.\,8.2]{Jupolytope}:
\begin{align*}
y_t(\al + \be) & =\max \{ \langle -c, \al + \be \rangle \mid c \in P_t \} \\
		& =\max \{ \langle -c, \al \rangle + \langle -c, \be \rangle \mid c \in P_t \}  \\
		& \leq \max \{ \langle -c, \al  \rangle : c \in P_t \} + \max \{ \langle -c, \be \rangle \mid c \in P_t \} \\
		& = y_t(\al) + y_t(\be).
\end{align*} 
\end{proof}

As before we can extend $y_t$ to a rational-valued map on $H_2(M,\bdd M ;\bQ)$ by linearity and then to a real-valued map on $H_2(M,\bdd M; \bR)$ by continuity.  Thus, for any balanced sutured manifold with $H_2(M)=0$ we can define a {\it geometric sutured function}
\[
y_t \colon H_2(M,\bdd M) \to \bR,
\]
such that $y_t(r \cdot \alpha)=r \cdot y_t(\al)$ and $y_r(\al + \be) \leq y_t(\al) + y_t(\be)$ for $r \in \bR$ and $\al, \be \in H_2(M,\bdd M ;\bR)$.

The following corollary says that $y_t$ is actually a (semi)norm for a lot of the often-studied sutured manifolds.

\begin{cor} \label{yt dual to polytope}
Let $(M,\ga)$ be a taut, strongly balanced sutured manifold such that $H_2(M)=0$.  If there exits a $t \in T(M,\ga)$ such that 
\[
y_t \colon H_2(M,\bdd M) \to \bR^{\geq 0}, \\
\]
then $y_t$ is an asymmetric semi-norm. In particular, this is the case when $H^2(M)=0$. Moreover, the unit ball of the semi-norm $y_t$ is the dual to the polytope $P(M,\ga,t)$.  Finally,  $y_t$ is a norm if and only if  $\dim P(M,\ga,t) = b_1(M)$.
\end{cor}

The fact that such a $t$ exits follows from Lemma 3.12 of \cite{Jupolytope}.  Basically, when there is no torsion in $H_1(M)$, then we can choose a trivialisation such that $P(M,\ga,t)$ contains $0 \in H^2(M,\bdd M)$. The very last statement in the corollary uses the same argument as in the proof of \cite[Prop.\,8.2]{Jupolytope}.

\begin{rmk} \label{rmk Ju norm}
Juh\'asz defines an asymmetric semi-norm $y$ whose dual norm unit ball is $-P(M,\ga)$, where  $-P(M,\ga)$ is the centrally symmetric image of $P(M,\ga)$  \cite[Def.\,8.1]{Jupolytope}.  Here $P(M,\ga)$ is the polytope with the centre of mass at $0 \in H^2(M,\bdd M;\bR)$.  Specifically he defines
\begin{align*}
y &: H_2(M,\bdd M ; \bR) \to \bR^{\geq 0} \\
y(\al) &:=\max \{\langle -c ,\al \rangle \mid c \in P(M,\ga) \}.
\end{align*}
When $t$ is such that the centre of mass of $P(M,\ga,t)$ lies at $0\in H^2(M,\bdd M)$, then $y_t=y$.
\end{rmk}

\subsection{The Euler characteristic} \label{subsection torsion}

Lastly, we explain how to compute the Euler characteristic of sutured Floer homology using a Fox calculus method developed in \cite{FJR10}.

Each of the groups $SFH(M,\ga,\fs)$ has a relative $\bZ_2$ grading, which can be made into an absolute grading (see Section \ref{subsection grading}).  Nevertheless, in order to speak of an absolute grading on the entire sutured Floer homology group, this must be done consistently for all $\fs \in \Spin^c(M,\ga)$. It turns out that this is possible and is equivalent to choosing a homology orientation $\om$ of $(M,R_-(\ga))$ \cite[p.\,436]{FJR10}.
 Then,  for every relative $\Spin^c$ structure $\fs$, the Euler characteristic  $\chi SFH(M,\ga,\fs)$ is well-defined with no sign ambiguity.  Theorem 1 of \cite{FJR10} tells us that the Euler characteristic with respect to the orientation $\om$, denoted by $\chi SFH(M,\ga,\fs,\om)$, is a function $T_{(M,\ga,\om)} \colon \Spin^c(M,\ga) \to \bZ$ that can be thought of as the maximal abelian torsion of the pair $(M,R_-(\ga))$, in the sense of Turaev \cite{Tu01}.
Fixing an affine isomorphism $\iota \colon \Spin^c(M,\ga) \to H_1(M)$ lets us collect all of these functions into a single generating function
\[
\tau(M,\ga):=\sum_{\fs \in \Spin^c(M,\ga)} T_{(M,\ga,\om)}(\fs) \cdot \iota(\fs).
\]
We refer to $\tau(M,\ga)$ as the {\it sutured torsion} invariant. 

In the case when $(M,\ga)$ is a manifold complementary to a Seifert surface we drop the reference to $\ga$ and write just $\tau(M)$ to mean $\tau(M,\ga)$.  Note that $\tau(M,\ga)$ is an element of the group ring $\bZ[H_1(M)]$, and that it is well-defined up to multiplication by an element of the form $\pm h$, where $h \in H_1(M)$.  We can extend the affine isomorphism $\iota$ linearly to a map on the group rings denoted by the same letter $\iota \colon \bZ[\Spin^c(M,\ga)] \to \bZ[H_1(M)]$. Then
\[
\tau(M,\ga)= \iota(\chi SFH(M,\ga)).
\]

\begin{rmk}
Notice that the abelian group $H_1(M)$ is thought of as a multiplicative group; hence the notion of being well-defined up to multiplication by an element.  Specifically, if $f= \pm h\cdot g$, for elements $f,g$ of the group ring $\bZ[H_1(M)]$, then we use the notation $f \doteq g$.
\end{rmk}

\subsection{Computation using Fox calculus}

Finally, let us describe how to compute the torsion $\tau(M,\ga)$ of a given irreducible balanced sutured manifold $(M,\ga)$ with connected subsurfaces $R_\pm(\ga)$. Fix a basepoint $p \in R_-(\ga)$.  Then Proposition 5.1 of \cite{FJR10} tells us how to compute the torsion from the map $\kappa_* \colon \pi_1(R_-(\ga),p) \to \pi_1(M,p)$ induced by the natural inclusion $\kappa \colon R_-(\ga) \hookrightarrow M$.  

First, take a {\it geometrically balanced} presentation of $\pi_1(M,p)$; that is, a presentation 
\[
\pi_1(M,p)=\langle a_1, \ldots, a_m| r_1, \ldots, r_n \rangle,
\]
where the deficiency of the presentation $m-n$ is equal to the genus $g(\bdd M)$ of the boundary of $M$.  

Obtaining a geometrically balanced presentation is not hard. Any balanced sutured manifold  $(M,\ga)$ can be reconstructed in a standard way from a balanced sutured diagram  $(\Si,\hal,\hbe)$.  Recall that to recover $(M,\ga)$, thicken $\Si$ to $\Si \times [0,1]$, regard $\hal$ as curves on $\Si \times \{0\}$, and $\hbe$ as curves on $\Si \times \{1\}$.  Then attach 2-handles along $\hal$ and $\hbe$ to obtain $M$ with sutures $\bdd \Si \times \{1/2\}$.  

Suppose that we picked the orientations so that $R_-(\ga)$ is the component of the boundary on ``the bottom'' that includes the boundaries of the 2-handles attached to $\al$.  Note that the 2-handles attached to $\al$ are precisely the 1-handles attached to $R_-(\ga)$.  Then the generators of the free group $\pi_1(R_-(\ga),p)$ and the cores of the 1-handles attached to $R_-(\ga)$ are a generating set for $\pi_1(M,p)$; the cores of the 2-handles attached to $\hbe$ give the relations of $\pi_1(M,p)$ in these generators.  Therefore, the deficiency of this presentation is equal to the number of generators of $\pi_1(R_-(\ga),p)$: say this number is $l$. Finally, as $M$ is balanced, $l$ is precisely equal to the genus of $\bdd M$. 

Let $\pi_1(R_-(\ga),p):=\langle \si_1, \ldots, \si_l \rangle$.  Then the images of $\si_j$ under the map $\kappa_*$ are words in the generators $a_i$ of $\pi_1(M,p)$.
Now we can form the square matrix of Fox derivatives
\[
\Theta_M:=
\begin{pmatrix}
\varphi \Bigl( \frac{\bdd \kappa_*(\si_j)}{\bdd a_i}\Bigr) &  \varphi \bigl( \frac{\bdd r_k} {\bdd a_i} \bigr)
\end{pmatrix},
\]
where $\varphi \colon \bZ[\pi_1(M,p)] \to \bZ[H_1(M)]$ is the map induced by the abelianisation of the fundamental group.
\begin{rmk}
We use the convention that the Fox derivative is computed left-to-right.  For example, take words $u,w \in \bZ[\pi_1(M,p)]$ and apply  the Fox derivative $\frac{\bdd }{\bdd a_i} \colon \bZ [\pi_1(M,p)] \to \bZ[\pi_1(M,p)]$ to $u w$.  Then 
\[
\frac{\bdd(u w) }{\bdd a_i}=\frac{\bdd u} {\bdd a_i} \mbox{aug}(w) + u \frac{\bdd w}{\bdd a_i},
\]
where $\mbox{aug} \colon \bZ[\pi_1(M,p)] \to \bZ$ is the augmentation map.
\end{rmk}

\begin{prop} \label{prop}  \cite[Prop.\,5.1]{FJR10} Let $(M,\ga)$ be a balanced sutured manifold such that $M$ is irreducible and the subsurfaces $R_\pm(\ga)$ are connected.    Then
\[
\tau(M,\ga) \doteq \det \Theta_{M}.
\]
\end{prop}

In particular, Proposition \ref{prop} can be applied in the case of a sutured manifold complementary to a minimal genus Seifert surface of a knot in $S^3$.

Lastly, let us say what it means for two sutured torsion polynomials $\tau_1:=\tau(M_1,\ga_1) \in \bZ[H_1(M_1)]$ and $\tau_2:=\tau(M_2,\ga_2) \in \bZ[H_1(M_2)]$ to be equivalent.  Note that the only relevant choices that we have made is that of the affine isomorphism $\iota_i \colon \Spin^c(M_i,\ga_i) \to H_1(M_i)$, for $i=1,2$.  Therefore, the two sutured torsion polynomials are {\it equivalent} $\tau_1 \sim \tau_2$ if there is an affine isomorphism $\psi \colon H_1(M_1) \to H_1(M_2)$, which extends linearly to a map on the group rings, such that $\psi(\tau_1)\doteq \tau_2$.  Also, we say that $\chi SFH(M_1,\ga_1)$ is {\it equivalent} to $\chi SFH(M_2,\ga_2)$ if $\tau_1 \sim \tau_2$.  Going back to Remark \ref{rmk my result}, we can now state the author's result precisely: \cite{Altman1} exhibits an infinite family of knots $K \subset S^3$ together with pairs of minimal genus Seifert surfaces $(S,R)$ such that the two Euler characteristics $\tau(S^3(S))$ and $\tau(S^3(R))$ are not equivalent, even though $SFH(S^3(R))=SHF(S^3(S))$.


\vspace{1cm}

{\sc \noindent Mathematics Institute, Zeeman Building, University of Warwick, UK.}
\\
E-mail address: irida.altman@gmail.com

\end{document}